\documentclass{amsart}

\usepackage{amsmath,amssymb,amsfonts,amscd}
\usepackage{latexsym}
\usepackage{epic,eepic,epsfig,psfrag}
\usepackage{amscd}
\usepackage{hyperref}
\usepackage[all]{xy}

\newcommand\ul{\underline}
\newcommand\ti{\tilde}

\def\rd{{\rm d}}
\def\rT{{\rm T}}

\def\ep{\varepsilon}
\def\e{\eta}
\def\g{\gamma}

\def\r{\rho}
\def\s{\sigma}

\def\l{\lambda}

\def\G{\Gamma}

\def\cC{{\mathcal C}}
\def\cD{{\mathcal D}}

\def\cM{{\mathcal M}}
\def\cN{{\mathcal N}}
\def\bM{{\overline{\mathcal M}}}

\def\cP{{\mathcal P}}
\def\cQ{{\mathcal Q}}

\def\cU{{\mathcal U}}
\def\cV{{\mathcal V}}

\def\R{{\mathbb R}}

\def\N{{\mathbb N}}

\def\half{{\textstyle{\frac 12}}}
\def\im{{\rm im}\,}
\def\dom{{\rm dom}\,}

\def\Gr{{\rm Gr}}
\def\ev{{\rm ev}}
\def\st{\: \big| \:}

\def\ev{{\rm ev}}

\newcommand{\leftexp}[2]{{\vphantom{#2}}^{#1}{#2}}
\newcommand{\leftsub}[2]{{\vphantom{#2}}_{#1}{#2}}

\newcommand{\tri}[6]{
\setlength{\unitlength}{1mm}
\begin{picture}(120,40)(7,3)
\put(0,40){\makebox(0,0)[b]\tiny {#1}}
\put(29,5){\makebox(0,0)[b]\tiny {#2}}
\put(100,23){\makebox(0,0)[b]\tiny {#3}}
\put(55,23){\makebox(0,0)[b]\tiny {#4}}
\put(86,31){\makebox(0,0)[b]\tiny {#5}}
\put(86,14){\makebox(0,0)[b]\tiny {#6}}
\path(70,34)(70,10)
\path(69,33)(70,34)(71,33)
\path(72,34)(98,26)
\path(72.5,32.5)(72,34)(73.5,34.8)
\path(72,10)(98,22)
\path(72.5,11.5)(72,10)(73.5,9.2)
\end{picture}
}

\newcommand{\trii}[6]{
\setlength{\unitlength}{1mm}
\begin{picture}(120,40)(7,3)
\put(0,40){\makebox(0,0)[b]\tiny {#1}}
\put(18,5){\makebox(0,0)[b]\tiny {#2}}
\put(90,23){\makebox(0,0)[b]\tiny {#3}}
\put(49,23){\makebox(0,0)[b]\tiny {#4}}
\put(80,31){\makebox(0,0)[b]\tiny {#5}}
\put(80,14){\makebox(0,0)[b]\tiny {#6}}
\path(64,34)(64,10)
\path(63,11)(64,10)(65,11)
\path(67,34)(88,26)
\path(87.5,27.5)(88,26)(87,25.2)
\path(67,10)(88,22)
\path(87,22.8)(88,22)(87.5,20.5)
\end{picture}
}

\newtheorem{dfn}{Definition}[section]
\newtheorem{lem}[dfn]{Lemma}
\newtheorem{prp}[dfn]{Proposition}
\newtheorem{thm}[dfn]{Theorem}
\newtheorem{rmk}[dfn]{Remark}
\newtheorem{cor}[dfn]{Corollary}

\begin{document}

\author{Katrin Wehrheim}

\title{Smooth structures on Morse trajectory spaces, featuring finite ends and associative gluing
}





\maketitle

\begin{abstract}
We give elementary constructions of manifold with corner structures and associative gluing maps on compactifications of spaces of infinite, half infinite, and finite Morse flow lines.
In the case of Euclidean metric in Morse coordinates near each critical point, these are naturally given by evaluations at end points and regular level sets. For finite ends this requires a blowup construction near trajectories ending at critical points.
\end{abstract}

\section{Introduction}

We begin with a summary of Morse theory in order to fix notation. For more background see e.g.\ \cite{morse, bott, milnor, witten, ab, schwarz, hutchings}.
Let $X$ be a compact manifold (without boundary).
A {\em Morse function} $f:X\to\R$ is a smooth function with nondegenerate critical points. That is, at each point of ${\rm Crit}(f)=\{p\in X \st \rd f(p)=0\}$ the Hessian ${\rm D}^2 f (p): \rT_pX\times \rT_pX\to \R$ is a nondegenerate (symmetric) bilinear form. The dimension of the negative eigenspaces of ${\rm D}^2 f (p)$ is called the Morse index $|p|\in\N_0$ of a critical point $p$.
By the Morse Lemma (e.g.\ \cite[Lemma 2.2]{milnor}) there exist coordinates $\R^n\supset B_\delta \overset{\phi}{\hookrightarrow}  X$ for a neighbourhood of each critical point $p$ that bring $f$ into the normal form
$$
(\phi^* f)(x_1,\ldots,x_n) = f(p) - \half( x_1^2 + \ldots + x_{|p|}^2 ) + \half ( x_{|p|+1}^2 + \ldots + x_n^2 ).
$$
This normal form shows that the sublevel sets of $f$ provide a decomposition of $X$ in terms of handle attachments, and hence capture the full (smooth) topology of $X$.
In order to read off the homology of $X$ from a Morse function, however, it is more useful to choose an auxiliary Riemannian metric $g$ on $X$ and study the flow lines of the gradient vector field $\nabla f\in\Gamma(\rT X)$.
More precisely, let
$$
\Psi: \R\times X \to X ,\quad (s,x)\mapsto\Psi_s(x)
$$
denote the negative gradient flow given by $\Psi_0(x)=x$ and $\frac\rd{\rd s}\Psi_s(x) = -\nabla f (\Psi_s(x))$.
Then we can consider the unstable and stable manifold for each critical point $p\in{\rm Crit}(f)$,
$$
W^-_p = \{ x \in X \st  \lim_{s\to-\infty} \Psi_s(x) = p \},
\qquad
W^+_p = \{ x \in X \st  \lim_{s\to\infty} \Psi_s(x) = p \}.
$$
These are smooth manifolds of dimension $|p|$ and $n-|p|$, respectively; see e.g.\ \cite[ch.5]{shub}.
The pair $(f,g)$ is called {\em Morse-Smale} if the unstable and stable manifolds intersect transversely.
The Morse complex of a Morse-Smale pair then reproduces the homology of $X$.
It is generated by the critical points $p\in{\rm Crit}(f)$, and the differential $\partial$ is defined by the intersection numbers of unstable and stable manifolds, i.e.\ the number of flow lines between critical points of index difference $1$. The fact that $\partial^2=0$ is proven by showing that the space of flow lines between critical points of index difference $2$ is a $1$-dimensional manifold, whose ends exactly correspond to the broken flow lines counted by~$\partial^2$.

More generally, the spaces of Morse flow lines have a natural compactification by broken flow lines, and the compactified Morse trajectory spaces (consisting of broken and unbroken flow lines)\footnote{
Throughout, all {\em unbroken flow lines} that we refer to will be unparametrized (i.e.\ parametrized negative gradient curves modulo time shift), {\em broken flow lines} are finite sequences of unbroken flow lines with matching limit critical points, and we will summarize unbroken and broken flow lines by the term {\em generalized trajectory}.
}
can be given the structure of a smooth manifold with corners, whose lower strata are given by products of Morse trajectory spaces, see Section~\ref{moduli} for more details. This is a folk theorem, possibly first stated in \cite{ab}, and with various partial proofs in the literature. A complete proof in the case of index difference $2$ is given in \cite{schwarz}, but all general treatments run into technical difficulties with the gluing of broken flow lines to nearby unbroken flow lines, which provides the charts near the boundary and corners.
These can likely be solved by subtle global estimates, but at this point we also expect a complete proof to arise from casting the Morse trajectory spaces in the general abstract framework of polyfolds, developed in \cite{hwz2} for dealing with moduli spaces of elliptic PDEs with geometric singularity formation.

The gluing difficulty is overcome by more elementary means in \cite{bh} by restricting to Morse-Smale pairs of a special normal form near each critical point for which the generalized trajectory spaces cut out smooth submanifolds with corners on the critical level sets of $f$. In that case the gluing analysis (working with implicit function theorems on Banach manifolds) can be replaced by finite dimensional intersection theory.
We will construct charts geometrically using the same normal form, and extend the results to general Morse-Smale metrics by topological conjugacy.
To make this precise we denote open balls by $B_\delta^k:=\{\ul{x}\in\R^k \,|\,|\ul{x}|<\delta\}$.

\begin{dfn} \label{ems}
A {\em Euclidean Morse-Smale pair} on a closed manifold $X$ is a pair $(f,g)$ consisting of a smooth function $f\in\cC^\infty(X,\R)$ and a Riemannian metric $g$ on $X$ satisfying a normal form and transversality condition as follows.
\begin{enumerate}
\item
For each critical point $p\in{\rm Crit}(f)$ there is a local chart $\phi_p: B_{\delta}^{n-|p|}\times B_{\delta}^{|p|}\overset{\sim}{\to} \ti U(p)$ to a neighbourhood $\ti U(p)\subset X$ of $p=\phi_p(0)$ such that
\begin{align}\label{normal}
(\phi_p^* f)(x_1,\ldots,x_n) &= f(p)  + \half ( x_{1}^2 + \ldots + x_{n-|p|}^2 ) - \half( x_{n-|p|+1}^2 + \ldots + x_{n}^2 ), \nonumber \\
(\phi_p^* g) &= \rd x_1 \otimes \rd x_1 + \ldots + \rd x_n \otimes \rd x_n .
\end{align}
\item
For every pair of critical points $p,q\in{\rm Crit}(f)$ the intersection of unstable and stable manifolds is transverse, $W^-_p \pitchfork W^+_q $.
\end{enumerate}
\end{dfn}

\begin{rmk} \label{rmk intro}
\begin{enumerate}
\item
Given any Morse function and metric, there exist $L^2$-small perturbations of the metric on annuli around the critical points that yield Morse-Smale pairs, by \cite[Prp.2]{bh}.
In particular, given a metric of normal form \eqref{normal} near the critical points, such a perturbation yields a Euclidean Morse-Smale pair.
\item
The flow $\Psi_s$ of any Morse-Smale pair is topologically conjugate to the flow $\Psi^0_s$ of a Euclidean Morse-Smale pair. That is, there exists a homeomorphism $h:X\to X$ such that $h\circ\Psi_s = \Psi^0_s\circ h$. We review the proof of this classical result in Remark~\ref{franks}.
\end{enumerate}
\end{rmk}

Assuming this normal form, the first of three goals of this paper is to give a geometrically explicit construction of a natural smooth structure on the Morse trajectory spaces in Theorem~\ref{thm corner}.
Here the specific normal form induces a natural smooth structure on a space of flow lines near each critical point. Deviating from the approach in \cite{bh} we separate this local smoothness issue from the generally smooth Morse flow on the complement of the critical points.
This setup, made precise in Section~\ref{sec:connections}, should also provide a useful framework for constructing smooth structures in infinite dimensional Floer theoretic settings. In fact, a similar setup was used in \cite{km-sw} to construct gluing maps for Seiberg-Witten Floer theory.
For general Morse-Smale pairs one does not expect a natural smooth structure since the evaluation at regular level sets has a singular image. However, any choice of topological conjugation to a Euclidean Morse-Smale flow induces a smooth structure.

Homotopy theoretic applications such as \cite{cjs} require moreover ``associative gluing maps"  near the boundary strata, introduced in detail in Section~\ref{asso}.
While it is a general fact \cite{qin} that manifolds with corners and a certain face structure of the boundary strata can be equipped with associative gluing maps, our second goal is to construct such gluing maps geometrically explicit in order to identify the gluing parameters as transition times through fixed neighbourhoods of critical points.\footnote{
While the pregluing maps that provide basic polyfold charts are evidently associative, it is unclear whether the polyfold setup can induce associative gluing maps on the Morse trajectory spaces. This is since the latter are merely cut out by a transverse section from the polyfold.
}
A precise definition and construction is given in Corollary~\ref{cor:asso} by inverting ``global charts" for the Morse trajectory spaces that are constructed in Theorem~\ref{thm:global}.

The final goal of this paper and main source of technical complications is to extend the above results to compactifications of spaces of half infinite and finite Morse flow lines. This is a natural step in the construction of associative gluing maps. More crucially, these spaces enter into the construction of several holomorphic curve moduli spaces such as trees of holomorphic disks with Morse edges (which yield finitely generated $A_\infty$-algebras associated to Lagrangian submanifolds \cite{jiayong}) and a polyfold theoretic proof of the Arnold conjecture \cite{arnold} based on moduli spaces of punctured spheres with half infinite Morse flow lines as in \cite{PSS}.
In both cases, a polyfold setup can be obtained as fiber product of SFT polyfolds (the main part of which is constructed in \cite{HWZgw}) with the compactified Morse trajectory spaces, if the latter are a priori given a manifold with corner structure with respect to which evaluations at finite ends are smooth maps.
In the first application it is important to isolate the boundary component given by zero length trajectories from all other boundary components given by broken trajectories. However, there are broken trajectories with endpoints near a critical point arbitrarily close in the Hausdorff topology to the zero length trajectory at the critical point.
To separate those boundary components we use the natural blowup construction of including the length of a trajectory in the Morse trajectory space, thus introducing a constant trajectory at the critical point for every length $L\in[0,\infty)$, converging to a broken trajectory with domains $[0,\infty), (-\infty,0]$ as $L\to\infty$.
More generally, we obtain a smooth structure for trajectories starting at or near a critical point (and potentially breaking there) by a similar blowup construction given by a natural variation in the definition of transition times near the critical point.

The following Section~\ref{results} lays out the main results of this paper, in particular the construction of ``global charts" in Theorem~\ref{thm:global with ends} for the Morse trajectory spaces with finite ends of a Euclidean Morse-Smale pair.
Section~\ref{top} establishes basic topological results for the Morse trajectory spaces and evaluation maps and deduces Theorem~\ref{thm corner} from Theorem~\ref{thm:global with ends}.
Section~\ref{near crit} prepares the proof by equipping the Morse trajectory spaces near critical points with a smooth structure and constructing various restriction maps to local trajectory spaces.
Finally, Section~\ref{global} constructs the ``global charts'' of Theorem~\ref{thm:global with ends}.

\medskip

{\it I would like to thank Alberto Abbondandolo and Jiayong Li for helpful discussions, the IAS for inspiring writing atmosphere, and the NSF for financial support.}

\section{Morse trajectory spaces, global charts, and associative gluing}
\label{results}

\subsection{Compactified Morse trajectory spaces} \label{moduli}

This section introduces the infinite, half infinite, and finite length versions of Morse trajectory spaces for a general Morse-Smale pair $(f,g)$.
For distinct critical points $p_-\neq p_+\in{\rm Crit}(f)$ the space of unbroken Morse flow lines is the space of parametrized gradient flow lines $\gamma:\R\to X$ modulo shift in the $\R$-variable,
\begin{align*}
\cM(p_-,p_+) &:= \bigl\{ \gamma:\R\to X \,\big|\, \dot{\gamma}=-\nabla f(\gamma) ,
\lim_{s\to\pm\infty}\gamma(s)=p_\pm  \bigr\} / \R \\
&\simeq \bigl(W^-_{p_-}\cap W^+_{p_+}\bigr) / \R
\;\simeq\; W^-_{p_-}\cap W^+_{p_+} \cap f^{-1}(c) .
\end{align*}
It is canonically identified with the intersection of unstable and stable manifold modulo the $\R$-action given by the flow $\Psi_s$, or their intersection with a level set for any regular value $c\in (f(p_+),f(p_-))$. In either formulation, they carry canonical smooth structures, see e.g.\ \cite[Section~2.4.1]{schwarz}.
We will consider the constant trajectories at a critical point as part of a larger trajectory space below, hence here declare $\cM(p,p):=\emptyset$.
For open subsets $U_-,U_+\subset X$ and critical points $p_-,p_+\in{\rm Crit}(f)$ the spaces of half infinite flow lines
\begin{align*}
\cM(U_-,p_+) &:= \bigl\{ \gamma: [0,\infty) \to X \,\big|\, \dot{\gamma}=-\nabla f(\gamma), \gamma(0)\in U_-, \lim_{s\to\infty}\gamma(s)=p_+  \bigr\} \simeq W^+_{p_+}\cap U_-, \\
\cM(p_-,U_+) &:= \bigl\{ \gamma: (-\infty,0] \to X \,\big|\, \dot{\gamma}=-\nabla f(\gamma), \lim_{s\to-\infty}\gamma(s)=p_-, \gamma(0)\in U_+  \bigr\} \simeq W^-_{p_-}\cap U_+
\end{align*}
inherit smooth structures directly from the unstable and stable manifold.
Finally, the space of finite unbroken flow lines
\begin{align*}
\cM(U_-,U_+) &:= \bigl\{ \gamma: [0,L] \to X \,\big|\, L\in [0,\infty), \dot{\gamma}=-\nabla f(\gamma), \gamma(0)\in U_-, \gamma(L)\in U_+  \bigr\} \\
&\simeq\; {\textstyle \bigcup_{L\in [0,\infty)}} U_- \cap \Psi_L^{-1}(U_+)
\;=\; \bigl( [0,\infty)\times  U_- \bigr) \cap \Psi^{-1}(U_+)
\end{align*}
can be identified with an open subset of $\cM(X,X)\simeq [0,\infty)\times X$ since the flow map $\Psi$ is continuous. Hence it naturally is a smooth manifold with boundary given by constant flow lines. These three types of spaces can contain constant trajectories at a critical point. Note in particular that we do not construct $\cM(X,X)$ by the images of finite Morse flow lines, $\{(x,x')\in X\times X \st x'\in \Psi_{[0,\infty)(x)} \}$ but replace the diagonal critical points $(x,x)$ with $x\in{\rm Crit}(f)$ in this image space by a half infinite interval $[0,\infty]\times\{x\}$ parametrizing the length (in time) of the trajectory.

From the smooth spaces of unbroken flow lines we obtain topological spaces of broken flow lines as follows:
To unify notation we denote by $\cU_\pm\subset X$ a set that is either open $\cU_\pm=U_\pm$ or a set consisting of a single critical point $\cU_\pm=p_\pm$.
For two such subsets $\cU_\pm\subset X$ (of same or different type) we define the set of $k$-fold broken flow lines (also called the {\em $k$-stratum}) by
$$
\bM(\cU_-,\cU_+)_k :=
\bigcup_{(p_1\ldots p_k)\in{\rm Critseq}(f,\cU_-,\cU_+)}
\cM(\cU_-,p_1)\times \cM(p_1,p_2) \ldots \times \cM(p_k,\cU_+),
$$
Here and throughout we use the notation of {\em critical point sequences between $\cU_\pm$}
\[
{\rm Critseq}(f,\cU_-,\cU_+) := \left\{ (p_1,\ldots, p_k) \left|
\begin{array}{l}
k\in\N_0, \;p_1,\ldots, p_k \in {\rm Crit}(f) , \\
 \cM(\cU_-,p_1), \cM(p_1,p_2) \ldots , \cM(p_k,\cU_+)\neq\emptyset
 \end{array}
 \right.\right\} .
\]
To simplify notation we identify $\ul{p}\in{\rm Critseq}(f,\cU_-,\cU_+)$ with the tuple $\ul{p}=(\cU_-,p_1,\ldots, p_k,\cU_+)$, and denote $p_0:=\cU_-$, $p_{k+1}:=\cU_+$.
Critical point sequences form a finite set since they have to decrease in function value.
For $k=0$ we only have the empty critical point sequence and hence
$\bM(\cU_-,\cU_+)_0 = \cM(\cU_-,\cU_+)$.
Now the  {\it Morse trajectory space} is the space of all {\it generalized trajectories},
\begin{align*}
\bM(\cU_-,\cU_+) &:= {\textstyle \bigcup_{k\in\N_0}}
\bM(\cU_-,\cU_+)_k .
\end{align*}
In the following we denote broken flow lines by $\ul{\gamma}=(\gamma_0,\gamma_1\ldots, \gamma_k)\in \bM(\cU_-,\cU_+)_k$ and also write $\ul{\gamma}=\gamma_0\in\bM(\cU_-,\cU_+)_0$ for the unbroken flow lines.
Note here that, by slight abuse of notation, we write $\gamma_i$ instead of $[\gamma_i]$ for the unparametrized flow lines in $\cM(p_i,p_{i+1})$. If $\cU_-$ resp.\ $\cU_+$ is a critical point, then $\gamma_0$ resp.\ $\gamma_k$ is an unparametrized flow line as well, otherwise it is defined on a half interval and hence parametrized.
With this notation we can define the evaluation maps at endpoints
\begin{equation}\label{eval intro}
\ev_-:\bM(X,p_+) \to X, \quad
\ev_+:\bM(p_-,X) \to X, \quad
\ev_-\times \ev_+:\bM(X,X) \to X\times X
\end{equation}
by $\ev_-(\gamma_0,\ldots,\gamma_k)=\gamma_0(0)$ for any $k\in\N_0$,
by $\ev_+(\gamma_0,\ldots,\gamma_k)=\gamma_k(0)$ for $k\geq 1$, and by
$\ev_+(\gamma_0:[0,L]\to X)=\gamma_0(L)$ for a single trajectory $k=0$.

Next, we define a metric on the Morse trajectory spaces
$$
d_\bM(\ul{\gamma},\ul{\gamma}'):= d_\text{Hausdorff}(\overline{\im\ul{\gamma}},\overline{\im\ul{\gamma}'})
+  \big|\ell(\ul{\gamma})- \ell(\ul{\gamma}') \big|
\qquad\text{for}\; \ul{\gamma},\ul{\gamma}'\in\bM(\cU_-,\cU_+) ,
$$
by the Hausdorff distance and the {\it renormalized length}
\begin{equation}\label{length}
\ell : \bM(\cU_-,\cU_+) \to [0,1], \qquad
\ul{\g} \mapsto
\begin{cases}
\tfrac{L}{1+L}  &; \ul{\g} = \bigl(\gamma:[0,L]\to X\bigr) ,\\
1 &; \text{otherwise}.
\end{cases}
\end{equation}
Here the image of a generalized trajectory $\ul{\gamma}=(\gamma_0,\ldots,\gamma_k)$ is the union of the images in $X$ of all constituting flow lines (which is independent of the parametrization),
$$
\im\ul{\gamma}:= \im\gamma_0 \cup \ldots \cup \im\gamma_k \subset X .
$$
The closure $\overline{\im\ul{\gamma}}$ contains in addition the critical points $\lim_{s\to\infty}\gamma_{j-1}=\lim_{s\to-\infty}\gamma_{j}$ for $j=1\ldots k$ as well as $\lim_{s\to-\infty}\gamma_0$ resp.\ $\lim_{s\to\infty}\gamma_{k}$ in case $\cU_-$ resp.\ $\cU_+$ is a single critical point, and hence $\overline{\im\ul{\gamma}}$ is a compact subset of $X$.
We use closures since the Hausdorff distance
$$
d_{\text{Hausdorff}}(V,W) = \max\bigl\{ \sup_{v\in V} \inf_{w\in W} d(v,w) , \sup_{w\in W} \inf_{v\in V}d_X(w,v)  \bigr\},
$$
is a metric on the set of non-empty compact subsets of $X$.

\begin{rmk} \label{definite}
\begin{enumerate}
\item
The length term in $d_\bM$ vanishes on $\bM(\cU_-,\cU_+)$ if at least one of the sets $\cU_\pm$ is a critical point (and hence all lengths are $1$).
\item
The length term is crucial in the case of open sets $\cU_\pm=U_\pm$ containing a critical point $p\in{\rm Crit}(f)\cap U_+\cap U_-$ in their intersection. In that case it provides the topological blowup construction at the trajectories whose image is a critical point.
More precisely, it separates trajectories in $\cM(U_-,U_+)\cup\cM(U_-,p)\times\cM(p,U_+)$ that are constant $\gamma\equiv p$ resp.\ $(\gamma_0\equiv p, \gamma_1\equiv p)$ but of different lengths.
\item
The Hausdorff distance is definite on $\bM(\cU_-,\cU_+)$ except for pairs of trajectories as in (ii) whose image is a critical point.
This is since the critical points in $\overline{\im\ul{\gamma}}$ are uniquely determined by the flow lines, and flow lines are in one-to-one correspondence with their images except for constant trajectories (where the length cannot be read off from the image).
Together with (ii) this shows that $d_\bM$ defines a metric.
\item
The identifications of the spaces of unbroken flow lines as above ($\cM(p_-,p_+)\simeq W^-_{p_-}\cap W^+_{p_+} \cap f^{-1}(c)$, $\cM(X,p_+)\simeq W^+_{p_+}$, $\cM(p_-,X) \simeq W^-_{p_-}$, and $\cM(X,X)\simeq [0,\infty)\times X$) are homeomorphisms with respect to the metric $d_{\bM}$.
This follows from the continuity of the evaluations maps as in Lemma~\ref{lem:cont eval} in one direction, and for the inverse from the continuity of the Morse flow together with the limit conditions.
\end{enumerate}
\end{rmk}

The renormalized length \eqref{length} is continous by definition, and we will establish continuity of the evaluation maps in Lemma~\ref{lem:cont eval}. With that, the Morse trajectory spaces for open sets $U_\pm\subset X$ are open subsets $\bM(U_-,p_+)= \ev_-^{-1}(U_-)$, $\bM(p_-,U_+)= \ev_+^{-1}(U_+)$, $\bM(U_-,U_+)= \ev_-^{-1}(U_-)\cap \ev_+^{-1}(U_+)$ of the Morse trajectory spaces for $U_\pm=X$.
So from now on we can restrict our discussion to the Morse trajectory spaces $\bM(\cU_-,\cU_+)$ for $\cU_\pm= X$ or $\cU_\pm=p_\pm\in{\rm Crit}(f)$.
In each case we will prove the following folk theorem. For reference, we recall the definition of a manifold with corners and its strata.

\begin{dfn}
A {\em smooth manifold with corners of dimension $n\in\N_0$} is a second countable Hausdorff space $M$ together with a maximal atlas of charts $\phi_\iota:M\supset U_\iota \to V_\iota\subset [0,\infty)^n$ (i.e.\ homeomorphisms between open sets such that $\cup_\iota U_\iota = M$) whose transition maps are smooth.
For $k=0,\ldots,n$ the $k$-th stratum $M_k$ is the set of all $x\in M$ such that for some (and hence every) chart the point $\phi_\iota(x)\in [0,\infty)^n$ has $k$ components equal to $0$.
\end{dfn}

\begin{thm} \label{thm corner}
Let $(f,g)$ be a Morse-Smale pair and let $\cU_-,\cU_+$ denote $X$ or a critical point ${\rm Crit}(f)$.
Then $(\bM(\cU_-,\cU_+),d_\bM)$ is a compact, separable metric space and can be equipped with the structure of a smooth manifold with corners. Its $k$-stratum is $\bM(\cU_-,\cU_+)_k$, with one additional $1$-stratum $\{0\}\times X$ given by the length $0$ trajectories in case  $\cU_-=\cU_+=X$.
\end{thm}

\begin{rmk}
In the case of a Euclidean Morse-Smale pair the smooth structure on each $\cM(\cU_-,\cU_+)$ will be naturally given by the flow time and evaluation maps at ends \eqref{eval intro} and regular level sets, as detailed in Section~\ref{sec:connections}.
As a consequence, the evaluation maps \eqref{eval intro} and evaluations at regular level sets are smooth maps $\bM(\cU_-,\cU_+)\to X$, as will be shown in Remarks~\ref{rmk:smooth eval}, \ref{rmk:smooth eval 2}.
\end{rmk}

This theorem will be deduced from much stronger constructions of global charts for Euclidean Morse-Smale pairs in the following section. The proof is given at the end of Section~\ref{top}, based on Theorem~\ref{thm:global with ends} and topological conjugacy for general Morse-Smale pairs.

\subsection{Global charts}

Assuming $(f,g)$ to be a Euclidean Morse-Smale pair from now on, we will go beyond Theorem~\ref{thm corner} to construct ``global charts" on ``large open subsets'' of the Morse trajectory spaces.
To state these results we fix a Euclidean normal neighbourhood $\ti U(p)\subset X$ as in Definition~\ref{ems} for each critical point $p\in{\rm Crit}(f)$, a family of neighbourhoods $\ti U_t(p)\subset U(p)$ for $t\in(0,1]$, and a further precompact neighbourhood $U(p)\subset\ti U(p)$.

\begin{rmk}
The highly specific choices
$$
\ti U_t(p)=\phi_p\left\{\left.
(\ul x, \ul y) \in  {\textstyle B^{|p|}_{\frac 12(1+t)\delta}\times B^{n-|p|}_{\frac 12(1+t)\delta}} \,\right|\, |\ul x| |\ul y| <t\delta \right\}, \qquad U(p)=\phi_p( B^{|p|}_{\delta/2}\times B^{n-|p|}_{\delta/2})
$$
are quite important and will be refined in Section~\ref{near crit} such that the neighbourhoods are disjoint for different critical points. Note moreover that we have precompact nesting $\ti U_t(p)\sqsubset \ti U_{t'}(p)$ for $t<t'$, where we write $\sqsubset$ for an inclusion whose closure is compact.
For $t\to 0$ the sets $\ti U_t(p)$ converge in the Hausdorff distance to the union of unstable and stable manifold in $U(p)$.
Moreover, Morse trajectories which intersects $\ti U_1(p)$ traverse the critical level set $ f^{-1}(f(p))$ within $U(p)$ or have an end within $U(p)$.
\end{rmk}

Now for any $t\in(0,1]$ and critical point sequence $\ul{q}=(q_1,\ldots,q_k)\in{\rm Critseq}(f,\cU_-,\cU_+)$ we define the {\em large open subset}
$$
\cV_t(\ul{q}) = \cV_t(\cU_-,q_1,\ldots,q_k,\cU_+) \;\subset\; \bM(\cU_-,\cU_+)
$$
as the subset of those generalized Morse trajectories that intersect the neighbourhoods $\ti U_t(q_i)\subset X$ of each of the critical points $q_1,\ldots, q_k$ and do not intersect any other critical points (other than $\cU_\pm=p_\pm$ in case this denotes a critical point).
A more formal definition of the large open sets $\cV_t(\ul{q})$ will be given in Section~\ref{global}, where we will also choose the $\ti U_t(p)$ sufficiently small to guarantee that $\cV_t(\ul{q})\neq\emptyset$ iff $\ul{q}\in{\rm Critseq}(f;\cU_-,\cU_+)$.
Next, we denote the intersection of the large open subsets with the strata of $\bM(\cU_-,\cU_+)$ by
$$
\cV_t(\ul{q})_m := \cV_t(\ul{q})\cap\bM(\cU_-,\cU_+)_m .
$$
The large open subset associated to the empty critical point sequence is the space of unbroken trajectories $\cV_t(\cU_-,\cU_+)= \cM(\cU_-,\cU_+)$. For general $\ul{q}$ we know that
$\cV_t(\ul{q})_0 = \cV_t(\ul{q})\cap\cM(\cU_-,\cU_+)$
is the intersection with the space of unbroken trajectories (hence carries a natural smooth structure).
Moreover, $\cV_t(\ul{q})_k = \cM(\cU_-,q_1)\times \cM(q_1,q_2) \ldots \times \cM(q_k,\cU_+)$ is the subset of maximally broken trajectories since we do not allow the trajectory to hit critical points other than $p_\pm$ and $q_1,\ldots,q_k$, and hence $\cV_t(\ul{q})_m = \emptyset$ is empty for $m>k$.

The following theorem provides global charts in the case of infinite Morse trajectories $\cU_\pm=p_\pm$, that is homeomorphisms between the large subset $\cV_t(\ul{q})$ and spaces with a fixed smooth structure (as manifold with boundary and corners).
The charts are moreover compatible in three ways: Firstly, the charts are compatible with the given smooth structure on the space of unbroken trajectories $\cM(p_-,p_+)$.
Secondly, they are given by the canonical maps on the maximally broken trajectories in $\cV_t(\ul{q})_k$.
Finally, the charts are compatible with each other in the sense that their transition maps are given by further chart maps on smaller domains. In particular, the transition maps are smooth, hence this induces an atlas for $\bM(p_-,p_+)$ as manifold with boundary and corners.
Moreover, it induces an identification of the boundary strata with products of smaller Morse trajectory spaces and the construction of associative gluing maps in Corollary~\ref{cor:asso}.

\begin{thm} \label{thm:global}
There is a uniform constant $t>0$ such that for every pair $p_\pm\in{\rm Crit}(f)$ there exist homeomorphisms (called {\em global charts})
$$
\phi(\ul{q}): \;
\cV_t(\ul{q}) \;\longrightarrow\;  \cM(p_-,q_1) \times [0,t) \times \cM(q_1,q_2)  \ldots \times [0,t) \times \cM(q_k,p_+)
$$
for every critical point sequence $(q_1,\ldots,q_k)\in {\rm Critseq}(f,p_-,p_+)$ satisfying the following:
\begin{enumerate}
\item
The restriction $\phi(\ul{q})|_{\cV_t(\ul{q})_0}$ is a diffeomorphism
$$
\cV_t(\ul{q})_0  \;\longrightarrow\;
\cM(p_-,q_1) \times (0,t)  \times \cM(q_1,q_2) \ldots \times (0,t) \times \cM(q_k,p_+).
$$
\item
The restriction $\phi(\ul{q})|_{\cV_t(\ul{q})_k}$ is the canonical bijection
\[
\begin{array}{ccc}
\cV_t(\ul{q})_k
 & \longrightarrow &
\cM(p_-,q_1) \times \{0\} \times \cM(q_1,q_2)  \ldots \times \{0\} \times \cM(q_k,p_+) \\
(\gamma_0,\gamma_1,\ldots,\gamma_k) &\longmapsto &
(\gamma_0, 0, \gamma_1, \ldots, 0 , \gamma_k) .
\end{array}
\]
In particular, the global chart for $\ul{q}=(p_-,p_+)$ (with $k=0$) is the identity $\phi(p_-,p_+)={\rm Id}$ on $\cV_t(p_-,p_+)_0=\cV_t(p_-,p_+)=\cM(p_-,p_+)$.
\item
The global charts are compatible as follows:
Let $\ul{q},\ul{Q}\in{\rm Critseq}(f,p_-,p_+)$ be such that
$\ul{Q}=(\ldots, q_i,q'_1,\ldots,q'_\ell, q_{i+1}, \ldots)$ is obtained from $\ul{q}$ by inserting another critical point sequence  $\ul{q}'=(q_i=q'_0, q'_1,\ldots,q'_\ell, q'_{\ell+1}=q_{i+1})\in{\rm Critseq}(f,q_i,q_{i+1})$.
Then we have
$$
\phi(\ul{q})( \cV_t(\ul{q}) \cap \cV_t(\ul{Q}) ) \subset
 \ldots \cM(q_{i-1},q_i)\times [0,t) \times
\cV_t(\ul{q}')_0 \times [0,t)\times \cM(q_{i+1},q_{i+2}) \ldots
$$
and
$$
\phi(\ul{Q})|_{ \cV_t(\ul{q}) \cap \cV_t(\ul{Q}) }
= \bigl({\rm Id}\times \phi(\ul{q}') \times {\rm Id} \bigr) \circ \phi(\ul{q})|_{ \cV_t(\ul{q}) \cap \cV_t(\ul{Q}) } .
$$
That is, the following triangle commutes.
$$
\tri{$ \ldots \cM(q_{i-1},q_i)\times [0,t) \times \underbrace{ \cM(q_i,q'_1)\times (0,t)\times \ldots (0,t) \times \cM(q'_\ell,q_{i+1})} \times [0,t)\times \cM(q_{i+1},q_{i+2}) \ldots  $}
{$ \ldots \cM(q_{i-1},q_i)\times [0,t) \times \cV_t(\ul{q}')_0 \times [0,t)\times \cM(q_{i+1},q_{i+2}) \ldots$}
{$ \cV_t(\ul{q}) \cap \cV_t(\ul{Q})$}
{$ {\rm Id} \times\phi(\ul{q}')\times {\rm Id}$}
{$\phi(\ul{Q})$}
{$\phi(\ul{q})$}
$$
\item
The corner parameters are given explicitly by $e^{-T_i}\in[0,t)$ associated to each $q_i\in{\rm Crit}(f)$ encoding the time $T_i$ for which the trajectory is contained in $U(q_i)$. In particular, $e^{-T_i}=0$ corresponds to the trajectory breaking at $q_i$.
\end{enumerate}
\end{thm}

For Morse trajectories with one or both ends finite we will obtain very similar charts, except that the natural construction of a global chart for $\cV_t(X,q_1,\ldots)\subset\bM(X,\cU_+)$ using the entry and exit points in $\partial U(q_1)$ does not match smoothly with the natural chart for trajectories with initial point in $U(q_1)$. The latter arises from the normal form \eqref{normal} and reflects the blowup construction at trajectories ending at $q_1$.
The analogous issue arises on $\cV_t(\ldots,q_k,X)\subset\bM(\cU_-,X)$ for trajectories ending in $U(q_k)$.
We could give a less natural smooth construction but would loose the geometric interpretation of the corner parameters. Instead, we have chosen to cover $\cV_t(X,q_1,\ldots)$ as well as $\cV_t(\ldots,q_k,X)$ by separate charts with the following domains.
Given a nonempty critical point sequence $\ul{q}=(q_1,\ldots,q_k)\in{\rm Critseq}(f,\cU_-,\cU_+)$ we cover $\cV_t(\ul{q})$ with one, two, or four open sets of the form
\begin{equation}\label{V with ends}
\cV_t(\cQ_0, q_1,\ldots,q_k, \cQ_{k+1} ) := ( \ev_-^{-1}\times \ev_+^{-1} ) (\cQ_0\times\cQ_{k+1}) \subset \cV_t(\cU_-,q_1,\ldots,q_k,\cU_+) .
\end{equation}
For infinite ends at critical points $\cU_-=p_-$ or $\cU_+=p_+$ we keep $\cQ_0:=p_-$ resp.\ $\cQ_{k+1}:=p_+$. For finite ends $\cU_-=X$ resp.\ $\cQ_{k+1}\subset X$ we introduce a choice of open subsets $\cQ_0\subset X$ resp.\ $\cU_+=X$, in each case allowing two open subsets that cover $X$
\begin{equation}\label{choice}
\cQ_0=X\setminus \overline{U(q_1)}  \;\;\text{or}\;\;  \cQ_0=\tilde U(q_1), \quad\text{resp.}\quad
\cQ_{k+1}=X\setminus \overline{U(q_k)}  \;\;\text{or}\;\; \cQ_{k+1}=\tilde U(q_k) .
\end{equation}
To simplify notation we will also write $\cV_t(\ul{q})$ for $\cV_t(\cQ_0, q_1,\ldots,q_k, \cQ_{k+1})$, viewing the choice of $\cQ_0$ and $\cQ_{k+1}$ as part of the critical point sequence $\ul{q}$.
The above observations on the strata $\cV_t(\ul{q})_m := \cV_t(\ul{q})\cap\bM(\cU_-,\cU_+)_m$ then generalize directly. In particular, $\cV_t(\ul{q})_k = \cM(\cQ_0,q_1)\times \cM(q_1,q_2) \ldots \times \cM(q_k,\cQ_{k+1})$ is the subset of maximally broken trajectories between $\cQ_0$ and $\cQ_{k+1}$.
With this notation we may state the generalization of Theorem~\ref{thm:global} to any combination of finite and infinite ends. We include some more technical details in order to be able to use this exact statement in the iterative proof.
For that purpose we use the normal coordinates to identify $\ti U(q)\simeq \ti B^+_q\times \ti B^-_q$ as product of balls in the stable and unstable manifold $\ti B^\pm_q:=W^\pm_q\cap\ti U(q)$. Then we pull back the Euclidean metric to $\ti U(q)$.

\begin{thm} \label{thm:global with ends}
There is a uniform constant $0<t\leq 1$ such that for every combination of $\cU_\pm=X$ and $\cU_\pm\in{\rm Crit}(f)$ there exist global charts for the open sets $\cV_t(\cQ_0,q_1,\ldots,q_k,\cQ_{k+1})\subset\bM(\cU_-,\cU_+)$ for every critical point sequence $(q_1,\ldots,q_k)\in {\rm Critseq}(f,\cU_-,\cU_+)$ and choice of the open subsets $\cQ_0\subset \cU_-$, $\cQ_{k+1}\subset\cU_+$ from \eqref{choice}. Each global chart is a homeomorphism $\phi(\ul{q})=\phi(\cQ_0,q_1,\ldots,q_k,\cQ_{k+1})$ of the form
$$
\cV_t(\cQ_0,q_1,\ldots q_k,\cQ_{k+1}) \;\overset{\sim}\longrightarrow\;  \cM(\cQ_0,q_1) \times [0,t) \times \cM(q_1,q_2)\times[0,t)  \ldots \times [0,t) \times \cM(q_k,\cQ_{k+1}) ,
$$
with the following blowup construction for trajectories starting near $q_1$ or ending near $q_k$.
\begin{itemize}
\item
In case $\cQ_0=\ti U(q_1)$ the factors $\cM(\cQ_0,q_1)\times [0,t)$ are replaced by
$$
\bigl\{ (\g, E) \in \cM(\ti U(q_1),q_1) \times [0,1+t) \st E|\ev_-(\g)|<t \Delta \bigr\} .
$$
\item
In case $\cQ_{k+1}=\ti U(q_k)$ the factors $[0,t)\times \cM(q_k,\cQ_{k+1})$ are replaced by
$$
\bigl\{ (E,\g) \in [0,1+t)\times \cM(q_k,\ti U(q_k)) \st E|\ev_+(\g)|<t \Delta \bigr\} .
$$
\item
In case $k=1$ and $\cQ_0=\ti U(q_1)=\cQ_2$ the image of the chart $\phi(\ti U(q), q, \ti U(q))$ is
\[
\left\{ (\g_0,E,\g_1) \in \cM(\tilde U(q), q) \times  [0,1]\times \cM(q, \tilde U(q))  \,\left|\,
\begin{aligned}
E |\ev_-(\g_0)| |\ev_+(\g_1)| <t\Delta^2 ,\\
E |\ev_-(\g_0)|, E |\ev_+(\g_1)| <(1+t)\Delta
\end{aligned}
\right.\right\} .
\]
\end{itemize}
Moreover, the global charts satisfy the following:
\begin{enumerate}
\item
The restriction $\phi(\ul{q})|_{\cV_t(\ul{q})_0}$ is a diffeomorphism
$$
\cV_t(\ul{q})_0  \;\longrightarrow\;
\cM(\cQ_0,q_1) \times (0,t)  \times \cM(q_1,q_2) \ldots \times (0,t) \times \cM(q_k,\cQ_{k+1}) .
$$
In case $\cQ_0=\ti U(q_1)$ resp.\ $\cQ_{k+1}=\ti U(q_k)$ this holds with the domains $\{ E>0\}$.
\item
The restriction $\phi(\ul{q})|_{\cV_t(\ul{q})_k}$ is the canonical bijection
\[
\begin{array}{ccc}
\cV_t(\ul{q})_k
 & \longrightarrow &
\cM(\cQ_0,q_1) \times \{0\} \times \cM(q_1,q_2)  \ldots \times \{0\} \times \cM(q_k,\cQ_{k+1}) \\
(\gamma_0,\gamma_1,\ldots,\gamma_k) &\longmapsto &
(\gamma_0, 0, \gamma_1, \ldots, 0 , \gamma_k) .
\end{array}
\]
In particular, the global chart for $\ul{q}=(\cQ_0=\cU_-,\cQ_1=\cU_+)$ (with $k=0$) is the identity $\phi(\cU_-,\cU_+)={\rm Id}$ on $\cV_t(\cU_-,\cU_+)_0=\cM(\cU_-,\cU_+)$.
\item
The global charts are compatible as follows:
\begin{itemize}
\item
Let $\ul{Q}=(\cQ_0,\ldots, q_i,q'_1,\ldots,q'_\ell, q_{i+1}, \ldots,\cQ_{k+1})$ for $0<i<k$, $\ell\geq 1$ be obtained from $\ul{q}=(\cQ_0,\ldots, q_i,q_{i+1}, \ldots,\cQ_{k+1})$ by inserting a critical point sequence $\ul{q}'=(q_i,q'_1,\ldots,q'_\ell,q_{i+1})$. Then we have
\begin{align*}
\qquad\qquad
\phi(\ul{q})( \cV_t(\ul{q}) \cap \cV_t(\ul{Q}) ) &\subset
 \ldots \cM(q_{i-1},q_i)\times [0,2) \times
\cV_t(\ul{q}')_0 \times [0,2)\times \cM(q_{i+1},q_{i+2}) \ldots  ,\\
\phi(\ul{Q})|_{ \cV_t(\ul{q}) \cap \cV_t(\ul{Q}) }
&= \bigl({\rm Id}\times \phi(\ul{q}') \times {\rm Id} \bigr) \circ \phi(\ul{q})|_{ \cV_t(\ul{q}) \cap \cV_t(\ul{Q}) }  .
\end{align*}
\item
Let $\ul{Q}=(\cQ_0',q'_1,\ldots,q'_\ell, q_1, \ldots )$ be obtained from $\ul{q}=(\cQ_0,q_1,\ldots)$ by inserting $\ell\geq 1$ critical points.\footnote{
We allow any choice of end point conditions $\cQ_0',\cQ_0$ depending on $q_1',q_1$.  Note however that the charts have nontrivial intersection only for $(\cQ_0',\cQ_0)= (q_-,q_-)$,
$(X\setminus\overline{U(q'_1)},X\setminus\overline{U(q_1)})$, or $(\ti U(q_1'),X\setminus\overline{U(q_1)})$.
}
Then with $\ul{q}'=(\cQ_0',q'_1,\ldots,q'_\ell,q_1)$ we have
\begin{align*}
\phi(\ul{q})( \cV_t(\ul{q}) \cap \cV_t(\ul{Q}) ) &\subset \cV_t(\ul{q}')_0 \times [0,1)\times \cM(q_1,q_2) \ldots, \\
\phi(\ul{Q})|_{ \cV_t(\ul{q}) \cap \cV_t(\ul{Q}) } &= \bigl(\phi(\ul{q}') \times {\rm Id} \bigr) \circ \phi(\ul{q})|_{ \cV_t(\ul{q}) \cap \cV_t(\ul{Q}) }.
\end{align*}
\item
Let $\ul{Q}=(\ldots, q_k,  q'_1,\ldots, q'_\ell, \cQ_{k+1}')$ be obtained from
$\ul{q}=(\ldots, q_k, \cQ_{k+1})$ by inserting $\ell\geq 1$ critical points.\footnote{
The intersection is nontrivial for $(\cQ_{k+1}',\cQ_{k+1})= (q_+,q_+)$,
$(X\setminus\overline{U(q'_\ell)},X\setminus\overline{U(q_k)})$, or $(\ti U(q_\ell'),X\setminus\overline{U(q_k)})$.
}
Then with $\ul{q}'=(q_k,q'_1,\ldots,q'_\ell,\cQ_{k+1}')$ we have
\begin{align*}
\phi(\ul{q})( \cV_t(\ul{q}) \cap \cV_t(\ul{Q}) ) &\subset \ldots \cM(q_{k-1},q_k)\times [0,1) \times \cV_t(\ul{q}')_0 , \\
\phi(\ul{Q})|_{ \cV_t(\ul{q}) \cap \cV_t(\ul{Q}) } &= \bigl({\rm Id}\times \phi(\ul{q}') \bigr) \circ \phi(\ul{q})|_{ \cV_t(\ul{q}) \cap \cV_t(\ul{Q}) }.
\end{align*}
\end{itemize}
\item
The corner structure, compatibility between charts with different $\cQ_0$ or $\cQ_{k+1}$, and explicit form for trajectories ending near critical points is given explicitly as follows.
\begin{itemize}
\item
For $1\leq i\leq k$ such that $\cQ_0\neq\ti U(q_1)$ in case $i=1$ and $\cQ_{k+1}\neq\ti U(q_k)$ in case $i=k$, the parameter $e^{-T_i}\in[0,t)$ associated to $q_i\in{\rm Crit}(f)$ encodes the time $T_i>-\ln t$ for which the trajectory is contained in $U(q_i)$. In the limit $T_i\to\infty$, the parameter $e^{-T_i}=0$ corresponds to the trajectory breaking at $q_i$.
\item
For $\cQ_0=\ti U(q_1)$ and ($k>1$ or $\cQ_2\neq\ti U(q_1)$), a parameter $e^{-T_1}\in[0,1)$ encodes the length of time $T_1>0$ for which the trajectory is defined and contained in $\Psi_{\R_-}(U(q_1))$, with $e^{-T_1}=0$ corresponding to the trajectory breaking at $q_1$. A parameter $e^{-T_1}\in[1,2)$ with nonpositive time $T_1\leq 0$ encodes the fact that the trajectory intersects $\Psi_{(-\infty,T]}(U(q_1))$ iff $T>-T_1$.\footnote{
This definition of transition time is the crucial part of the blowup construction near trajectories with initial point $q_1$.
The extension to negative transition times is technically useful for the proof.
}
Moreover we have
$$
\qquad\qquad\qquad  \phi(\ul{q})(\ul{\g})=\bigl({\rm pr}_{W^+_{q_1}}(\ev_-(\ul\g)), \ldots \bigr)
\quad\text{and}\quad
\phi(\ul{q})^{-1}(\g, \ldots ) \subset \ev_-^{-1}(\ev_-(\g) \times \ti B^-_{q_1}).
$$
\item
Analogously, for $\cQ_{k+1}=\ti U(q_k)$ and ($k>1$ or $\cQ_0\neq \ti U(q_1)$), we encode the time $T_k>0$ for which the trajectory is defined and contained in $\Psi_{\R_+}(U(q_k))$, with $e^{-T_k}=0$ corresponding to breaking at $q_k$, resp.\ the time $T_k\leq 0$ for which it intersects the closure of $\Psi_{[T_k,\infty)}(U(q_k))$.
Moreover we have
$$
\qquad\qquad\qquad  \phi(\ul{q})(\ul{\g})=\bigl(\ldots, {\rm pr}_{W^-_{q_k}}(\ev_+(\ul\g))\bigr)
\quad\text{and}\quad
\phi(\ul{q})^{-1}( \ldots ,\g) \subset \ev_+^{-1}(\ti B^+_{q_k}\times \ev_+(\g) ).
$$
\item
In case $k=1$ and $\cQ_0=\ti U(q_1)=\cQ_2$ we have
$$
\phi(\ul{q})(\ul{\g})=\bigl({\rm pr}_{W^+_{q_1}}(\ev_-(\ul\g)), e^{-T_1}, {\rm pr}_{W^-_{q_1}}(\ev_+(\ul\g)) \bigr),
$$
where the parameter $e^{-T_1}$ encodes the length of the time interval on which the trajectory is defined. In particular, $e^{-T_1}=0$ corresponds to the trajectory breaking at $q_1$, and $e^{-T_1}=1$ corresponds to the trajectory having length $T_1=0$.
\item
For any nontrivial critical point sequence $(q_1,\ldots,q_k)\in{\rm Critseq}(f,\cU_-,\cU_+)$ and fixed $\cQ_{k+1}\subset\cU_+$ the transition map $\phi(X\setminus\overline{U(q_1)},q_1,\ldots) \circ \phi(\ti U(q_1),q_1,\ldots )^{-1}$ is a diffeomorphism between open subsets of $\cM(\ti U(q_1)\setminus\overline{U(q_1)},q_1) \times [0,2) \times \ldots\cM(q_k,\cQ_{k+1})$ given by the identity on all but the second factor, and the family of linear reparametrizations $E\mapsto E\tfrac{|\ev_-(\g)|}{\Delta}$ for $\g\in \cM(\tilde U(q_1)\setminus\overline{U(q_1)}, q_1)$ in the normal coordinates \eqref{normal}.

For fixed $\cQ_0\subset\cU_-$ the transition map
$\phi(\ldots,q_k,X\setminus\overline{U(q_k)}) \circ \phi(\ldots,q_k,\ti U(q_k) )^{-1}$ is analogously given by $E\mapsto E\tfrac{|\ev_+(\g)|}{\Delta}$ for $\g\in \cM(q_1,\tilde U(q_1)\setminus\overline{U(q_1)})$.
\end{itemize}
\end{enumerate}
\end{thm}

\begin{rmk} \label{rmk:gen comp}
A direct consequence of concatenating the commuting triangles in Theorem~\ref{thm:global with ends}~(iii) is the following more general compatibility.
Let $\ul{q}=(\cQ_0,q_1,\ldots q_k,\cQ_{k+1})$ and $\ul{Q}=(\cQ_0',\ldots,\cQ_{k+1}')$ be two tuples of critical point sequences and end conditions such that $\ul{Q}$ is obtained from $\ul{q}$ by changing the end conditions and inserting critical point sequences
$\ul{q}^0= (\cQ_0',q^0_1,\ldots,q^0_{\ell^0},q_1)$, $\ul{q}^1= (q_1,q^1_1,\ldots,q^1_{\ell^1},q_2), \ldots,$ $\ul{q}^k= (q_k,q^k_1,\ldots,q^k_{\ell^k},\cQ_{k+1}')$ with $\ell^0+\ell^1+\ldots+\ell^k\geq 1$.
Then we have
\begin{align*}
\phi(\ul{Q})|_{ \cV_t(\ul{q}) \cap \cV_t(\ul{Q}) }
&= \bigl(\phi(\ul{q}^0)\;\times\;\;{\rm Id}\;\;\times \;\phi(\ul{q}^1) \;\; \ldots \times\;\;{\rm Id}\;\;\times \phi(\ul{q}^k) \bigr) \circ \phi(\ul{q})|_{ \cV_t(\ul{q}) \cap \cV_t(\ul{Q}) } \\
\text{on}\quad
\phi(\ul{q})( \cV_t(\ul{q})\cap \cV_t(\ul{Q}) ) &\subset
\cV_t(\ul{q}^0)_0 \times [0,t) \times \cV_t(\ul{q}^1)_0 \ldots \times [0,t) \times \cV_t(\ul{q}^k)_0 .
\end{align*}
\end{rmk}

The proof of Theorems~\ref{thm:global} and \ref{thm:global with ends} is the main content of this paper in Section~\ref{global}.

\subsection{Associative gluing maps} \label{asso}
Inversion of the compatible global charts gives rise to associative gluing maps.
Here we restrict ourselves to the case of a Euclidean Morse-Smale pair $(f,g)$ and the standard Morse trajectories relevant to \cite{cjs}.
We note the generalization to Morse trajectories with finite ends and general Morse-Smale pairs in Remarks~\ref{rmk:gen glue}, ~\ref{rmk:c0asso}.

\begin{cor} \label{cor:asso}
There is a uniform constant $t>0$ such that for every $p_-,p_+\in{\rm Crit}(f)$ and $\ul{q}=(q_1,\ldots,q_k)\in{\rm Critseq}(f,p_-,p_+)$ there exists a homeomorphism onto its image (called {\em gluing map})
$$
\rho(\ul{q}) : \; \bM(p_-,q_1) \times [0,t) \times \bM(q_1,q_2)  \ldots \times [0,t) \times \bM(q_k,p_+)
\;\longrightarrow\; \bM(p_-,p_+).
$$
that satisfy the following:
\begin{enumerate}
\item
Each $\rho(\ul{q})$ restricts to a smooth map
$$
\cM(p_-,q_1) \times (0,t) \times \cM(q_1,q_2)  \ldots \times (0,t) \times \cM(q_k,p_+) \longrightarrow \cM(p_-,p_+).
$$
\item
Each $\rho(\ul{q})$ restricts to the canonical map
\[
\begin{array}{ccc}
\cM(p_-,q_1) \times \{0\} \times \cM(q_1,q_2)  \ldots \times \{0\} \times \cM(q_k,p_+)
 & \longrightarrow & \bM(p_-,p_+) \\
(\gamma_0, 0, \gamma_1, \ldots, 0 , \gamma_k) &\longmapsto & (\gamma_0,\gamma_1,\ldots,\gamma_k) .
\end{array}
\]
\item
The gluing maps are associative in the following sense:
Let $\ul{q},\ul{Q}\in{\rm Critseq}(f,p_-,p_+)$ be such that
$\ul{Q}=(\ldots, q_j,q'_1,\ldots,q'_\ell, q_{j+1}, \ldots)$ is obtained from $\ul{q}$ by inserting another critical point sequence  $\ul{q}'=(q_j=q'_0, q'_1,\ldots,q'_\ell,q'_{\ell+1}=q_{j+1})\in{\rm Critseq}(f,q_j,q_{j+1})$.
Then we have
$$
\rho(\ul{Q})
=\rho(\ul{q}) \circ \bigl( {\rm Id} \times \rho(\ul{q}') \times {\rm Id} \bigr),
$$
that is the following triangle commutes.
$$
\trii{$\quad\, \bM(p_-,q_1) \times \ldots [0,t) \times \underbrace{\bM(q_j,q'_{1}) \times \ldots \times\bM(q'_{\ell},q_{j+1})} \times [0,t)   \ldots \times \bM(q_k,p_+) $}
{$\,\bM(p_-,q_1) \times \ldots [0,t) \times \bM(q_j,q_{j+1}) \times [0,t)  \ldots \times \bM(q_k,p_+) $}
{$ \bM(p_-,p_+)$}
{${\rm Id} \times\rho(\ul{q}')\times {\rm Id}$}
{$\rho(\ul{Q})$}
{$\rho(\ul{q})$}
$$
\end{enumerate}
\end{cor}

\begin{rmk} \label{rmk:gen asso}
A direct consequence of concatenating the commuting triangles in Corollary~\ref{cor:asso}~(iv) is the following general associativity:
For any critical point sequences $\ul{q}=(q_1,\ldots q_k)$ and $\ul{Q}=\ul{q}\cup \bigcup_{j=0}^k \ul{q}^j$ as in Remark~\ref{rmk:gen comp} we have
\begin{equation} \label{asso eq}
\rho(\ul{Q}) =\rho(\ul{q}) \circ \bigl(\rho(\ul{q}^0)\times{\rm Id}\times \rho(\ul{q}^1) \ldots \times{\rm Id}\times \rho(\ul{q}^k) \bigr) .
\end{equation}
Moreover, the canonical form in Corollary~\ref{cor:asso}~(iii) generalizes as follows:
Let $L=\{\ell_1,\ldots,\ell_m\}\subset\{1,\ldots,k\}$ be an ordered subset.
Then the restriction of $\rho(\ul{p})$ to the set where $L$ indexes the vanishing gluing parameters,
$$
\bigl\{ (\ul{\gamma}_0,\tau_1,\ul{\gamma}_1,\ldots, \tau_k, \ul{\gamma}_k ) \in \dom\rho(\ul{p}) \st
\tau_{\ell} = 0 \Leftrightarrow \ell\in L \bigr\},
$$
takes values in the subset of trajectories breaking at exactly $p_{\ell_1},\ldots, p_{\ell_m}$,
$$
\rho(\ul{p})\bigl( \bigr\{  \tau_{\ell} = 0 \Leftrightarrow \ell\in L \bigr\}\bigr)
\subset
\cM(p_-,p_{\ell_1}) \times
\cM(p_{\ell_1},p_{\ell_2}) \times \ldots \times \cM(p_{\ell_m},p_+) .
$$
This follows from the canonical form (iii) for $\rho(p_{\ell_1},\ldots, p_{\ell_m})$, expressing $\rho(\ul{p})$ in the form of \eqref{asso eq} with $\ul{q}=(p_{\ell_1},\ldots, p_{\ell_m})$, and property (ii) for the factors $\rho(\ul{q}^j)$.
\end{rmk}

\begin{rmk} \label{rmk:gen glue}
The constructions for Corollary~\ref{cor:asso} also provide further gluing maps for the compactified moduli spaces of types $\bM(X,p_+)$, $\bM(p_-,X)$, and $\bM(X,X)$, which together with the gluing maps for the spaces of type $\bM(p_-,p_+)$ satisfy the general associativity relations.
However, there are different gluing maps for the same critical points but different end conditions. These are related by a reparametrization in the first or last real valued parameter.
For $\bM(X,p_+)$ the elementary gluing maps are
\begin{align*}
&\bM(X\setminus\overline{U(q_1)},q_1) \times [0,t) \times \ldots \times [0,t) \times \bM(q_k,p_+) \;\longrightarrow\; \bM(X\setminus\overline{U(q_1)},p_+) , \\
&\bigl\{ (x, E) \in \cM(\ti U(q_1),q_1) \times [0,1+t) \st E|x|<t \Delta \bigr\} \times \ldots \times [0,t) \times \bM(q_k,p_+) \;\longrightarrow\; \bM(\ti U(q_1),p_+) .
\end{align*}
For $\bM(p_-,X)$ the elementary gluing maps are
\begin{align*}
&\bM(p_-,q_1) \times [0,t) \times \ldots \times [0,t) \times \bM(q_k,X\setminus\overline{U(q_k)}) \;\longrightarrow\; \bM(p_-,X\setminus\overline{U(q_k)}) , \\
& \bM(p_-,q_1) \times [0,t) \times \ldots \times \bigl\{ (E,y) \in [0,1+t)\times \cM(q_k,\ti U(q_k)) \st E|y|<t \Delta \bigr\} \;\longrightarrow\; \bM(p_-,\ti U(q_k)) .
\end{align*}
For $\bM(X,X)$ the elementary gluing maps are
{\fontsize{2.8mm}{3.8mm}\selectfont
\begin{align*}
& \bM(X\setminus\overline{U(q_1)},q_1) \times [0,t) \times  \ldots \times [0,t) \times \bM(q_k,X\setminus\overline{U(q_k)}) \;\longrightarrow\; \bM(X\setminus\overline{U(q_1)},X\setminus\overline{U(q_k)}) ,\\
& \bigl\{ (x,E) \st E|x|<t \Delta \bigr\} \times \bM(q_2,q_3) \times [0,t)  \ldots \times\bM(q_{k-1},q_k)\times \bigl\{ (E,y) \st E|y|<t \Delta \bigr\} \;\longrightarrow\; \bM(\ti U(q_1),\ti U(q_k)) , \\
& \cM(\tilde U(q_1), q_1) \times  [0,1]\times \cM(q_1, \tilde U(q_1))  \supset \bigl\{ E |x| |y| <t\Delta^2 , E|x|, E|y| < (1+t)\Delta \bigr\}
\;\longrightarrow\; \bM(\ti U(q_1),\ti U(q_1)).
\end{align*}}
\hspace{-2mm} Here we use the evaluations $\ev_-:\cM(\ti U(q),q)\overset{\sim}{\to}\ti U(q)$ and
$\ev_+:\cM(q,\ti U(q))\overset{\sim}{\to}\ti U(q)$ to pull back the Euclidean metric on $\ti U(q)$ to the stable and unstable manifold near $q$.
\end{rmk}

\begin{rmk} \label{rmk:c0asso}
The previous constructions can be pulled back by topological conjugacy to the Morse trajectory space for any general Morse-Smale pair.
More precisely, the homeomorphism $h:X\to X$ intertwining a Morse-Smale flow $\Psi_s$ with the flow $\Psi_s^0$ of a Euclidean Morse-Smale pair induces homeomorphisms $h^* : \bM_{\Psi}(\cU_-,\cU_+) \to \bM_{\Psi^0}(\cU_-,\cU_+)$ between the corresponding Morse trajectory spaces, see \eqref{eq:h}.
Conjugation of the gluing maps $\rho$ for $\Psi^0$ with $h^*$ in each component then yields gluing maps for the trajectory spaces $\bM_{\Psi}(\cU_-,\cU_+)$ which satisfy the standard form (ii) and associativity (iii), but may not be smooth in the interior. So the associated global charts for general Morse-Smale pairs may not be compatible with the smooth structure on the unbroken Morse trajectory spaces induced by evaluations.
\end{rmk}

\begin{proof}[Proof of Corollary~\ref{cor:asso}]
Recall that each global chart $\phi(\ul{q})$ has image
$$
\im\phi(\ul{q})
= \cM(p_-,q_1) \times [0,t) \times \cM(q_1,q_2)  \ldots \times [0,t) \times \cM(q_k,p_+) ,
$$
which covers the interior (and some boundary strata) of the domain
$$
\dom\rho(\ul{q}) =
\bM(p_-,q_1) \times [0,t) \times \bM(q_1,q_2)  \ldots \times [0,t) \times \bM(q_k,p_+) .
$$
So we can define $\rho(\ul{q})|_{\im\phi(\ul{q})} := \phi(\ul{q})^{-1}$ for all critical point sequences and pairs $p_-,p_+$ of critical points, and deduce (i) and (ii) from Theorem~\ref{thm:global}~(i) and (ii).
On the further boundary strata of their domains, the gluing maps will be determined by the associativity (iii) and canonical form~(ii).
For example, the trivial critical point sequence $\ul{q}=\emptyset\in{\rm Critseq}(f,p_-,p_+)$ yields $\im\phi(p_-,p_+)=\cM(p_-,p_+)$ and we defined $\rho(p_-,p_+)|_{\cM(p_-,p_+)} :={\rm Id}_{\cM(p_-,p_+)}$. This evidently satsifies (i) and (ii) and has a unique continuous extension to the homeomorphism $\rho(p_-,p_+) :={\rm Id}_{\bM(p_-,p_+)}$.
For general $\ul{q}\in{\rm Critseq}(f,p_-,p_+)$ we also wish to define $\rho(\ul{q})$ as continuous extension of $\phi(\ul{q})^{-1}$. For that purpose we express the domain as disjoint union
$$
\dom\rho(\ul{q}) = {\textstyle \bigsqcup_{\ul{q}^0,\ldots,\ul{q}^k} }
\cV_t(\ul{q}^0)_{\ell_0} \times [0,t) \times \cV_t(\ul{q}^1)_{\ell_1}  \ldots \times [0,t) \times \cV_t(\ul{q}^k)_{\ell_k}
$$
over all $k+1$-tuples of critical point sequences $\ul{q}^j=(q^j_1,\ldots, q^j_{\ell_j})\in{\rm Critseq}(f,q_j,q_{j+1})$; allowing trivial sequences and using the notation $q_0=p_-$, $q_{k+1}=p_+$.
Here each factor
$$
\cV_t(\ul{q}^j)_{\ell_j} =
\cM(q_j,q^j_1)\times \cM(q^j_1,q^j_2) \ldots \times \cM(q^j_{\ell_j},q_{j+1})
\;\subset\; \bM(q_j,q_{j+1})
$$
is the maximally broken stratum of $\cV_t(\ul{q}^j)\subset \bM(q_j,q_{j+1})$. For trivial tuples $\ul{q}^0=\ldots=\ul{q}^k=\emptyset$ we have $\cV_t(\ul{q}^j)= \cM(q_j,q_{j+1})$ and already defined
$
\rho(\ul{q})|_{\cV_t(\ul{q}^0)_{\ell_0} \times [0,t)   \ldots \times \cV_t(\ul{q}^k)_{\ell_k}
} =  \phi(\ul{q})^{-1}$.
If some of the $\ul{q}^j$ are nontrivial then the corresponding component of the domain $\dom\rho(\ul{q})$ has a neighbourhood $\cV_t(\ul{q}^0) \times [0,t)   \ldots \times \cV_t(\ul{q}^k)$. On its interior $\cV_t(\ul{q}^0)_0 \times [0,t)   \ldots \times \cV_t(\ul{q}^k)_0$ the compatibility of global charts for $\ul{Q}:=\ul{q}\cup\bigcup_{j=0}^k \ul{q}^j$ gives
$$
\phi(\ul{q})^{-1}
= \phi(\ul{Q})^{-1}
\circ
\bigl(\phi(\ul{q}^0)\times{\rm Id}\times \phi(\ul{q}^1) \ldots \times{\rm Id}\times \phi(\ul{q}^k) \bigr) .
$$
Here the right hand side extends continuously to the maximally broken stratum $\cV_t(\ul{q}^0)_{\ell_0} \times [0,t)   \ldots \times \cV_t(\ul{q}^k)_{\ell_k}$, so shows that $\rho(\ul{q})=\phi(\ul{q})^{-1}$ continuously extends to this component.
Taking the inverse of this identity also shows that $\rho(\ul{q})^{-1}=\phi(\ul{q})$ extends continuously.
This defines $\rho(\ul{q})$ as continuous map on its entire domain. Its inverse is continuous since we checked continuous extension of $\rho(\ul{q})^{-1}$ to each of the disjoint components, and their images do not overlap since the image of the component corresponding to tuples $\ul{q}^0,\ldots,\ul{q}^k$ consists exactly of those trajectories that break at all of these critical points and a subset of $\ul{q}$.

Finally, with this definition of the gluing maps, the above identity becomes
$$
\rho(\ul{q})|_{ \cV_t(\ul{q}^0) \times [0,t)   \ldots \times \cV_t(\ul{q}^k) } = \rho(\ul{Q})
\circ
\bigl(\rho(\ul{q}^0)\times{\rm Id}\times \rho(\ul{q}^1) \ldots \times{\rm Id}\times\rho(\ul{q}^k) \bigr)^{-1} .
$$
The domain of this identity is dense in the image of $\rho(\ul{q}^0)\times{\rm Id}\times \rho(\ul{q}^1) \ldots \times{\rm Id}\times\rho(\ul{q}^k)$, hence precomposition and continuous extension prove the general associativity \eqref{asso eq}.
\end{proof}

\section{Geometry and topology of Morse trajectory spaces}
\label{top}

This section reviews various geometric and topological constructions on Morse trajectory spaces. In particular, we introduce evaluations and reparametrizations of Morse trajectories, compare different metrics on the Morse trajectory spaces, prove the topological content of Theorem~\ref{thm corner}, and show how the manifold with corner structure is induced by the global charts in Theorem~\ref{thm:global with ends} and topological conjugacy. We fix a Morse-Smale pair $(f,g)$ and begin by introducing some restricted spaces of Morse trajectories.

\begin{dfn} \label{def intersecting}
Let $\cU_\pm\subset X$ be open sets or single critical points, and let $V_1,\ldots, V_k, W\subset X$ be open subsets. We denote the set of trajectories intersecting all $V_i$ by
$$
\bM(\cU_-,\cU_+;V_1,\ldots,V_k) := \bigl\{ \ul{\gamma}\in \bM(\cU_-,\cU_+) \st \im\ul{\gamma}\cap V_i \neq\emptyset \;\forall i=1,\ldots,k \bigr\},
$$
and we denote the set of trajectories additionally contained in $W$ by
$$
\bM(\cU_-,\cU_+;V_1,\ldots,V_k;W) := \bigl\{ \ul{\gamma}\in \bM(\cU_-,\cU_+; V_1,\ldots,V_k) \st \im\ul{\gamma}\subset W \bigr\}.
$$
\end{dfn}

The openness of these subsets follows from the definition of Hausdorff distance.

\begin{lem} \label{lem intersecting}
The subsets $\bM(\cU_-,\cU_+;V_1,\ldots,V_k;W) \subset \bM(\cU_-,\cU_+)$  are open.
\end{lem}

Next, we define the evaluation at regular level sets and other local slices to the flow. Let $H\subset X$ be a submanifold of codimension $1$ whose closure is transverse to $\nabla f$ (i.e.\ $\nabla f$ is nowhere tangent to $H$), and such that $\Psi_{\R_+}(H) \cap H = \emptyset$, where $\R_+=(0,\infty)$.
Then $\Psi_{\R_-}(H)\subset X$ and $\Psi_{\R_+}(H)\subset X$ are open sets
and we can define the evaluation map
\begin{equation}\label{eval H}
\ev_H: \bM(\cU_-,\cU_+; \Psi_{\R_-}(H) , \Psi_{\R_+}(H) ) \to H ,
\quad \ul{\gamma} \mapsto \im\ul{\gamma}\cap H .
\end{equation}
for all trajectories that intersect $H$ but don't end there.
Furthermore, we recall the evaluation maps at endpoints from \eqref{eval intro},
\begin{equation}\label{eval}
\ev_-:\bM(X,p_+) \to X, \quad
\ev_+:\bM(p_-,X) \to X, \quad
\ev_-\times \ev_+:\bM(X,X) \to X\times X
\end{equation}
given by $\ev_-(\gamma_0,\ldots,\gamma_k)=\gamma_0(0)$ for any $k\in\N_0$,
by $\ev_+(\gamma_0,\ldots,\gamma_k)=\gamma_k(0)$ for $k\geq 1$, and by
$\ev_+(\gamma_0:[0,L]\to X)=\gamma_0(L)$ for $k=0$.
We will show below that these are continuous, and hence the Morse trajectory spaces
$\bM(\cU_-,p_+)= \ev_-^{-1}(\cU_-)$,
$\bM(p_-,\cU_+)= \ev_+^{-1}(\cU_+)$,
and
$\bM(\cU_-,\cU_+)= \ev_-^{-1}(\cU_-)\cap \ev_+^{-1}(\cU_+)$
for open sets $\cU_\pm\subset X$ are open subsets of the Morse trajectory spaces for $\cU_\pm=X$.

\begin{lem} \label{lem:cont eval}
The evaluation maps \eqref{eval} and \eqref{eval H} are continuous with respect to the Hausdorff distance.
When restricted to the subsets of unbroken trajectories
$\cM(p_-,p_+)$, $\cM(X,p_+)$, $\cM(p_-,X)$, resp.\ $\cM(X,X)$,
the evaluation maps are smooth.
In fact, $\ev_H: \cM(p_-,p_+)\supset {\rm dom}(\ev_H) \to H$, $\ev_-: \cM(X,p_+) \to X$,  $\ev_+:\cM(p_-,X)\to X$, and $\ev_-\times \ev_+ :\cM(X,X)^* \to X \times X$ are embeddings, where $\cM(X,X)^*$ denotes the nonconstant trajectories.
\end{lem}
\begin{proof}
We show continuity in \eqref{eval} representatively for $\ev_+:\bM(X,X) \to X$ at a fixed $\ul{\gamma}\in\bM(X,X)$ with $\ev_+(\ul{\gamma})=:e$.
Note that we drop the length term from the metric $d_\bM$ and work with the weaker Hausdorff pseudometric $d_H\leq d_\bM$.
Consider $\ul{\gamma}_i\in\bM(X,X)$ with $d_H(\ul{\gamma}_i,\ul{\gamma})\to 0$ and
$\ev_+(\ul{\gamma}_i)=:e_i\in X$.
By assumption we have $d_X(e_i,\im\ul{\gamma})\to 0$, so there exist $g_i\in\im\ul{\gamma}$ such that $d_X(e_i,g_i)\to 0$. By uniform continuity of $f$  (on the compact $X$) that also implies $|f(e_i)-f(g_i)|\to 0$.
On the other hand, we claim that $F_i:=f(e_i)\to F:=f(e)$. Indeed, for those $i\in\N$ with $F_i<F$ we have
$$
d_X( f^{-1}(F_i) , f^{-1}[F,\infty) ) \leq d_X(e_i, \im\ul{\gamma})  \to 0
$$
since $f(\im\ul{\gamma})\subset [F,\infty)$; and similarly for those $i\in\N$ with $F_i>F$ we have
$$
d_X( f^{-1}[F_i,\infty) , f^{-1}(F) ) \leq d_X(e, \im\ul{\gamma}_i)  \to 0 .
$$
Since the level sets and superlevel sets of $f$ are compact, this implies $F_i\to F$.
Putting things together we have $g_i\in\im\ul{\gamma}$ with $f(g_i)\to f(e)$, which implies $g_i\to e$ since $f$ decreases monotone along the concatenation of flow lines in $\ul{\gamma}$. The previously established $d_X(e_i,g_i)\to 0$ now implies $e_i\to e$, which proves continuity.

The spaces of unbroken trajectories $\cM(X,p_+)$ resp.\ $\cM(p_-,X)$ inherit their smooth structure from the evaluation maps $\ev_-$ resp.\ $\ev_+$, making the restrictions $\ev_-|_{\cM(X,p_+)}$, $\ev_+|_{\cM(p_-,X)}$ embeddings by definition.
The space of unbroken trajectories $\cM(X,X)$ inherits its smooth structure from the evaluation map $\ev_-$ together with the length,
\begin{align*}
\cM(X,X) = \bigl\{ \gamma: [0,L] \to X \,\big|\, L\in [0,\infty), \dot{\gamma}=-\nabla f(\gamma) \bigr\}
\to [0,\infty)\times  X  , \quad
\gamma \mapsto (L, \gamma(0) ) .
\end{align*}
That is, this map is an embedding by definition. In particular, $\ev_-|_{\cM(X,X)}$ is smooth. The second evaluation $\ev_+|_{\cM(X,X)}$ is smooth since in the above global chart of $\cM(X,X)$ it corresponds to the smooth Morse flow $[0,\infty)\times  X \to X$, $(L,x_0)\mapsto \Psi_L(x_0)$.
The product $(\ev_-\times\ev_+)|_{\cM(X,X)}$ is the composition of the above embedding with the map
$[0,\infty)\times  X \to X\times X$, $(L,x_0)\mapsto (x_0,\Psi_L(x_0))$, which is an embedding on the complement of $[0,\infty)\times{\rm Crit}(f)$, corresponding to the constant trajectories in $\cM(X,X)$.

The proof of continuity in \eqref{eval H} is somewhat more technical.
We fix a generalized trajectory $\ul{\gamma}=(\gamma_0,\ldots,\gamma_k)$ and note that due to the transversality of $H$ and $\nabla f$, the intersection point $\im\ul{\gamma}\cap H = \ev_H(\ul{\gamma}) =:x_0$ cannot be a critical point of $f$.
Moreover, the gradient flow provides a diffeomorphism
$$
(-\delta,\delta)\times H \; \overset{\sim}{\longrightarrow} \Psi_{(-\delta,\delta)}(H) =:\cN_\delta \subset X , \qquad (s,x) \longmapsto \Psi_s(x)
$$
such that any generalized trajectory $\ul{\gamma}'\in\bM(\cU_-,\cU_+; \Psi_{\R_-}(H) , \Psi_{\R_+}(H))$ has the intersection $\im\ul{\gamma}'\cap \cN_\delta \simeq I\times\{y\}$ for $y=\ev_H(\ul{\gamma}')\in H$ and an interval $I\subset(-\delta,\delta)$ containing~$0$.
Moreover, $\cN_\delta$ will contain a neighbourhood $B_\Delta(x_0)\subset X$ of radius $\Delta>0$

Now if $\ul{\gamma}'$ has Hausdorff distance $d_H(\ul{\gamma}',\ul{\gamma})\leq\ep$, then it has to pass by $x_0$ within  distance $d_X(\im\ul{\gamma}',x_0)\leq\ep$.
Since $\im\ul{\gamma}'\setminus\cN_\delta$ is contained in the complement of the ball $B_\Delta(x_0)$ we can ensure by choosing $\ep<\Delta$ that
$d_X(\im\ul{\gamma}'\cap\cN_\delta,x_0)\leq d_{\rm Hausdorff}(\ul{\gamma}',\ul{\gamma})$.
In the following we will use the product metric $d_{\R\times H}$ on $\cN_\delta$, which on $B_\Delta(x_0)$ is equivalent to $d_X$ with a constant $C$.
Then we obtain continuity
$$
d_X(y,x_0) \leq d_{\R\times H}( \im\ul{\gamma}'\cap\cN_\delta, x_0)
\leq C d_X(\im\ul{\gamma}'\cap\cN_\delta,x_0) \leq C d_{\rm Hausdorff}(\ul{\gamma}',\ul{\gamma}) .
$$

Finally, we need to check the smoothness of the evaluation map $\ev_H$ on unbroken trajectories.
When one or both of $\cU_\pm\subset X$ are open sets, then this domain
$\cM(\cU_-,\cU_+; \Psi_{\R_-}(H) , \Psi_{\R_+}(H) )_0$
is simply an unbroken trajectory space.
Let us denote the restricted open subsets in these cases by $\cU_\pm^H:=\cU_\pm \cap \Psi_{\R_\pm}(H)$,  then the evaluation map is given as follows:
$$
\ev_H : \cM(\cU_-^H , p_+ ) \simeq \cU_-^H \cap W^+_{p_+} \to H, \quad
x \mapsto \Psi_T(x) ;
$$
$$
\ev_H : \cM(p_-, \cU_+^H ) \simeq \cU_+^H \cap W^-_{p_-} \to H, \quad
x \mapsto \Psi_T(x) ;
$$
$$
\ev_H : \cM( \cU_-^H , \cU_+^H ) \simeq \bigl( [0,\infty)\times \cU_-^H \bigr) \cap \Psi^{-1}(\cU_+^H)
\to H, \quad
(L,x) \mapsto   \Psi_T(x) ;
$$
where in each case $T\in\R$ is the solution of $\Psi(T,x)\in H$.
This is a transverse equation since $\nabla f$ is transverse to $H$, hence $T\in\R$ depends smoothly on the parameter $x$, and this proves smoothness of the evaluation map $\ev_H(L,x)=\Psi(T,x)$ in these cases (dropping $L$ in the first two cases).
In case $\cU_\pm=p_\pm$ the domain of unbroken trajectories inherits its smooth structure\footnote{
The independence of the smooth structure from the choice of $c$ is one case of this smoothness statement with $H=f^{-1}(c')$ for another choice of regular value $c'$.
}
from the identification
$$
\cM(p_-,p_+; \Psi_{\R_-}(H) , \Psi_{\R_+}(H) )_0
\simeq W^-_{p_-}\cap W^+_{p_+} \cap f^{-1}(c) \cap \Psi_\R(H)
$$
for any regular value $c\in (f(p_+),f(p_-))$.
Now the evaluation map $\ev_H(x)=\Psi(T,x)$ is smooth since it is again given by solving $\Psi(T,x)\in H$ for $T\in\R$, depending on the parameter $x$ in an open subset of
$W^-_{p_-}\cap W^+_{p_+} \cap f^{-1}(c)$. The same argument proves smoothness of the inverse and hence the embedding property.
\end{proof}

With the notion of evaluation maps in place, we can compare the Hausdorff distance to other natural distance functions on the Morse trajectory spaces.

\begin{rmk} \label{rmk:metric}
\begin{enumerate}
\item
On $\cM(p_-,p_+)$, $\cM(X,p_+)$, and $\cM(p_-,X)$ the Hausdorff distance $d_\bM$ is not equivalent to the distance on $W^-_{p_-}\cap W^+_{p_+} \cap f^{-1}(c) $ resp.\ $W^+_{p_+}$ resp.\ $W^-_{p_-}$.
(A counterexample for $\cM(S^1,p_+)$ is a Morse function with one maximum and one minimum at $p_+$. Then consider Morse trajectories starting near the maximum. These initial points can be arbitrarily close, but if they lie on different sides of the maximum then the associated Morse trajectories have large Hausdorff distance.)
However, it still induces the same topology.
(This follows from the continuity of the flow in one direction and from the continuity of the evaluation maps in the other.)
\item
On $\cM(X,X)$ the distance $d_\bM(\gamma:[0,L]\to X, \gamma':[0,L']\to X)$ is not equivalent to the distance $d_X(\gamma(0),\gamma'(0))+|L-L'|$ on $[0,\infty)\times X$, but they still generate the same topology. (This follows from the continuity of the flow and evaluation maps as well as the length conversion $L\mapsto\frac{L}{1+L}$.)
\end{enumerate}
\end{rmk}

\begin{lem} \label{lem:rep}
For $(f,g)$ Euclidean Morse-Smale, a continuous reparametrization map
$$
\bM(\cU_-,\cU_+) \to \cC^0([0,1],X) , \quad \ul{\g} \mapsto \Gamma_{\ul{\g}}
$$
is defined by parametrizing the image $\im\ul{\g}$ with a continuous map $\Gamma_{\ul{\g}}:[0,1]\to \im\ul{\g} \subset X$ given by requiring linear growth of the function value
$$
f(\Gamma_{\ul{\g}}(s)) = (1-s) \cdot f(\ev_-(\ul{\g})) + s \cdot f(\ev_+(\ul{\g})).
$$
On the complement of the trajectories of zero length, this is in fact a homeomorphism to its image since $d_{\rm Hausdorff}(\im\ul{\g},\im\ul{\g}') \leq d_{\cC^0}(\G_{\ul{\g}},\G_{\ul{\g}'})$.\footnote{The Hausdorff and $\cC^0$ metric are in fact equivalent, as can be seen from adding linear estimates in the following proof.}
\end{lem}
\begin{proof}
The reparametrization map is well defined since the image of any generalized Morse trajectory is a connected finite union of critical points and embedded submanifolds along which $f$ strictly decreases.
Continuity of the inverse of this map follows from the inequality
$d_{\rm Hausdorff}(\im\G,\im\G') \leq d_{\cC^0}(\G,\G')$ for any pair of maps $\G,\G':[0,1]\to X$.
Conversely, we claim that $\G_{\ul{\g}'}(s_0)\to \G_{\ul{\g}}(s_0)$ for any fixed $s_0\in[0,1]$ as $\ul{\g}'\to\ul{\g}$ in the Hausdorff metric. (This suffices to prove convergence of the $\cC^0$-distance due to the continuity of the paths $\G$ and the compactness of their domain.)

If $\G_{\ul{\g}}(s_0)\not\in{\rm Crit}(f)$ then we can pick a coordinate chart diffeomorphic to a product $B_1 \times (-\delta,\delta)$ near $\G_{\ul{\g}}(s_0)\simeq (0,0)$ on which the Morse function and flow are linear $f: (z,\tau) \mapsto f(\G_{\ul{\g}}(s_0)) + \tau$, $\Psi_t:(z,\tau)\mapsto(z,\tau+t)$.
The metric on $X$ is equivalent with a constant $C$ to the product metric on $B_1 \times (-\delta,\delta)$, so that for $d_{\rm Hausdorff}(\im\ul{\g},\im\ul{\g}')$ sufficiently small the trajectory $\ul\g'$ has to take image in $\{z'\}\times(-\delta,\delta)$ with $z'\to 0$ as $d_{\rm Hausdorff}(\im\ul{\g},\im\ul{\g}')\to 0$.
Due to the explicit form of the flow, $\ul{\g}'$ now has to pass any function values near $ f(\G_{\ul{\g}}(s_0))$ within this coordinate chart. In particular, since evaluation is continuous with respect to the Hausdorff distance, we can ensure that $f(\Gamma_{\ul{\g}'}(s_0)) = (1-s_0) f(\ev_-(\ul{\g}')) + s_0  f(\ev_+(\ul{\g}'))$ is sufficiently close to $f(\Gamma_{\ul{\g}}(s_0))$ to guarantee that $\Gamma_{\ul{\g}'}(s_0)\simeq (z',\tau_0')$ lies in the coordinate chart. With that we have $f(\Gamma_{\ul{\g}'}(s_0))=f(\Gamma_{\ul{\g}}(s_0))+\tau_0'$ and can deduce $\Gamma_{\ul{\g}'}(s_0)\simeq (z',\tau_0') \to (0,0)\simeq\Gamma_{\ul{\g}}(s_0)$ from the continuity of the evaluation maps in
\begin{align*}
|\tau_0'| &= \bigl| f(\Gamma_{\ul{\g}'}(s_0)) - f(\Gamma_{\ul{\g}}(s_0)) \bigr| \\
&\leq (1-s_0) \bigl| f(\ev_-(\ul{\g'})) - f(\ev_-(\ul{\g})) \bigr| + s_0 \bigl| f(\ev_+(\ul{\g'})) - f(\ev_+(\ul{\g})) \bigr| \;\to\; 0 .
\end{align*}
If $\G_{\ul{\g}}(s_0)$ is a critical point then we can work in a Euclidean coordinate chart $B_\Delta\times B_\Delta$ for $f$ and the metric in which $\G_{\ul{\g}}(s_0)\simeq (0,0)$ and $f(x,y)=f(\G_{\ul{\g}}(s_0)) - \half |x|^2 + \half |y|^2$.
As before, Hausdorff convergence $\ul{\g}'\to\ul{\g}$ implies convergence of the function value $f(\Gamma_{\ul{\g}'}(s_0))\to f(\Gamma_{\ul{\g}}(s_0))$. This implies that $\Gamma_{\ul{\g}'}(s_0)\simeq (x',y')$ lies in the coordinate chart for $\ul{\g}'$ sufficiently close to $\ul{\g}$. Indeed, a trajectory passing the function value $f(\Gamma_{\ul{\g}}(s_0))$ outside of the coordinate chart will never intersect the chart in backward or forward time, so cannot be closer to $\ul{\g}$ than $\Delta$ in the Hausdorff distance.
With that we have $f(\Gamma_{\ul{\g}'}(s_0))=f(\Gamma_{\ul{\g}}(s_0))-\half |x'|^2+ \half |y'|^2$ and can deduce $|x'|-|y'|\to 0$.
Moreover, $\ul{\g}$ is part of a trajectory that breaks or ends at $(0,0)$, so $\im\ul{\g}\subset B_\Delta\times \{0\} \cup \{0\} \times B_\Delta$ and hence the distance between $(x',y')$ and $\im\ul{\g}$ is bounded below by $\min\{|x'|,|y'|\}$. On the other hand, this distance is bounded above by the Hausdorff distance. So its convergence to zero implies that $\Gamma_{\ul{\g}'}(s_0)\simeq(x',y') \to (0,0)\simeq \Gamma_{\ul{\g}}(s_0)$.
\end{proof}

Finally, we prove the topological content of Theorem~\ref{thm corner} and deduce the smooth structure from Theorem~\ref{thm:global with ends} and the following topological conjugacy.

\begin{rmk} \label{franks}
Let $\Psi_s$ be the negative gradient flow of a Morse-Smale pair. Then there exists a homeomorphism $h:X\to X$ such that $h\circ\Psi_s = \Psi^0_s\circ h$, where $\Psi^0_s$ is the flow of a Euclidean Morse-Smale pair. Let us give a few more details on the proof outlined in \cite{franks}.

Near each critical point we can choose coordinates $X \supset U \simeq B_\Delta\subset \R^n$ in which the Hessian ${\rm D}\nabla f(p) \simeq {\rm diag}(\l_1,\ldots,\l_n)$ is diagonalized and $p\simeq 0$. Let $Y_{\rm lin}(\ul x):= \sum \l_i x_i \partial_{x_i}$ denote the linearized vector field, and let $\phi\in \cC^\infty([0,\Delta), [0,1])$ be a compactly supported cutoff function with $\phi|_{[\frac \Delta 2,\Delta)}\equiv 1$. Then $Y_r(\ul{x}):= (1-\phi(r^{-1}|\ul{x}|))\nabla f(\ul{x}) + \phi(r^{-1}|\ul{x}|) Y_{\rm lin}(\ul{x})$ defines vector fields on $X$ that $\cC^1$-converge to $\nabla f$ with $r\to 0$.
So by structural stability \cite{palis,palis-smale} for some $r>0$ the flows of $\nabla f$ and $Y_r$ are topologically conjugate, with $Y_r$ still satisfying the Smale condition (transversality of stable and unstable manifolds).

Next we construct a further homeomorphism $h:X\to X$ supported in the balls $B_{\frac \Delta 2 r}$ near each critical point. In the local coordinates we have $\l_i\neq 0$ by non degeneracy, so $ x \mapsto  (\frac \Delta 2 r)^{1-|\l|^{-1}} {\rm sign}(x) |x|^{|\l|^{-1}}$ defines a homeomorphism of $[-\frac \Delta 2 r,\frac \Delta 2 r]$, which we can extend smoothly to $[-\Delta,\Delta]$ such that $h_i(x)=x$ near $|x|=\Delta$.
Then $h: (x_i) \mapsto (h_i(x_i))$ extends to a homeomorphism of $X$ that is smooth on the complement of the critical points and pulls back $Y_r$ to a vector field $h^*Y_r$ that has the standard form $\sum {\rm sign}(\l_i) x_i \partial_{x_i}$ on a neighbourhood of each critical point, and hence smoothly extends by $h^*Y_r|_{{\rm Crit}(f)}:=0$.
Moreover, this homeomorphism is the identity on the complement of neighbourhoods of the critical points, and within these neighbourhoods leaves the unstable and stable manifolds of the critical point invariant. Thus the stable and unstable manifolds of $h^*Y_r$ agree with those of $Y_r$ on the complement of the neighbourhoods of critical points, which suffices to guarantee the Smale condition.
(Transversality between given unstable and stable manifolds can be checked at a single regular level set, since it is preserved by the flow.)
Moreover, in the coordinates near each critical point, $h^*Y_r$ is the negative gradient of a standard Morse function $ \frac 12 \sum {\rm sign}(\l_i)x_i^2$ with respect to the Euclidean metric. Now by the classification of gradient dynamical systems \cite{smale}, there is a Morse function $f^0:X\to\R$ which coincides with the given functions near critical points up to a constant, and for which $h^*Y_r$ is negative gradient-like, i.e.\ ${\rd f^0(h^*Y_r)<0}$ at noncritical points. Finally, one finds a metric such that $h^*Y_r=-\nabla f^0$ and that equals the Euclidean metric near each critical point.
Indeed, starting with any metric $\ti g$ equal to the Euclidean near critical points, we have $h^*Y_r=-\ti\nabla f^0$ near the critical points and $\ti g ( h^*Y_r ,-\ti\nabla f^0 ) <0$ elsewhere. Then it remains to smoothly adjust $\ti g$ on ${\rm span}( h^*Y_r ,-\ti\nabla f^0)$, which is an exercise in linear algebra.
\end{rmk}

\begin{proof}[Proof of Theorem~\ref{thm corner}]
The metric axioms are easily checked; in particular we discussed definiteness in Remark~\ref{definite}.
It follows that the space $(\bM(\cU_-,\cU_+),d_\bM)$ is Hausdorff. To check separability just note that the space is a finite union of the sets $\bM(\cU_-,\cU_+)_k$, which themselves are unions of products of finite dimensional submanifolds of $X$.
Note here that due to $f$ being Morse on a compact manifold, there are only finitely many critical point sequences, i.e.\  tuples $q_1\ldots q_k\in{\rm Crit}(f)$ such that $f(q_1)>f(q_2)\ldots>f(q_k)$.
Since we are dealing with a metric space, separability also implies second countability.

Sequential compactness for $\bM(p_-,p_+)$ is proven by \cite[Prp.3]{bh} together with Lemma~\ref{lem:rep}.
For sequences $(\ul\g^n)_{n\in\N}$ in $\bM(X,p_+)$, $\bM(p_+,X)$, or $\bM(X,X)$ we use analogous arguments as follows. Lemma~\ref{lem:rep} provides continuous parametrizations $\G^n:[0,1]\to X$ of $\im\ul\g^n$ with bounded derivative $|\frac{\rd}{\rd s}\G^n(s)|\leq C_\ep$ on the complement of neighbourhoods of the critical points, $\G^n(s)\in X \setminus \{ x\in X \,|\, |\nabla f(x)| < \ep \}$. As in \cite{bh} this proves equicontinuity of the $\G^n$, hence the Arzel\`a-Ascoli theorem provides a $\cC^0$-convergent subsequence of $(\G^n)$. By Lemma~\ref{lem:rep} this implies Hausdorff-convergence of the corresponding subsequence of $(\ul\g^n)$. On $\bM(X,X)$ convergence of the rescaled length in $[0,1]$ follows by taking another subsequence.

For Euclidean Morse-Smale pairs, the manifold with corner structure is provided by the global charts in Theorem~\ref{thm:global with ends} and the canonical manifold structure for each space of unbroken flow lines, given in Section~\ref{moduli}. The open subsets $\cV_t(\ul{q})$ cover $\bM(\cU_-,\cU_+)$ since any generalized trajectory either does not break (hence lies at least in the subset for $\ul{q}=\emptyset$) or breaks at a finite number of critical points $q_1,\ldots,q_k$ and hence lies in $\cV_t(\ul{q})$ for some choice of end conditions $\cQ_0,\cQ_{k+1}$ as in \eqref{choice}.
The smooth structure on this atlas is given by the natural smooth structure on, firstly, the unbroken trajectories $\cM(\cU_-,\cU_+)\subset\bM(\cU_-,\cU_+)$. Secondly, the images of the global charts are open subsets of
$$
\cM(X,q_1) \times [0,2) \times \cM(q_1,q_2)  \ldots \times [0,2) \times \cM(q_k,X) ,
$$
or in the special case $k=1$ and $\cQ_0=\cQ_2=\ti U(q_1)$ of
$$
 \cM(\ti U(q_1),q_1) \times [0,1] \times \cM(q_1,\ti U(q_1)) ,
$$
all of which have the natural structure of a manifold with boundary and corners.
Using these charts, the $k$-stratum $\bM(\cU_-,\cU_+)_k$ naturally is the subset of $(k-1)$-fold broken trajectories, except that $\bM(X,X)_1$ has as additional boundary stratum the trajectories of length $0$. The latter appear in the chart $\cV_t(X,X)=\cM(X,X) \simeq [0,\infty)\times X$, where $\ev_-: \partial\cV_t(X,X)\simeq\{0\}\times X \to X$ identifies the boundary component, and in the chart $\cV(\ti U(q),\ti U(q))$, where $\ev_-$ identifies the boundary component $\phi(q)^{-1}\bigl(\cM(\ti U(q),q) \times\{1\} \times \cM(q,\ti U(q))\bigr)$ with $\ti U(q)\subset X$.

The transition maps between different charts with different critical point sequences can be read off from Remark~\ref{rmk:gen comp}.
If $\ul{q}$ and $\ul{Q}$ are related by inserting critical points into $\ul{q}$ and potentially changing the end conditions, then the transition map for $\cV_t(\ul{q}) \cap \cV_t(\ul{Q})$ is
$$
\phi(\ul{Q})|_{ \cV_t(\ul{q}) \cap \cV_t(\ul{Q}) }
\circ \phi(\ul{q})^{-1}
=
\phi(\ul{q}^0)\times{\rm Id}\times \phi(\ul{q}^1) \ldots \times{\rm Id}\times\phi(\ul{q}^k) ,
$$
a product of chart maps on $\cV_t(\ul{q}^0)_0 \times [0,t) \times \cV_t(\ul{q}^1)_0
\ldots \times [0,t) \times \cV_t(\ul{q}^k)_0$,
where they are diffeomorphisms by .
Generally, if $\ul{q}'$ and $\ul{q}''$ contain different critical points and $\cV_t(\ul{q}') \cap \cV_t(\ul{q}'')\neq\emptyset$, then $\ul{Q}:=\ul{q}'\cup \ul{q}''$ (with the induced end conditions) also is a critical point sequence (since $\cV_t(\ul{Q})$ contains this nonempty intersection).
More precisely, $\cV_t(\ul{q}') \cap \cV_t(\ul{q}'') \subset \cV_t(\ul{Q})$ is a subset of those trajectories that break at most at the critical points $\ul{q}'\cap\ul{q}''$,
hence is contained in both  $\cV_t(\ul{q}') \cap \cV_t(\ul{Q})$ and  $\cV_t(\ul{q}'') \cap \cV_t(\ul{Q})$. Now the transition map is a composition of the corresponding two transition maps of the previous type, and hence is smooth.

The compatibility above also applies to the case of $\ul{q}=(\cQ_0,\cQ_1)$ being the trivial critical point sequence with any end conditions, when $\phi(\ul{q})={\rm Id}_{\cM(\cQ_0,\cQ_1)}$. It remains to consider the transition map on an overlap of domains
$\cV_t(\cQ_0,q_1,\ldots,q_{k}, \cQ_{k+1}) \cap \cV_t(\cQ'_0,q_1,\ldots,q_{k}, \cQ'_{k+1})$ for the same critical points but different end conditions.
It is smooth since by Theorem~\ref{thm:global with ends}~(iv) it is the reparametrization in the last or first real valued parameter.

The proof of smoothness for the evaluation maps for Euclidean Morse-Smale pairs is given in Remark~\ref{rmk:smooth eval} as part of the proof of Theorem~\ref{thm:global with ends}.
For a general Morse-Smale pair, the topological conjugation of Remark~\ref{franks} induces homeomorphisms between the Morse trajectory spaces
\begin{equation} \label{eq:h}
h^* : \bM_{\Psi}(\cU_-,\cU_+) \to \bM_{\Psi^0}(\cU_-,\cU_+) , \qquad
([\g_i])_{i=0,\ldots,k} \mapsto ([h\circ\g_i])_{i=0,\ldots,k} .
\end{equation}
Indeed, this is a well defined map under reparametrizations; it preserves the length (in time) of trajectories in $\cM(X,X)$, and transforms the images by a homeomorphism $\overline{\im(h\circ\g_i)} = h(\overline{\im\g_i})$. Hence both $h^*$ and its inverse, given by composition with $h^{-1}$, are continuous in the Hausdorff metric.
Now the smooth structure on $\bM_{\Psi^0}(\cU_-,\cU_+)$ constructed above can be pulled back with $h^*$ to equip $\bM_{\Psi}(\cU_-,\cU_+)$ with a smooth structure whose corner strata are given by broken trajectories as claimed, since $h^*$ preserves the breaking points.
\end{proof}

\section{Restrictions to local and connecting trajectory spaces} \label{near crit}

This section constructs natural charts with boundary for the local trajectory spaces near the critical points of a Euclidean Morse-Smale pair. These charts, together with the smooth flow map, will induce the smooth structure on the general Morse trajectory spaces. For that purpose we construct restriction maps from general Morse trajectory spaces to the local trajectory spaces as well as to connecting trajectory spaces of unbroken flow lines between the boundaries of neighbourhoods of different critical points.

\subsection{Trajectories near critical points}

Let $(f,g)$ be a Euclidean Morse-Smale pair as in Definition~\ref{ems}. Then for some $\Delta>0$ and any $p\in{\rm Crit}(f)$ we have normal coordinates
$$
\R^{n-|p|}\times \R^{|p|}\supset B^{n-|p|}_{2\Delta}\times B^{|p|}_{2\Delta}
\; \overset{\phi_p}{\longrightarrow} \; \ti U(p)\subset X
$$
on the product of open balls such that $\phi_p(0,0)=p$ and
\begin{align}\label{normal 2}
(\phi_p^* f)(x,y) = f(p) _+ \half  {\textstyle \sum_i} x_i^2  - \half {\textstyle \sum_j} y_j^2 , \qquad
(\phi_p^* g) &={\textstyle \sum_i} \rd x_i \otimes \rd x_i +  {\textstyle \sum_j} \rd y_j \otimes \rd y_j .
\end{align}
Here we write $x=(x_i)_{i=1,\ldots,n-|p|}$ and will abbreviate $|x|^2=\sum_i x_i^2$ and similarly for $y=(y_j)_{j=1,\ldots,|p|}$.
These coordinates are unique up to orthogonal diffeomorphisms ${O(n-|p|)} \times {O(|p|)}$ and the choice of $\Delta>0$. We choose $\Delta>0$ so small that the closure of the neighbourhoods $\ti U(p)$ for different critical points $p$ are disjoint.

\begin{rmk}\label{rmk:U small}
For future purposes we note that by sufficiently small choice of $\Delta>0$ we can guarantee that there exists a finite flow line from $\ti U(p^-)$ to $\ti U(p^+)$ iff there exists an unbroken Morse trajectory between $p^-$ and $p^+$. That is, we may assume
$$
(\ev_-\times \ev_+)(\cM(X,X)) \;\cap\;  \bigl( \ti U(p^-)\times \ti U(p^+) \bigr) \neq \emptyset
\qquad\Longleftrightarrow\qquad
\cM(p^-,p^+)\neq \emptyset .
$$
This is possible since, on the one hand, given $\Delta>0$, every infinite flow line in $\cM(p^-,p^+)$ contains a finite part that intersects $\ti U(p^-)$ and $\ti U(p^+)$.
On the other hand, suppose that we cannot choose $\Delta>0$ sufficiently small for the opposite implication to hold. Then we find $T^\pm_i\in \R_+$ and $x_i\in X$ in the complement of separating neighbourhoods of $p^-\neq p^+$ such that $\Psi(\pm T^\pm_i,x_i)\to p^\pm$. By continuity of the flow we deduce $T^\pm_i\to\infty$, and by compactness of $X$ may choose a subsequence of the $x_i$ converging to $x\in X\setminus\{p^-,p^+\}$, hence $\Psi(T_i^\pm,x)\to p^\pm$, proving the assertion by contradiction.

The analogous assertion for half infinite Morse trajectories,
${\ev_-(\cM(X,X)) \;\cap\;  \ti U(p^-) \neq \emptyset}$
$\Leftrightarrow \cM(p^-,X)\neq \emptyset$ holds automatically since by definition $\cM(p^-,X)$ always contains a constant trajectory; and similarly for $\cM(X,p^+)$.
\end{rmk}

The gradient in these coordinates is $\nabla f ({x},{y}) = ({x} , - {y})$, so the negative gradient flow is
$$
\Psi_t({x},{y}) = (e^{-t}{x} , e^{t}{y}).
$$
In particular, the identification of the trajectory spaces $\cM(\tilde U(p), p)$ and $\cM(p, \tilde U(p))$ with the stable resp.\ unstable manifold in normal coordinates yields balls
\begin{align} \label{Bpm}
\ev_- : \; \cM(\tilde U(p), p) &\;\overset{\sim}{\to}\; W^+_q\cap\tilde U(p) \;\simeq\; B^{n-|p|}_{2\Delta}\times \{0\}\;=:\ti B^+_p , \\
\ev_+ : \; \cM(p, \tilde U(p)) &\;\overset{\sim}{\to}\; W^-_q\cap\tilde U(p)\;\simeq\;\;\, \{0\} \times B^{|p|}_{2\Delta}\;\;\,=:\ti B^-_p. \nonumber
\end{align}
From now on we will identify points in normal coordinates $(x,y)\in\ti B^-_p\times \ti B^+_p$ with their image $\phi_p(x,y)\in\ti U(p)\subset X$.
In particular, we use these coordinates to construct the global chart in Theorem~\ref{thm:global with ends} for trajectories near the critical point $p$. 

\begin{lem} \label{lem:tiphi}
The open set $\ti\cV(p) := (\ev_-\times \ev_+)^{-1}(\ti U(p)\times \ti U(p))  \subset \bM(X,X)$ supports a homeomorphism
$$
\ti\phi(p):= \ti\tau_p \times ({\rm pr}_{\ti{B}^+_p}\times {\rm pr}_{\ti{B}^-_p}) \circ (\ev_-\times\ev_+)  \; : \;  \ti\cV(p) \;\longrightarrow\; [0,1] \times  \cM(\tilde U(p), p) \times \cM(p, \tilde U(p))
$$
given by the evaluations \eqref{eval}, the projections in normal coordinates ${\rm pr}_{\ti{B}^\pm_p} : \ti{B}^+_p\times \ti{B}^-_p \to \ti{B}^\pm_p$, the identification \eqref{Bpm}, and the rescaling of the renormalized length \eqref{length},
\begin{equation} \label{titip}
\ti\tau_p \; : \;  \ti\cV(p) \;\longrightarrow\; [0,1] , \qquad \ul\g \mapsto  e^{- \ell(\g)/(1-\ell(\g))}
=\begin{cases}
e^{-L}  &; \ul{\g} = \bigl(\gamma:[0,L]\to X\bigr) ,\\
0 &; \text{otherwise}.
\end{cases}
\end{equation}
Moreover, $\ti\phi(p)$ satisfies the properties of a global chart in Theorem~\ref{thm:global with ends} as follows.
\begin{enumerate}
\item The restriction to the unbroken trajectories $\ti\cV(p)_0=\ti\cV(p)\cap \bM(X,X)_0$
 is a diffeomorphism
$\ti\phi(p)|_{\ti\cV(p)_0} : \ti\cV(p)_0 \to  (0,1]\times \cM(\tilde U(p), p) \times \cM(p, \tilde U(p))$.
\item
The restriction to the maximally broken trajectories $\ti\cV(p)_1=\ti\cV(p)\cap\bM(X,X)_1$ is the canonical bijection
$ \ti\cV(p)_1 \to \{0\} \times \cM(\ti U(p),p) \times  \cM(p,\ti U(p)), (\gamma_-,\gamma_+) \mapsto (0, \gamma_-, \gamma_+)$.
\setcounter{enumi}{3}
\item
The parameter $e^{-T}\in[0,1]$ encodes the length $T$ of the time interval $[0,T]$ on which the trajectory is defined. In particular, $e^{-T}=0$ corresponds to the trajectory breaking at $p$, and $e^{-T}=1$ corresponds to a trajectory of length $0$.
\end{enumerate}
Finally, the evaluation maps are smooth with respect to this chart, that is
$(\ev_-\times\ev_+)\circ\ti\phi(p)^{-1}$ maps smoothly to $X\times X$.
\end{lem}

\begin{rmk} \label{rmk:tiphi}
The inverse of the homeomorphism $\ti\phi(p)$ in Lemma~\ref{lem:tiphi},
$$
\ti\phi(p)^{-1} : \;  [0,1] \times  \cM(\tilde U(p), p) \times \cM(p, \tilde U(p))  \;\longrightarrow\; \ti\cV(p) , \qquad (\tau,x,y) \mapsto \ul\g_{\tau,x,y} ,
$$
is explicitly given in the normal coordinates by the unbroken flow lines for $\tau>0$,
\begin{equation} \label{gamma T}
\gamma_{\tau,x,y}: [0,T]\to\ti U(p), \; s \mapsto (e^{-s}{x},e^{s-T}{y}) \qquad\text{with}\;\; T:=-\ln \tau ,
\end{equation}
and the broken flow lines $\ul{\g}_{\tau,x,y}:=(\gamma_+,\gamma_-)$ for $\tau=0$ given by
\begin{equation} \label{gamma infty}
\gamma_+: [0,\infty)\to\ti U(p), \; s \mapsto (e^{-s}{x},0) ,\qquad
\gamma_-: (-\infty,0]\to\ti U(p), \; s \mapsto (0, e^s {x}) .
\end{equation}
\end{rmk}

\begin{proof}[Proof of Lemma~\ref{lem:tiphi} and Remark~\ref{rmk:tiphi}.]
Bijectivity of $\ti\phi(p)$, the canonical form (ii), and the formulas for $\ti\phi(p)^{-1}$ are seen by checking that \eqref{gamma T} and \eqref{gamma infty} uniquely characterize the trajectories of the flow $\Psi_t$ in $\ti U(p)$.
Indeed, these trajectories can break at most at $p$, hence are determined by an initial point $(x,y')$ and end point $(x',y)$. If they are connected by a flow of lenght $T$ then $y'=e^{-T}y$ and $x'=e^{-T}x$. If they are connected by a broken flow, then $y'=0$ and $x'=0$ corresponding to $\tau=0$. Continuity of $\ti\phi(p)$ follows from the continuity of the evaluation maps (see Lemma~\ref{lem:cont eval}), the renormalized length (by definition of the metric on $\bM(X,X)$), the projections in normal coordinates, and the diffeomorphism $[0,1)\ni\ell\mapsto e^{- \ell/(1-\ell)}\in(0,1]$, which extends continuously to $1\mapsto 0$.
So it remains to check (i) and the continuity of $\phi(p)^{-1}$.

The renormalized length $\ell(\ul{\gamma}_{\tau,{x},{y}})=\frac{-\ln\tau}{1-\ln\tau}$ is a continuous function of $\tau\in(0,1]$ which for $\tau\to 0$ converges to
$\lim_{\tau\to 0}\frac{-\ln\tau}{1-\ln\tau} =1=\ell(\ul{\gamma}_{0,{x},{y}})$.
Hence we obtain uniform continuity (independent of ${x},{y}$) with respect to the length term in the metric on $\bM(X,X)$.
To check continuity of the Hausdorff distance near a fixed  $(\tau,{x},{y})\in [0,1)\times \ti B^+_p  \times \ti B^-_p$ note that the image of the generalized trajectory is
$$
\im\ul{\gamma}_{\tau,{x},{y}} =
\bigl\{ \bigl(z \cdot x , \tfrac\tau z \cdot {y} \bigr) \st z\in[\tau,1] \bigr\} \cup
\bigl\{ \bigl(\tfrac\tau w \cdot {x} , w \cdot {y} \bigr) \st w\in[\tau,1] \bigr\} .
$$
(In case $\tau>0$ both sets are the same.)
For $\tau>0$ one easily obtains for $(\tau',{x}',{y}')\in [0,1) \times \ti B^+_p \times \ti B^-_p$ the estimate
$
d_H\bigl( \ul{\gamma}_{\tau',{x}',{y}'}, \ul{\gamma}_{\tau,{x},{y}} \bigr)
\leq  | {x}' - {x} | + | {y}' - {y} | + 2\Delta (1 + \tau^{-1}) |\tau' - \tau|
$.
For $\tau=0$ we obtain
$
d_H\bigl( \ul{\gamma}_{\tau',{x}',{y}'}, \ul{\gamma}_{0,{x},{y}} \bigr)
\leq  | {x}' - {x} | + | {y}' - {y} | + 4\Delta \sqrt{\tau'}
$.
Indeed, the distance to the point $(z\cdot {x} , 0 )$
(and similarly for $(0, w\cdot {y})$) for all $z\in[0,1]$ is
\begin{align*}
d_{\R^n} \bigl( (z\cdot {x} , 0 ) \,,\, \bigl\{ \bigl( z' \cdot {x}' , \tfrac{\tau'}{z'} \cdot {y}' \bigr) \st  z'\in[\tau',1] \bigr\}   \bigr)
&\leq z  | {x}' - {x} | + |{x}'| \cdot | z' - z | +  |{y}'| \cdot \tfrac{\tau'}{z'}  \\
&\leq  | {x}' - {x} | + 4 \Delta \sqrt{\tau'}
\end{align*}
by choosing $z'=\max\{z,\sqrt{\tau'}\}$ such that $0 \leq z'-z = \max\{ 0 , \sqrt{\tau'} - z \}\leq\sqrt{\tau'}$.
Conversely, the distance to the point $( z' \cdot {x}' , \tfrac{\tau'}{z'} \cdot {y}' )$ for $z'\in[\sqrt{\tau'},1]$ can be estimated by picking $z=z'$ as
$
d_{\R^n} \bigl( \bigl\{ (z\cdot{x} , 0 ) \st z\in[0,1] \bigr\} \,,\, \bigl( z' \cdot {x}' , \tfrac{\tau'}{z'} \cdot {y}' \bigr)  \bigr)
\leq  | {x}' - {x} | + 2 \Delta \sqrt{\tau'}
$,
and for all remaining $w'=\tau'/z'\in[\sqrt{\tau'},1]$ by picking $w=w'$ as
$
d_{\R^n} \bigl( \bigl\{ (0, w\cdot {y} , 0 ) \st w\in[0,1] \bigr\} \,,\, \bigl( \tfrac{\tau'}{w'} \cdot {x}' , w' \cdot  {y}' \bigr)  \bigr)
\leq  | {y}' - {y} | + 2 \Delta \sqrt{\tau'}
$.
This finishes the proof of continuity of $\phi(p)^{-1}$.

For (i) note that in the smooth coordinates $\bM(X,X)_0\simeq [0,\infty)\times X$ we have
$$
\ti\cV(p)_0\simeq \bigl\{ (T,z)\in [0,\infty)\times\ti U(p) \st \Psi_T(z)\in\ti U(p)  \bigr\} .
$$
The smooth structure for the trajectory spaces $\cM(\tilde U(p), p) \simeq \ti B^+_p$ and
$\cM(p, \tilde U(p)) \simeq \ti B^-_p$ is given by \eqref{Bpm}.
Now in these coordinates and with $\ti U(p)\simeq\ti B^-_p\times\ti B^+_p$ the map
$\ti\phi(p) : \bigl( T,(x,y) \bigr) \mapsto \bigl( x , e^{-T} , y \bigr)$ evidently is a diffeomorphism as claimed.
Finally, the evaluation map is given by the evidently smooth map
\begin{align*}
(\ev_-\times\ev_+)\circ\ti\phi(p)^{-1}: \; [0,1] \times  \cM(\tilde U(p), p) \times \cM(p, \tilde U(p)) &\;\longrightarrow \; U(p)\times U(p) \\
(\tau, x, y ) &\;\longmapsto\; \bigl( (x, \tau y) , (\tau x , y ) \bigr) .
\end{align*}
\end{proof}

Next, we introduce the half size neighbourhood of $p$, which is precompact in $\ti U(p)$,
$$
U(p):= \phi_p \bigl(  B^{n-|p|}_{\Delta}\times B^{|p|}_{\Delta} \bigr) .
$$
From the above characterization of Morse trajectories we can read off its entry and exit sets,
$$
\ti S^+_p := \{ |{x}|=\Delta\} = S^+_p \times B^-_p ,
\qquad
\ti S^-_p := \{ |{y}|=\Delta\} = B^+_p \times S^-_p ,
$$
where $S^+_p := \partial B^+_p$ and  $S^-_p := \partial B^-_p$ are spheres in the stable resp.\ unstable manifolds and we abbreviated
$$
B^+_p := B^{n-|p|}_\Delta \simeq W^+_p\cap U(p), \quad B^-_p :=B^{|p|}_\Delta \simeq W^-_p\cap U(p) .
$$
Indeed, $\ti S^+_p\cup \ti S^-_p$ is the boundary of the domain $U(p)\simeq B^+_p \times B^-_p$ and the intersection of any broken or unbroken flow line with $U(p)$ has its endpoints on $\ti S^+_p$ and $\ti S^-_p$.
With this we can introduce the local trajectory space near $p$ as the set of broken or unbroken trajectories that start and end on the entry and exit set,
$$
\bM_p :=  (\ev_-\times \ev_+)^{-1}(\ti{S}^-_p,\ti{S}^+_p) \subset \bM(X,X)
$$
with topology induced from $\bM(X,X)$. The following gives the local trajectory space $\bM_p$ the structure of a smooth manifold with boundary in which the evaluations $\ev_\pm$ are smooth.

\begin{lem} \label{lem:local traj}
The evaluations $(\ev_-\times \ev_+) : \bM_p \to \tilde{S}^+_p \times \tilde{S}^-_p$ composed with the projection
\begin{equation}\label{pr}
{\rm pr}_p: \ti{S}^+_p \times \ti{S}^-_p \to  [0,1) \times S^+_p \times S^-_p , \quad
(x,y',x',y) \mapsto \bigl(  \tfrac{|x'|+|y'|}{2\Delta}  ,x,y  \bigr)
\end{equation}
define a homeomorphism
\begin{equation} \label{Mp hom}
{\rm pr}_p\circ (\ev_-\times \ev_+) : \;
\bM_p \;\longrightarrow \;  [0,1) \times S^+_p \times S^-_p  .
\end{equation}
\end{lem}

\begin{proof}
The map ${\rm pr}_p\circ (\ev_-\times \ev_+)$ is the restriction of the homeomorphism $\phi(p)$ of Lemma~\ref{lem:tiphi} to $\bM_p\subset\ti\cV(p)$. Indeed, the endpoints of a trajectory $\ul{\g}$ of length $T<\infty$ are of the form $(x,y'=e^{-T}y,x'=e^{-T}x,y)$, hence the length parameter $ e^{- \ell(\ul{\g})/(1-\ell(\ul{\g}))}$ is given by $e^{-T}= \tfrac{|x'|+|y'|}{2\Delta}$.
Broken trajectories are of the form $(x,y'=0,x'=0,y)$, hence again the length parameter is given by $0=  \tfrac{|x'|+|y'|}{2\Delta}$.
Here it is important to note that $(\ev_-\times \ev_+)(\bM_p)\subset\ti S^+_p \times \ti S^-_p$ so that the projection map \eqref{pr} is only defined at $({x},{y}',{x}',{y})$ with $|{x}|=|{y}|=\Delta>0$, thus continuous.
Surjectivity onto $[0,1) \times S^+_p \times S^-_p$ follows from checking that the inverse map given by \eqref{gamma T} and \eqref{gamma infty} indeed provides trajectories in $\bM_p$, i.e.\ with endpoints on $\ti S^\pm_p$.
\end{proof}

\subsection{Restrictions to local trajectory spaces}

In the construction of the smooth corner structure for general Morse trajectory spaces we will use restriction maps from the spaces of trajectories passing near a critical point to the local trajectory space of that point.
For that purpose we introduce the following families of open neighbourhoods for $t\in(0,1]$,
\begin{align*}
\ti U_{t}(p) &:= \bigl\{ \phi_p(x,y) \,\bigl|\, |x|<(1+t)\Delta, |y|<(1+t)\Delta ,  |x|  |y| < \Delta^2 t \bigr\} \;\subset X , \\
U_{t}(p) &:= \ti U_t(p)\cap U_t(p) = \bigl\{ \phi_p(x,y) \,\bigl|\, |x|<\Delta, |y|<\Delta, |x|  |y| < \Delta^2 t \bigr\} \;\subset X .
\end{align*}
These neighbourhoods are precompactly nested $\ti U_t(p) \sqsubset \ti U_{t'}(p)$ for $t<t'$ (i.e.\ the compact closure of $\ti U_t(p)$ is contained in $\ti U_{t'}(p)$), and for $t\to 0$ converge to the union of stable and unstable manifold,
$\bigl\{ \phi_p(x,y)  \,\bigl|\, x=0 \;\text{or}\; y=0 \bigr\}=(W^-_p\cup W^+_p)\cap U(p)$. The nesting $ U_t(p) \subset  U_{t'}(p)$ and convergence also holds for $U_t(p)$, all of which are precompact in $\ti U(p)$, and with $U_1(p)=U(p)$.
We will keep identifying $\ti U(p)$ with $\ti B^-_p\times \ti B^+_p \subset \R^{n-|p|}\times\R^{|p|}$.

\begin{rmk} \label{rmk:local traj}
The entry and exit sets for $U_t(p)$ are the nested subsets
$$
\ti S^-_p(t) := S^-_p\times t B^+_p \subset \ti S^-_p , \qquad\qquad \ti S^+_p(t) := t B^-_p \times S^+_p \subset \ti S^+_p .
$$
The set of trajectories traversing $U_{t}(p)$ is
$\bM_{p,t} := (\ev_-\times \ev_+)^{-1}(\ti S^-_p (t), \ti S^+_p(t)) \subset \bM_p$.
The homeomorphism \eqref{Mp hom} then restricts to
$\bM_{p,t} \overset{\sim}{\to}  [0,t) \times S^+_p \times S^-_p$.
The global chart for the tuple $\ul{q}=(\ti U(p),p,\ti U(p))$ and $t>0$ in Theorem~\ref{thm:global with ends} will be defined as restriction $\phi(\ul{q}):=\ti\phi(p)|_{\cV_t(\ul{q})}$ to the open subset $\cV_t(\ul{q})\subset\ti\cV(p)$ given by those trajectories that intersect $\ti U_t(p)$.
Using Remark~\ref{rmk:tiphi} we may read off the image
\begin{align*}
\ti\phi(p)\bigl(\cV_t(\ul{q})\bigr)  &\subset \cM(\tilde U(q), q) \times  [0,1]\times \cM(q, \tilde U(q))    \\
&= \bigl\{ (\g_0,E,\g_1) \st   E |\ev_-(\g_0)| ,E |\ev_+(\g_1)| <(1+t)\Delta , E |\ev_-(\g_0)| |\ev_+(\g_1)| <t\Delta^2 \bigr\} .
\end{align*}
Indeed, the end points are of the form $\ev_-(\g_0)=(x,Ey)$, $\ev_+(\g_1)=(Ex,y)$, and since the product of norms of the coordinates in $\ti B^+_p$ and $\ti B^-_p$ is preserved by the flow, the condition $E |x| |y| <t\Delta^2$ is equivalent to the trajectory intersecting $\Psi_{\R}(\ti U_t(p))$.
The conditions $E |\ev_-(\g_0)|<(1+t)\Delta$ and $E |\ev_+(\g_1)| <(1+t)\Delta$ are equivalent to the trajectory not being entirely contained in  $\Psi_{\R}(\ti U_t(p))\setminus \ti U_t(p)$.
\end{rmk}

In order to construct restriction maps from spaces of Morse trajectories traversing $U(p)$ to the local trajectory space $\bM_p$ we will use evaluation at the entry and exit sets $\ti{S}^\pm_p$. These are transverse to $\nabla f$, hence are local slices for the flow such that the evaluation maps (for $\cU_\pm$ any open sets or critical points)
$$
\ev_{\ti{S}^\pm_p}: \;\;  \bM(\cU_-,\cU_+;\Psi_{\R_-}(\ti{S}^\pm_p),\Psi_{\R_+}(\ti{S}^\pm_p) ) \; \longrightarrow \; \ti{S}^\pm_p
$$
are well defined, see Definition~\ref{def intersecting}.
From these we can construct a restriction map
\begin{equation}\label{rest1}
 (\ev_-\times \ev_+)^{-1}\circ  (\ev_{\ti{S}^+_p}\times \ev_{\ti{S}^-_p}) :\;
\bM(\cU_-,\cU_+;\Psi_{\R_-}(\ti{S}^+_p),\Psi_{\R_+}(\ti{S}^-_p) ) \; \longrightarrow \; \bM_p ,
\end{equation}
which is well defined and continuous since it can be written as composition of the homeomorphism \eqref{Mp hom} with
${\rm pr}_p\circ (\ev_{\ti{S}^+_p}\times \ev_{\ti{S}^-_p}) :  \bM(\cU_-,\cU_+;\Psi_{\R_-}(\ti{S}^+_p),\Psi_{\R_+}(\ti{S}^-_p) ) \to [0,1) \times S^+_p \times S^-_p$.
Continuity of the latter map follows from Lemma~\ref{lem:cont eval} for the evaluation map and continuity of the projection ${\rm pr}_p$ defined in \eqref{pr} holds as in Lemma~\ref{lem:local traj}.
In particular, the latter map contains the (rescaled) transition time through $U(p)$, which we separately denote by
\begin{align}\label{trans}
\tau_p := & \; E_p\circ (\ev_{\ti{S}^+_p}\times \ev_{\ti{S}^-_p}) :  \bM(\cU_-,\cU_+;\Psi_{\R_-}(\ti{S}^+_p),\Psi_{\R_+}(\ti{S}^-_p) ) \longrightarrow [0,1) \\
\text{with}\;\;\; &  \;
E_p({x},{y}',{x}',{y}) :=\tfrac{|{y}'|+|{x}'|}{2\Delta}.  \nonumber
\end{align}
Note that the restriction of trajectories intersecting $U_t(p)$ then takes values in $[0,t) \times S^+_p \times S^-_p$ with actual transition time $-\ln \tau_p > -\ln t$ bounded below.

The above restriction maps will be used in the construction of charts for Morse trajectories starting and ending outside of $U(p)$.  The case of trajectories that start and end in $\ti U(p)$ was already dealt with in Lemma~\ref{lem:tiphi}. So it remains to construct restrictions to local trajectory spaces for trajectories with one end in $\ti U(p)$.
Let us give an outlook on the use of the restriction maps in order to justify the subsequent technical constructions.

The global charts for Morse trajectory spaces will be obtained from a fibered product of local trajectory spaces and spaces of flow lines between the exit and entry set $\ti S^-_{p}$ and $\ti S^+_{p'}$ of different critical points. The construction of tubular neighbourhoods of $\cM(p,p')$ in the latter will require a smooth extension of restriction maps to trajectories from $\ti S^-_p \subset \partial U(p)$ to $X\setminus U(p)$. Hence we will not restrict ourselves to trajectories intersecting $U(p)$.
However, evaluation at $\ti S^-_p$ is still important, so we will extend this definition to trajectories starting in $\Psi_{[0,\infty)}(\ti S^-_p)$ as the unique intersection point of the extended trajectory.
With this the natural transition time for trajectories from $\ti U(p)$ to $X\setminus U(p)$ is the time for which the trajectory is defined and contained in $\Psi_{\R_-}(U(p))$. For trajectories starting in $\ti U(p) \cap \Psi_{[0,\infty)}(U(p))=\Psi_{[0,\ln 2)}(\ti S^-_p)$ this leads to negative numbers, or in the exponential rescaling to factors $E\in[1,2)$ between the $y$-coordinates of initial point and evaluation to~$\ti S^-_p$.

So for trajectories with initial or end point in $\ti U(p)$ we consider the local trajectory spaces
\begin{align*}
\leftexp{-}{\phantom\cM}\hspace{-4.5mm}\bM_{p,t} &:=  (\ev_-\times \ev_+)^{-1}(\ti U(p), \partial \ti U(p) ) \subset \bM(X,X; \ti U_t(p)), \\
\leftexp{+}{\phantom\cM}\hspace{-4.5mm}\bM_{p,t} &:=  (\ev_-\times \ev_+)^{-1}(\partial \ti U(p), \ti U(p) ) \subset \bM(X,X; \ti U_t(p)) .
\end{align*}
The intersection condition $\im\ul\g\cap\ti U_1(p)\neq\emptyset$ implies that the exit resp.\ entry point of the trajectory lies in $\Psi_{\mp\ln 2}(\ti S^\pm_p)\subset \partial \ti U(p)$.
We may hence define extended evaluation maps at $\ti S^\pm_p$ on $\leftexp{\pm}{\phantom\cM}\hspace{-4.5mm}\bM_p$, and more generally
\begin{align}\label{ext eval}
\ev_{\ti{S}^-_p} &: \bM(\ti U(p),\cU_+;\Psi_{\R_+}(\ti{S}^-_p) ) \to \ti{S}^-_p ,
\qquad
\ul{\g} \mapsto \ti{S}^-_p \cap \Psi_\R(\im\ul\g) , \\
\ev_{\ti{S}^+_p} &: \bM(\cU_-,\ti U(p);\Psi_{\R_-}(\ti{S}^+_p) ) \to \ti{S}^+_p ,
\qquad
\ul{\g} \mapsto \ti{S}^+_p \cap \Psi_\R(\im\ul\g) .  \nonumber
\end{align}
We use these evaluations to give the local trajectory spaces a smooth structure as follows.

\begin{lem} \label{lem:local traj +-}
The extended evaluation maps \eqref{ext eval} are continuous, and smooth when restricted to $\cM(\cU_-,\cU_+)$.
The evaluations $\ev_- \times \ev_{\ti{S}^-_p}$ resp.\ $\ev_{\ti{S}^+_p} \times \ev_+$ composed with
\begin{align}\label{+-pr}
\leftexp{-}{\rm pr}_p \;:\; \ti U(p) \times \ti{S}^-_p \;\to\;  [0,2) \times \ti B^+_p \times S^-_p ,
\quad & \quad
\big((x,y'),(x',y)\bigr) \;\mapsto\; \bigl( \tfrac{|y'|}{\Delta} , x , y  \bigr) , \\
\leftexp{+}{\rm pr}_p \;:\; \ti{S}^+_p  \times \ti U(p) \;\to\;  [0,2) \times S^+_p \times \ti B^-_p ,
\quad & \quad
\big((x,y'),(x',y)\bigr)  \;\mapsto\; \bigl( \tfrac{|x'|}{\Delta} , x , y   \bigr)  \nonumber
\end{align}
define homeomorphisms
\begin{align*}
\leftexp{-}{\phantom\cM}\hspace{-4.5mm}\bM_{p,t} &\;\longrightarrow\;
\bigl\{ (E,x,y) \in [0,1+t) \times \ti B^+_p \times S^-_p  \,\big|\, E |x| < t\Delta \bigr\}, \\
\qquad
\leftexp{+}{\phantom\cM}\hspace{-4.5mm}\bM_{p,t} &\;\longrightarrow\;
\bigl\{ (E,x,y) \in [0,1+t) \times S^+_p \times \ti B^-_p \,\big|\, E |y| < t\Delta \bigr\}.
\end{align*}
\end{lem}

\begin{proof}
The initial points $\ev_-(\ul{\g})$ of trajectories in $\bM(\ti U(p),\cU_+;\Psi_{\R_+}(\ti{S}^-_p)$ lie within $\ti U_1(p)$ by the intersection condition with $\Psi_{\R_+}(\ti{S}^-_p)=\{(x,y) \,|\, |x||y|\leq \Delta^2, \Delta \leq |y| \leq 2 \Delta \}$. The standard evaluation map $\ev_{\ti{S}^-_p}(\ul{\g})$ is well defined for $\ev_-(\ul{\g})\in \ti U_1(p) \setminus \Psi_{[0,\infty)}(\ti{S}^-_p)$ and has the claimed regularity by Lemma~\ref{lem:cont eval}.
So it suffices to establish the regularity of the extended evaluation on the open subset $\ev_-^{-1}(\ti U_1(p)\setminus W^+_p)$, where it can be expressed as composition $\ev_{\ti{S}^-_p}= \rho_- \circ \ev_-$  with the smooth map in normal coordinates
$$
\rho_- : \bigl\{ (x,y) \in \ti U(p) \,\big|\, y\neq 0 \bigr\} \to \ti{S}^-_p , \qquad
(x,y) \mapsto \bigl(\tfrac{|y|}{\Delta}|x|,\tfrac{\Delta}{|y|} y\bigr) .
$$
The extension of $\ev_{\ti{S}^+_p}$ has an analogous expression.
The regularity then follows from the fact that the endpoint evaluations $\ev_\pm$ are continuous resp.\ smooth on $\cM(\cU_-,\cU_+)$ by Lemma~\ref{lem:cont eval}.

As in the proof of Lemma~\ref{lem:local traj}, note that the map $\leftexp{-}{\rm pr}_p\circ (\ev_-\times \ev_{\ti S^-_p})$ restricted to $\ev_-^{-1}(\Psi_{\R_-}(\ti S^-_p))$ (i.e.\ trajectories that actually intersect $\ti S^-_p$) is a restriction of the homeomorphism $\phi(p)$ from Lemma~\ref{lem:tiphi}, mapping onto the subset $\{E<1\}$ of the claimed image.
The complement of $\ev_-^{-1}(\Psi_{\R_-}(\ti S^-_p))\subset \leftexp{-}{\phantom\cM}\hspace{-4.5mm}\bM_{p,t}$ are the trajectories that intersect $\ti U_t(p)\setminus U(p) =  \Psi_{[0,\ln (1+t))}(\ti{S}^-_p)\cap \{ |x||y| < t \Delta^2\}$ but not $\Psi_{\R_-}(\ti S^-_p)$. These are uniquely determined by their initial points $\{(x,y')\in\ti U(p) \,|\, \Delta \leq |y'|\leq(1+t)\Delta, |x||y'|<t\Delta^2 \}$. Their generalized evaluation at $\ti{S}^-_p$ is given as above by $(x',y)=\rho_-(x,y')$, and $\leftexp{-}{\rm pr}_p$ identifies these pairs of points with the subset $\{1\leq E < 1+t\}$ of the claimed image. This shows bijectivity of $\leftexp{-}{\rm pr}_p\circ (\ev_-\times \ev_{\ti S^-_p})$. Continuity and openness can be checked separately on the open sets $\ev_-^{-1}(\Psi_{\R_-}(\ti S^-_p))$ and $\ev_-^{-1}(\Psi_{(-\ln 2,\ln 2)}(\ti{S}^-_p))$, which cover the domain. On the first subset, regularity follows from the homeomorphism property of $\phi(p)$. On the latter, we may use the coordinate chart $\ev_-$ to express the map in local coordinates as the evident homeomorphism
\begin{align*}
\bigl\{(x,y')\in\ti U(p) \,\big|\, |y'|>\tfrac\Delta 2, |x||y'|<t\Delta^2 \bigr\} &\longmapsto \bigl\{ (E,x,y) \in (\half,2) \times \ti B^+_p \times S^-_p \,\big|\, E |x| <t\Delta \bigr\} , \\
(x,y') &\longmapsto \bigl( \tfrac{|y'|}{\Delta} , x , \tfrac{\Delta}{|y'|}x  \bigr) .
\end{align*}
This establishes the homeomorphism for $\leftexp{-}{\phantom\cM}\hspace{-4.5mm}\bM_{p,t}$; the proof for $\leftexp{+}{\phantom\cM}\hspace{-4.5mm}\bM_{p,t}$ is analogous.
\end{proof}

Now we obtain restriction maps to these local trajectory spaces (with $t=1$)
\begin{align}  \label{rest2}
(\ev_-\times \ev_+)^{-1}\circ  (\ev_-\times \ev_{\ti{S}^-_p}) :\;
\bM(\ti U(p),\cU_+;\Psi_{\R_+}(\ti{S}^-_p) ) &\longrightarrow \leftexp{-}{\phantom\cM}\hspace{-4.5mm}\bM_p , \\
(\ev_-\times \ev_+)^{-1}\circ  (\ev_{\ti{S}^+_p}\times \ev_+) :\;
\bM(\cU_-,\ti U(p);\Psi_{\R_-}(\ti{S}^+_p) ) &\longrightarrow \leftexp{+}{\phantom\cM}\hspace{-4.5mm}\bM_p , \nonumber
\end{align}
which are well defined and continuous since it they be written as composition of the homeomorphisms
of Lemma~\ref{lem:local traj +-} with $\leftexp{-}{\rm pr}_p\circ (\ev_- \times \ev_{\ti{S}^-_p})$ resp.\ $\leftexp{-}{\rm pr}_p\circ (\ev_{\ti{S}^+_p} \times \ev_+)$.
We separately denote the transition time, namely the rescaling of the time for which the trajectory is defined and contained in $\Psi_{\R_-}(U(p))$ resp.\  $\Psi_{\R_+}(U(p))$, by
\begin{align}  \nonumber
\leftexp{-}{\tau}_p :=& \leftexp{-}{E}_p \circ (\ev_- \times \ev_{\ti{S}^-_p}) : \;\;
\bM(\ti U(p),\cU_+;\Psi_{\R_+}(\ti{S}^-_p) ) \; \longrightarrow\;
[0,2) , \\
\leftexp{+}{\tau}_p := &\leftexp{+}{E}_p \circ (\ev_{\ti{S}^+_p} \times \ev_+) : \;\;
\bM(\cU_-,\ti U(p);\Psi_{\R_-}(\ti{S}^+_p)) ) \; \longrightarrow\;
[0,2) ,  \label{+-trans} \\
\text{with} \;\;\;\; & \leftexp{-}{E}_p({x},{y}',{x}',{y}) := \tfrac {|y'|}{\Delta}, \qquad
\leftexp{+}{E}_p({x},{y}',{x}',{y}) := \tfrac {|x'|}{\Delta}.  \nonumber
\end{align}
A natural extension of the local trajectory spaces of trajectories with one end in $\ti U(p)$ are the spaces of trajectories from $\Psi_{\R_-}(\ti S^+_p)\subset X\setminus\overline{U(p)}$ to $\ti S^-_p$ resp.\ from $\ti S^+_p$ to $\Psi_{\R_+}(\ti S^-_p)\subset X\setminus\overline{U(p)}$,
\begin{align*}
\leftexp{-}{\phantom\cM}\hspace{-4.5mm}\widetilde\cM_p  &:=
 (\ev_-\times \ev_+)^{-1}\bigl( \Psi_{\R_-}(\ti S^+_p) \times \ti S^-_p  \bigr) \subset \bM(X,X)  \\
\leftexp{+}{\phantom\cM}\hspace{-4.5mm}\widetilde\cM_p  &:=
(\ev_-\times \ev_+)^{-1}\bigl(  \ti S^+_p \times \Psi_{\R_+}(\ti S^-_p)  \bigr)
\subset \bM(X,X) .
\end{align*}
The global charts will also involve the restriction maps to these spaces,
\begin{align}  \label{rest3}
(\ev_-\times \ev_+)^{-1}\circ  (\ev_-\times \ev_{\ti{S}^-_p}) :\;
\bM(\Psi_{\R_-}(\ti{S}^+_p),\cU_+;\Psi_{\R_+}(\ti{S}^-_p)) &\longrightarrow \leftexp{-}{\phantom\cM}\hspace{-4.5mm}\widetilde\cM_p , \\
(\ev_-\times \ev_+)^{-1}\circ  (\ev_{\ti{S}^+_p}\times \ev_+) :\;
\bM(\cU_-,\Psi_{\R_+}(\ti{S}^-_p);\Psi_{\R_-}(\ti{S}^+_p)) &\longrightarrow \leftexp{+}{\phantom\cM}\hspace{-4.5mm}\widetilde\cM_p . \nonumber
\end{align}
However, instead of extending the charts for $\leftexp{\pm}{\phantom\cM}\hspace{-4.5mm}\bM_p$,
the natural charts for these trajectory spaces are given by combining the charts for $\bM_p$ with flow times $T^\pm$ outside of $U(p)$.
The restriction maps are then well defined and continuous since they are a composition of the following charts with $\bigl(T^-,\leftexp{-}{\rm pr}_p\circ (\ev_- \times \ev_{\ti{S}^-_p})\bigr)$ resp.\ $\bigl(T^+,\leftexp{+}{\rm pr}_p\circ (\ev_{\ti{S}^+_p} \times \ev_+)\bigr)$.

\begin{lem} \label{lem:local traj last}
The flow times given by solving $\ev_\pm(\ul{\g})\in\Psi_{T^\pm}(\ti S^\mp_p)$ define continuous maps
\begin{align*}
T^- : \bM(\Psi_{\R_-}(\ti S^+_p),\cU_+)  \to  \R_- , \qquad
T^+ : \bM(\cU_-, \Psi_{\R_+}(\ti S^-_p))  \to \R_+ .
 \end{align*}
Restricted to $\cM(X,\cU_+)$ resp.\ $\cM(\cU_-,X)$ the flow times are smooth.
Together with the maps ${\rm pr}_p\circ (\ev_{\ti{S}^+_p}\times \ev_{\ti{S}^-_p})$ they define homeomorphisms
\begin{align*}
\leftexp{-}{\phantom\cM}\hspace{-4.5mm}\widetilde\cM_p \;\longrightarrow \;  \R_-\times  [0,1) \times S^+_p \times S^-_p  , \qquad
\leftexp{+}{\phantom\cM}\hspace{-4.5mm}\widetilde\cM_p
 \;\longrightarrow \; \R_+\times  [0,1) \times S^+_p \times S^-_p  .
\end{align*}
The subspaces of trajectories $\leftexp{\pm}{\phantom\cM}\hspace{-4.5mm}\widetilde\cM_{p,t}$ intersecting $\ti U_t(p)$ have image $\R_+\times  [0,t) \times S^+_p \times S^-_p$.
\end{lem}
\begin{proof}
We may express $T^-$, and similarly $T^+$,  as composition of the evaluation $\ev_-$ and the map $\Psi_{\R_-}(\ti S^+_p) \to\R_- , z \mapsto t$ given by solving  $\Psi(t,z) \in \ti{S}^+_p$. The latter is well defined and smooth by the implicit function theorem since $\ti S^+_p$ is a local slice to the Morse flow.
The regularity of $\ev_-$ is as claimed by Lemma~\ref{lem:cont eval}.

Using the homeomorphism \eqref{Mp hom}, we may view the maps on $\leftexp{\pm}{\phantom\cM}\hspace{-4.5mm}\widetilde\cM_{p,t}$ as products of $T^\pm$ with the continuous restriction map to $\bM_{p,t}$. They are bijective since the trajectories in the domains are uniquely determined by the respective flow time and their behaviour in $\ti U(p)$.
To see that the inverses are continuous in the Hausdorff distance, we express the image of the trajectory associated to $(T^-,\tau,x,y)$ as $\im\ul\g_{\tau,x,y} \cup \Psi([T^-,0],\phi_p(x,\tau y))$, and similarly for the second map, and quote continuity of \eqref{Mp hom}, $T^\pm$ and the flow.
\end{proof}

Finally, we compare the charts for the local trajectory spaces $\leftexp{\pm}{\phantom\cM}\hspace{-4.5mm}\bM_p$ and $\leftexp{\pm}{\phantom\cM}\hspace{-4.5mm}\widetilde\cM_p$.
They differ only in the transition times, which we moreover compare with the rescaled length of time interval from Lemma~\ref{lem:tiphi} for trajectories in $\ti\cV(p)$ entirely contained in $\ti U(p)$.

\begin{lem} \label{lem:trans smooth}
The transition times $\tau_p$ and $\leftexp{\pm}{\tau}_p$ defined in \eqref{trans} and \eqref{+-trans} are continuous, and smooth when restricted to $\cM(\cU_-,\cU_+)$.
On the overlap of domains $\ev_-^{-1}\bigl(\Psi_{\R_-}(\ti{S}^+_p)\cap \ti U(p) \bigr)$ resp.\ $\ev_+^{-1}\bigl(\Psi_{\R_+}(\ti{S}^-_p)\cap \ti U(p) \bigr)$ they are related by
$$
\leftexp{-}{\tau}_p(\ul\g) = e^{T_-(\ul\g)}\cdot \tau_p(\ul\g)  = \tfrac{\Delta \cdot \tau_p(\ul\g) }{ |{\rm pr}_{\ti B^+_p} \ev_-(\ul\g)|}, \qquad
\leftexp{+}{\tau}_p(\ul\g) = e^{-T_+(\ul\g)}\cdot \tau_p(\ul\g)  = \tfrac{\Delta \cdot \tau_p(\ul\g) }{ |{\rm pr}_{\ti B^-_p} \ev_+(\ul\g)|} .
$$
The rescaled length $\ti\tau_p$ from \eqref{titip} is related to $\leftexp{\pm}{\tau}_p$ on the overlap of domains  $\ev_+^{-1}(\Psi_{\R_+}(\ti{S}^-_p))\subset\ti\cV(p)$ resp.\ $\ev_-^{-1}(\Psi_{\R_-}(\ti{S}^+_p) ) \subset \ti\cV(p)$ by
$$
\ti\tau_p(\ul\g) = \tfrac{\Delta \cdot\leftexp{-}{\tau}_p(\ul\g) }{ |{\rm pr}_{\ti B^-_p} \ev_+(\ul\g)|}, \qquad
\ti\tau_p(\ul\g) = \tfrac{\Delta \cdot \leftexp{+}{\tau}_p(\ul\g) }{ |{\rm pr}_{\ti B^+_p} \ev_-(\ul\g)|} .
$$
\end{lem}
\begin{proof}
Both continuity and smoothness in the interior follow from the corresponding regularity of the evaluation maps, see Lemma~\ref{lem:cont eval}, and the maps $E_p, \leftexp{\pm}{E}_p$ which are smooth on the complement of $(x'=0,y'=0)$, corresponding to the broken trajectories.
The relations on the overlaps follow from the definitions and the explicit form of the flow on $\ti U(p)$.
\end{proof}

\subsection{Connecting trajectory spaces and fibered products} \label{sec:connections}

For pairs of critical points $p_-,p_+\in{\rm Crit}(f)$ with $\cM(p_-,p_+)\neq\emptyset$ we construct the connecting trajectory space
$$
\cM(\ti{S}^-_{p_-},\ti{S}^+_{p_+}) := (\ev_- \times \ev_+)^{-1}\bigl(\ti{S}^-_{p_-}\times \ti{S}^+_{p_+}\bigr)  \;\subset\; \cM(X,X)
$$
as space of unbroken flow lines between the exit set $\ti{S}^-_{p_-}$ and the entry set $\ti{S}^+_{p_+}$.
The embedding $\ev_- \times \ev_+$ identifies it with the graph of the flow,
\begin{align} \label{Gr}
\Gr^{p_-}_{p_+} &:= {\rm graph}(G^{p_-}_{p_+}) \subset \ti{S}^-_{p_-}\times \ti{S}^+_{p_+}  , \\
G^{p_-}_{p_+} &: \ti{S}^-_{p_-}\cap \Psi_{\R_-}(U(p_+))  \to  \ti{S}^+_{p_+} ,\quad  z\mapsto \ev_{\ti{S}^+_{p_+}}(\Psi(\cdot,z) ) . \nonumber
\end{align}
These are indeed graphs of smooth maps defined on open subsets of $\ti{S}^-_{p_-}$, since the entry sets $\ti{S}^+_{p_+}$ are local slices to the flow. This gives $\cM(\ti{S}^-_{p_-},\ti{S}^+_{p_+})$ the structure of a smooth manifold.
Moreover, we have a continuous restriction map to the connecting trajectory space
\begin{align}  \label{rest4}
\rho^{p_-}_{p_+}:=(\ev_-\times \ev_+)^{-1}\circ  (\ev_{\ti{S}^+_{p_-}}\times \ev_{\ti{S}^-_{p_+}}) :\;
\bM(\cU_-,\cU_+; U(p_-), U(p_+)) \;\longrightarrow\; \cM(\ti{S}^-_{p_-},\ti{S}^+_{p_+}) .
\end{align}
In the special cases $\cU_-=p_-$ resp.\ $\cU_+=p_+$ the same restriction map takes values in the
subspaces $\cM(S^-_{p_-},\ti{S}^+_{p_+}):=\ev_-^{-1}(S^-_{p_-})$ resp.\ $\cM(\ti S^-_{p_-},S^+_{p_+}):=\ev_+^{-1}(S^+_{p_+})$ of $\cM(\ti{S}^-_{p_-},\ti{S}^+_{p_+})$, which are identified by the evaluations $\ev_\pm$ with intersections of the unstable resp.\ stable manifold with the opposing entry resp.\ exit set,
\begin{align}  \label{rest5}
\bM(p_-,\cU_+; U(p_+)) \overset{\rho^{p_-}_{p_+}}{\longrightarrow} \cM(S^-_{p_-},\ti{S}^+_{p_+}) \;\overset{\ev_+}{\underset{\sim}{\longrightarrow}}\; \leftexp{-}\Gr^{p_-}_{p_+} := W^-_{p_-}\cap \ti S^+_{p_+}  , \\
\bM(\cU_-,p_+; U(p_-)) \overset{\rho^{p_-}_{p_+}}\longrightarrow \cM(\ti S^-_{p_-},S^+_{p_+}) \;\overset{\ev_-}{\underset{\sim}{\longrightarrow}}\;  \leftexp{+}\Gr^{p_-}_{p_+}
:= W^+_{p_+}\cap \ti S^-_{p_-}  . \nonumber
\end{align}

We can now give an outline of how the restriction maps \eqref{rest1}, \eqref{rest4}, \eqref{rest5} are employed to construct the smooth structure and global charts for the compactified spaces of Morse trajectories between critical points $\cU_\pm=q_\pm$.

For any critical point sequence $\ul{q}=(q_1,\ldots,q_k)\in{\rm Critseq}(f;q_-,q_+)$ the open set $\cV_t(\ul{q})$ of trajectories intersecting all $\ti U_t(q_i)$ supports restriction maps to the connecting trajectory spaces $\cM(S^-_{q_-},\ti{S}^+_{q_1})$, $\cM(\ti{S}^-_{q_i},\ti{S}^+_{q_{i+1}})$ for $i=1,\ldots,k-1$, and $\cM(\ti{S}^-_{q_k},S^+_{q_+})$ as well as the restriction maps \eqref{rest1} to the local trajectory spaces $\bM_{q_i,t}$ for $i=1,\ldots,k$. Now trajectories in $\cV_t(\ul{q})$ are exactly given by tuples of trajectories in all these spaces that fit together on the entry and exit sets. Thus we have identified $\cV_t(\ul{q})$ with the fibered product
\begin{align*}
\bigl( \cM(S^-_{q_-},\ti{S}^+_{q_1}) \times {\textstyle\prod_i} \cM(\ti{S}^-_{q_i},\ti{S}^+_{q_{i+1}}) \times \cM(\ti{S}^-_{q_k},S^+_{q_+}) \bigr) \leftsub{\rm Ev}{\times}_{\rm Ev'}
\left( \bM_{q_1,t} \times  \bM_{q_2,t} \times \ldots \times  \bM_{q_k,t} \right),
\end{align*}
where both products of evaluations ${\rm Ev}= \ev_+ \times \prod_{i=1}^{k-1} ( \ev_- \times \ev_+) \times \ev_-$ and ${\rm Ev'}=\prod_{i=1}^{k} ( \ev_- \times \ev_+)$ map to $\prod_{i=1}^k \ti{S}^+_{q_i} \times \ti{S}^-_{q_i}$.
Here the product of connecting trajectory spaces on the left carries a natural smooth structure without boundary, induced by evaluation at a local slice to the flow from the smooth structure of $X$.
The product of local trajectory spaces on the right was equipped above with a natural smooth structure with boundary and corners, induced by evaluations at local slices and a projection to unstable and stable sphere and a time parameter. Once we have proven transversality of the evaluation maps (reformulated in Remark~\ref{rmk:smooth eval}), this induces a smooth structure on $\cV_t(\ul{q})$, with the corner strata determined by the transition times in the local trajectory spaces.

\begin{rmk}
The smooth structure on the local trajectory spaces depends on the choice of a homeomorphism $\rho:(0,\infty)\cup\{\infty\} \overset{\sim}{\to} [0,1)$. In the polyfold setup of \cite{hwz2}, this is known as the choice of a  gluing profile. We fixed the exponential gluing profile $\rho_e(T)=e^{-T}$ that naturally arises from the evaluation maps as $\Delta^{-1}|{\rm pr}_{\ti B^\pm_q}(\ev_\mp(\ul\g)|$, and thus ensures smoothness of the evaluation maps $\ev_\pm:\bM_{q,t}\to X$. Any other choice of homeomorphism $\r$ would yield a diffeomorphic smooth structure on each $\bM_{q,t}$. The induced smooth structures on $\bM(\cU_-,\cU_+)$ may also be diffeomorphic, if the diffeomorphisms $\r_e^{-1}\circ\r$ on the gluing parameters can be extended to a homeomorphism of $\bM(\cU_-,\cU_+)$ with the help of the associative gluing maps.
However, the regularity of the evaluation map is given by the regularity of the function $\r_e\circ\r^{-1}:[0,1)\to[0,1)$, which differentiates between some of these smooth structures.
We are constructing a smooth structure that not only does not depend on any abstract choices, but also uses the geometrically natural choice of gluing profile.
\end{rmk}

In order to construct the global chart $\cV_t(\ul{q}) \to [0,t)^k\times \cM(q_-,q_1)\times\ldots\times \cM(q_k,q_+)$ we will read off the transition times from the local trajectory spaces and project each connecting trajectory space $\cM(\ti{S}^-_{q_i},\ti{S}^+_{q_{i+1}})$ to the corresponding Morse trajectories between the critical points $\cM(q_i,q_{i+1})$, which are embedded into the former by restrictions.
To make this precise we need to show that the local trajectories for fixed transition times $\ul\tau\in[0,t)^k$ intersect each fiber of the product of these projections transversely in a unique point.
For that purpose we will iteratively construct the projections as tubular neighbourhoods of the embedding
$$
\cM(q_-,q_1)\times {\textstyle \prod_{i=1}^{k-1}} \cM(q_i,q_{i+1})\times \cM(q_k,q_+) \;\hookrightarrow\;
 \leftexp{-}\Gr^{q_-}_{q_1} \times {\textstyle\prod_{i=1}^{k-1}} \Gr^{q_i}_{q_{i+1}} \times  \leftexp{+}\Gr^{q_k}_{q_+}
$$
given by $\ev_{\ti S^-_{q_1}} \times \prod_{i=1}^{k-1} \bigl(\ev_{\ti S^+_{q_i}} \times \ev_{\ti S^-_{q_{i+1}}}\bigr) \times \ev_{\ti S^+_{q_k}}$.
The construction of these tubular neighbourhoods will iteratively proceed by pulling back previously defined charts for $\cM(\ti U(p_-),\ti U(p_+))$ near broken trajectories to $\Gr^{p_-}_{p_+}$, where the charts induce tubular neighbourhood submersions, which then just need to be extended to a compact set. In fact, this is enforced by the associativity. Thus the construction of associative gluing maps for standard Morse trajectory spaces naturally uses Morse trajectory spaces with finite ends.

For the Morse trajectory spaces with finite ends, we will use a similar fibered product setup, making use of the restriction maps \eqref{rest2} and \eqref{rest3} and the following connecting trajectory spaces for pairs of one finite end condition $\cP_\mp=X$ and a critical point $p_\pm\in{\rm Crit}(f)$,
\begin{align*}
\cM(X,\ti{S}^+_{p_+}) :=\ev_+^{-1}\bigl(\ti{S}^+_{p_+}\bigr)  \;\subset\; \cM(X,X) , \qquad
\cM(\ti{S}^-_{p_-},X) := \ev_-^{-1}\bigl(\ti{S}^-_{p_-})  \;\subset\; \cM(X,X) .
\end{align*}
The evaluations $\ev_-$ resp.\ $\ev_+$ identify them with the set of initial points $\Psi_{\R_-}(\ti{S}^+_{p_+})$ resp.\ end points $\Psi_{\R_+}(\ti{S}^-_{p_-})$.
However, the initial conditions for our global charts will also allow for trajectories with initial resp.\ end point in $\ti U_1(p_\pm)$ that do not intersect the entry resp.\ exit set. For e.g.\ initial point in $\ti U_1(p_+)\setminus \ti B^-_{p_+}$ we can extend the trajectory backwards in time to obtain an intersection with $\ti{S}^+_{p_+}$; however this definition does not extend to initial points on the unstable manifold $\ti B^-_{p_+}$. As a consequence, we lack a complete identification with a space of connecting trajectories for the relevant sets of initial resp.\ end points
\begin{equation}\label{GrX}
\leftexp{-}{\Gr}^X_{p_+} := \Psi_{\R_-}(\ti U_1(p_+))  , \qquad
\leftexp{+}{\Gr}^{p_-}_X := \Psi_{\R_+}(\ti U_1(p_-)).
\end{equation}
We do however have continuous restriction maps to the connecting trajectory spaces
\begin{align}  \label{rest6}
\ev_-^{-1} \circ \ev_- : \;
\bM(X,\cU_+; U(p_+)) \supset \ev_-^{-1}\bigl( \Psi_{\R_-}(\ti{S}^+_{p_+}) \bigr)
&\;\longrightarrow\; \cM(X,\ti{S}^+_{p_+})  ,  \\
\ev_+^{-1} \circ \ev_+ : \;
\bM(\cU_-, X; U(p_-)) \supset \ev_+^{-1}\bigl( \Psi_{\R_+}(\ti{S}^-_{p_-}) \bigr)
&\;\longrightarrow\;  \cM(\ti{S}^-_{p_-},X) . \nonumber
\end{align}

\section{Global charts for Morse trajectory spaces}
\label{global}

This section constructs the global charts of Theorem~\ref{thm:global with ends}, following the outline in Section~\ref{sec:connections}, and thus providing associative gluing maps by Corollary~\ref{cor:asso}, and equipping the Morse trajectory spaces with a smooth corner structure, finishing the proof of Theorem~\ref{thm corner}.

\subsection{Domains and targets}

Recall that we restrict ourselves to the Morse trajectory spaces $\bM(\cU_-,\cU_+)$ with free endpoint(s) $\cU_\pm= X$ or limits at critical points $\cU_\pm=q_\pm\in{\rm Crit}(f)$.
We then need to construct global charts for all critical point sequences
\[
{\rm Critseq}(f,\cU_-,\cU_+) := \left\{ (q_1,\ldots, q_k)
\left|
\begin{array}{l}
 k\in\N_0 ;  q_1,\ldots, q_k \in {\rm Crit}(f) ;  \\
 \cM(\cU_-,q_1), \cM(q_1,q_2),\ldots, \cM(q_k,\cU_+)\neq\emptyset
 \end{array}
 \right.\right\}
\]
and end conditions $\cQ_0\subset\cU_-$, $\cQ_{k+1}\subset\cU_+$ as in \eqref{choice}.
Recall here that the end condition $\cQ_0$ is either $q_-$ in case $\cU_-=q_-$ or one of $X\setminus \overline{U(q_1)}$ or $\tilde U(q_1)$ in case $\cU_-=X$, and analogously $\cQ_{k+1}$ is either $q_+$ or one of $X\setminus \overline{U(q_k)}$ or $\tilde U(q_k)$.
For unified notation we will also denote the tuple of end conditions and critical point sequence by
$\ul{q}=(\cQ_0=q_0,q_1,\ldots, q_k,q_{k+1}=\cQ_{k+1})$ and write $q_0=\cQ_0$ resp.\ $q_{k+1}=\cQ_{k+1}$ in case these are critical points rather than open sets.
The domain of the global chart for $\ul{q}$ will be the set of all trajectories starting in $\cQ_0$, ending in $\cQ_{k+1}$, intersecting each of the neighbourhoods $\ti U_{t}(q_1), \ldots,\ti U_{t}(q_k)$, and not touching any other critical point.
More precisely, we define the domains for $t>0$ as
\[
\cV_t(\ul{q}) := \bigl\{ \ul{\gamma}\in \bM(\cU_-,\cU_+)  \,\big|\,
\ev_-(\ul{\g})\in\cQ_0, \ev_+(\ul{\g})\in\cQ_{k+1}, \im\ul{\gamma}\subset X^*, \im\ul{\gamma}\cap \ti U_{t}(q_i) \neq\emptyset \;\;\forall i
\bigr\}
\]
with $X^*:= (X\setminus {\rm Crit}(f)) \cup \{q_-,q_1,\ldots,q_k,q_+\}$, where we only add $q_\pm$ in case $\cU_\pm=q_\pm$.

\begin{rmk} \label{rmk:Vt}
The domains $\cV_t(\ul{q}) \subset \bM(\cU_-,\cU_+)$ are open subsets by Lemma~\ref{lem intersecting} since they are defined by open sets $\ti U_t(q_i)$ and $X^*$.
The inclusions $\cV_t(\ul{q})\subset \cV_{t'}(\ul{q})$ for $t<t'$ are precompact up to breaking, that is the closure of $\cV_t(\ul{q})_0 \subset \cM(\cU_-,\cU_+)$ is contained in $\cV_{t'}(\ul{q})_0$.
Indeed, this follows from the precompact inclusion $\ti U_t(p)\sqsubset \ti U_{t'}(p)$.

Moreover, by Remark~\ref{rmk:U small} the domains $\cV_t(\ul{q})$ for $t>0$ sufficiently small are nonempty iff the subspace of maximally broken trajectories is nonempty, i.e.\
$$
\cV_t(\ul{q})\neq \emptyset \qquad\Longleftrightarrow \qquad
\cM(\ul{q}) := \; \cV_t(\ul{q})_k =  \cM(\cQ_0,q_1)\times \cM(q_1,q_2) \ldots \times \cM(q_k,\cQ_{k+1}) \; \neq \emptyset.
$$
This also coincides with the definition of critical point sequences $\ul{q}\in{\rm Critseq}(f;\cU_-,\cU_+)$ unless $q_1$ is a local minimum or $q_{k+1}$ is a local maximum. In the latter case we have critical point sequences $(q_1,\ldots)\in{\rm Critseq}(f;X,\cU_+)$ resp.\ $(\ldots,q_k)\in{\rm Critseq}(f;\cU_-,X)$ and nonempty domains $\cV_t(\ti U(q_1),\ldots)$ resp.\ $\cV_t(\ldots,\ti U(q_k))$ (these contain e.g.\ broken trajectories starting at $q_1$ resp.\ ending at $q_k$, corresponding to $\cM(\ti U(q_1),q_1)\simeq\{q_1\}$ resp.\ $\cM(q_k,\ti U(q_k))\simeq\{q_k\}$), but the domains for $\cQ_0=X\setminus\overline{U(q_1)}$ resp.\ $\cQ_{k+1}=X\setminus\overline{U(q_k)}$ are empty, corresponding to $\cM(X\setminus\overline{U(q_1)},q_1)=\emptyset$ resp.\
$\cM(q_k,X\setminus\overline{U(q_k)})=\emptyset$.
\end{rmk}

We will prove Theorem~\ref{thm:global with ends} by constructing for every such tuple $\ul{q}$ a homeomorphism
\begin{equation}\label{eq:phi q}
\phi(\ul{q}) : \;
\cV_t(\ul{q}) \;\overset{\sim}{\longrightarrow}\;  {\textstyle\bigcup_{\ul\tau\in I_t(\ul{q})}} \{\ul{\tau}\} \times \cM_{t,\ul\tau}(\ul{q}) \;\subset\; [0,1+t)^k \times \cM(\ul{q})
\end{equation}
to the open subset given by
\[
\cM_{t,\ul{\tau}}(\ul{q}) := \left\{ (\g_0 ,\ldots,\g_k) \in \cM(\ul{q}) \left|
\begin{smallmatrix}
\tau_1 |\ev_-(\g_0)|  < t \Delta \;\; \text{in case}\;\; \cQ_0=\ti U(q_1) , \;\;\\
\tau_k |\ev_+(\g_k)| < t \Delta \;\; \text{in case}\;\; \cQ_{k+1}=\ti U(q_k)  \\
\end{smallmatrix}
\right.\right\}
\]
and
\[
I_t(\ul{q}) := \left.\begin{cases}
[0,1+t) &; \cQ_0=\ti U(q_1) \\
[0,t) &;  \text{otherwise}
\end{cases}\right\}
\times [0,1)^{k-2}\times
\left.\begin{cases}
[0,1+t) &; \cQ_{k+1}=\ti U(q_k) \\
[0,t) &;  \text{otherwise}
\end{cases}\right\}
\]
except in the special case $\ul{q}_1=(\ti U(q_1), q_1, \ti U(q_1))$, when $\phi(\ul{q}_1)$ will be defined as in Remark~\ref{rmk:local traj} with image in an open subset of $[0,1] \times \cM(\ul{q})$ given by
\begin{align*}
I_t(\ul{q}_1)&:=[0,1] , \\
\cM_{t,\tau_1}(\ul{q}_1) &:=
\bigl\{ (\g_0,\g_1) \st
\tau_1 |\ev_-(\g_0)| , \tau_1 |\ev_+(\g_1)| <(1+t)\Delta,
\tau_1 |\ev_-(\g_0)| |\ev_+(\g_1)| <t\Delta^2
\bigr\} .
\end{align*}
In case $k=0$ with $\ul{q}=(\cQ_0,\cQ_1)$ we interpret $\bigcup_{\ul\tau\in I_t(\ul{q})} \cM_{t,\ul\tau}(\ul{q}) = \cM(\cQ_0,\cQ_1)$.
Let us also summarize the further properties required by Theorem~\ref{thm:global with ends} of the homeomorphisms \eqref{eq:phi q}.

\begin{enumerate}
\item
The restriction $\phi(\ul{q})|_{\cV_t(\ul{q})_0}$ is a diffeomorphism
$\cV_t(\ul{q})_0  \longrightarrow \bigcup_{\ul\tau\in I_t(\ul{q})\cap (0,\infty)^k} \cM_{t,\ul\tau}(\ul{q})$.
\item
The restriction $\phi(\ul{q})|_{\cV_t(\ul{q})_k}$ is the canonical bijection
$\cV_t(\ul{q})_k  \to \{0\}^k \times \cM(\ul{q})$.
\item
Let $\ul{q},\ul{Q}$ be tuples such that
$\ul{Q}=(\cQ_0,\ldots, q_i,q'_1,\ldots,q'_\ell, q_{i+1}, \ldots, \cQ_{k+1})$ is obtained from $\ul{q}=(\cQ_0,\ldots, q_i, q_{i+1}, \ldots, \cQ_{k+1})$ by inserting a nontrivial critical point sequence $(q'_1,\ldots,q'_\ell)$.
Then we have
$\phi(\ul{Q})= \bigl({\rm Id}\times \phi(\ul{q}') \times {\rm Id} \bigr) \circ \phi(\ul{q})$ on $\cV_t(\ul{q}) \cap \cV_t(\ul{Q})$ with
$\ul{q}'=\bigl( \bigl\{ \begin{smallmatrix}  \cQ_0 &; i=0 \\ q_i &; i>1 \end{smallmatrix}\bigr\}, q'_1,\ldots,q'_\ell, \bigl\{ \begin{smallmatrix}  \cQ_{k+1} &; i=k\\ q_{i+1} &; i<k \end{smallmatrix}\bigr\})$.
\item
The real parameters, transition maps between different end conditions for $\cU_\pm=X$, and the form of charts for $\cQ_0=\ti U(q_1)$ resp.\ $\cQ_{k+1}=\ti U(q_1)$ are given explicitly.
\end{enumerate}

We will construct global charts $\phi(\ul{q})$ with these properties iteratively.
Before going into the general construction we take note of two special cases that are already constructed.

\subsection{Construction of global chart for $k=0$} \label{sec:k=0}

The open sets associated to the shortest critical point sequences with $k=0$,
\begin{align*}
\cV_t((\cQ_0,\cQ_1)) &= \bigl\{ \ul{\g} \in \bM(\cU_-,\cU_+) \st
\ev_-(\ul{\g})\in\cQ_0, \ev_+(\ul{\g})\in\cQ_1, \im(\ul{\g})\subset X^* \bigr\} \\
&= (\ev_-\times\ev_+)^{-1}(\cQ_0\times\cQ_1)\;\subset\; \cM(\cU_-,\cU_+),
\end{align*}
are the subsets of unbroken flow lines with the given end conditions, and by (ii) with $k=0$ these homeomorphisms are set to be the identities $\phi((\cQ_0,\cQ_1))={\rm Id}_{\cM(\cU_-,\cU_+)}|_{\cV_t((\cQ_0,\cQ_1))}$. This chart also clearly satisfies (i), will trivially fit into (iii), has no real parameters to which (iv) would apply, and the transition maps for different choices of $\cQ_0$ or $\cQ_1$ are the identity. In fact, there is no need to separate $\cU_-=X$ or $\cU_+=X$ into two domains in this case.

\subsection{Construction of global chart for $k=1$ with end conditions $\cQ_0=\ti U(q_1)=\cQ_2$}
\label{sec:k=1 special}

For the special tuples $\ul{q}_1=(\ti U(q_1),q_1,\ti U(q_1))$ with end conditions near the same critical point we constructed the charts
$\phi(\ul{q}_1) := \ti\phi(q_1)|_{\cV_t(\ul{q}_1)}$ for any $t>0$ in Lemma~\ref{lem:tiphi} and Remark~\ref{rmk:local traj}. For future reference,
$$
\phi(\ul{q}_1):= \ti\tau_{q_1} \times ({\rm pr}_{\ti B^+_{q_1}}\times {\rm pr}_{\ti B^-_{q_1}}) \circ (\ev_-\times\ev_+)  \; : \;  \cV(\ul{q}_1) \;\longrightarrow\;
{\textstyle\bigcup_{\tau_1\in [0,1]}} \{\tau_1\} \times \cM_{t,\tau_1}(\ul{q}_1)
$$
is given by a transition time $\tau(\ul{q}_1):=\ti\tau_{q_1}$, evaluation ${\rm Ev}(\ul{q}_1):=\ev_-\times\ev_+$, and projection $\pi(\ul{q}_1):={\rm pr}_{\ti B^+_{q_1}}\times {\rm pr}_{\ti B^-_{q_1}}$.
These charts are completely fixed by (iv), and by construction satisfy (i) and (ii).
Note that $\cV_t(\ul{q}_1)$ has nonempty intersection with another chart $\cV_t(\ul{q})$ only for $\ul{q}=(\cQ_0,q_1,\cQ_2)$ since the trajectories are contained in $\ti U(q_1)$, which is disjoint from the neighbourhood of any other critical point. Hence this chart appears in (iii) only in the trivial identity $\phi(\ul{q}_1)= \phi(\ul{q}_1) \circ {\rm Id}_{\cM(X,X)}$ on $\cM(X,X) \cap \cV_t(\ul{q}_1)$. The transition maps in (iv) will be established in the iterative construction.

\medskip

For all other end conditions and critical point sequences $\ul{q}$ the global charts on $\cV_t(\ul{q})$ will be constructed similarly as composition of transition times and evaluations, which we introduce next, and tubular neighbourhoods of $\cM(\ul{q})$ generalizing the projections from $\ti U(q)$ to $\ti B^+_q\simeq \cM(\ti U(q),q)$ and $\ti B^-_q\simeq \cM(q,\ti U(q))$, which will be constructed iteratively.

\subsection{Evaluations and transition times}
\label{sec:ev}
For any tuple $\ul{q}$ of a critical point sequence $(q_1,\ldots, q_k)\in{\rm Critseq}(f,\cU_-,\cU_+)$ and choices of end conditions $\cQ_0\subset \cU_-$, $\cQ_{k+1}\subset\cU_+$ from \eqref{choice} we define the evaluation map
$$
{\rm Ev}(\ul{q}):= \ev_{(\cU_-,q_1)} \times \ev_{\ti S^-_{q_1}}\times \ev_{\ti S^+_{q_2}}\ldots \; \ev_{\ti S^-_{q_{k-1}}}\times \ev_{\ti{S}^+_{q_k}} \times \ev_{(q_k,\cU_+)} ,
$$
\[
\ev_{(\cU_-,q_1)} := \begin{cases}
\ev_{\ti S^+_{q_1}} &; \cU_-=q_- , \\
\ev_- &; \cU_-=X, \\
\end{cases}
\qquad
\ev_{(q_k,\cU_+)} := \begin{cases}
\ev_{\ti S^-_{q_k}} &; \cU_+=q_+ , \\
\ev_+ &; \cU_+=X.
\end{cases}
\]
This generalizes the evaluation ${\rm Ev}(X,q_1,X)=\ev_-\times\ev_+$ from Section~\ref{sec:k=1 special}. However, due to the time parameter in $[0,1]$ this special case does not quite fit into the language of the rest of this section, where we build up to showing in Proposition~\ref{prp:Ev} that for
$\ul{q}$ not covered by Section~\ref{sec:k=0} or \ref{sec:k=1 special} this evaluation defines a homeomorphism to its image
$$
{\rm Ev}(\ul{q}) : \; \cV_t(\ul{q}) \;\longrightarrow\;
\ti{S}^+_{(\cU_-,q_1)}\times \ti{S}^-_{q_1}  \times \; \ti{S}^+_{q_2} \ldots \ti{S}^-_{q_{k-1}} \times
\ti{S}^+_{q_k} \times \ti{S}^-_{(q_k,\cU_+)} \; =: \; \ti S(\ul{q})
$$
in the target space given by the entry and exit sets, with the notation
\[
\ti{S}^+_{(\cU_-,q_1)} := \begin{cases}
\ti S^+_{q_1} &; \cU_-=q_- ,\\
X &; \cU_-=X ,
\end{cases}
\qquad
\ti{S}^-_{(q_k,\cU_+)}  := \begin{cases}
\ti S^-_{q_k} &; \cU_+=q_+, \\
X &; \cU_+=X .
\end{cases}
\]
Since the evaluations of $\cV_t(\ul{q})$ are connected by flow lines between each consecutive $\ti{S}^-_{q_{i}}$ and $\ti{S}^+_{q_{i+1}}$, and the initial resp.\ end evaluation is connected by a flow line to $\cU_-$ resp.\ $\cU_+$, the image of the evaluation map is contained in the submanifold
$$
\Gr(\ul{q}):= \leftexp{-}{\Gr}^{\cU_-}_{q_1} \times\Gr^{q_1}_{q_2} \times \ldots \times \Gr^{q_{k-1}}_{q_k} \times \leftexp{+}{\Gr}^{q_k}_{\cU_+} \;\subset\; \ti S(\ul{q}) .
$$
Here ${\Gr}^{p_-}_{p_+}$ are the graphs of the flow from \eqref{Gr}, homeomorphic to the connecting spaces of trajectories from exit set $\ti{S}^-_{p_-}$ to entry set $\ti{S}^+_{p_+}$.
For critical point end conditions, $\leftexp{\pm}{\Gr}^{p_-}_{p_+}$ are the restrictions from \eqref{rest5} to trajectories starting on the unstable sphere resp.\ ending on the stable sphere.
For finite end conditions, the behaviour of the trajectories before $\ti{S}^-_{q_1}$ resp.\ after $\ti{S}^+_{q_k}$ will be encoded in the local trajectory space for $(\cQ_0,q_1)$ resp.\ $(q_k,\cQ_{k+1})$, so we
merely use the spaces $\leftexp{-}{\Gr}^X_{p_+}$ resp.\ $\leftexp{+}{\Gr}^{p_-}_X$ of possible initial resp.\ end points. To summarize,
\[
\leftexp{-}{\Gr}^{\cU_-}_{q_1} = \begin{cases}
W^-_{q_-}\cap \ti S^+_{q_1} &; \cU_-=q_- , \\
\Psi_{\R_-}(\ti U_1(q_1))  &; \cU_-=X ,
\end{cases}
\qquad
\leftexp{+}{\Gr}^{q_k}_{\cU_+} = \begin{cases}
W^+_{q_+}\cap\ti S^-_{q_k} &; \cU_+=q_+  , \\
\Psi_{\R_+}(\ti U_1(q_k))  &;  \cU_+=X .
\end{cases}
\]
On the other hand, the evaluations of trajectories in $\cV_t(\ul{q})$ are also connected by trajectories in $\bM_{q_i}$ for $i=1,\ldots,k$, except for $i=1,k$ and $\cU_\pm =X$ when we need to use the local trajectory spaces $\leftexp{\pm}{\phantom\cM}\hspace{-4.5mm}\widetilde\cM_{q_i}$ resp.\
$\leftexp{\pm}{\phantom\cM}\hspace{-4.5mm}\overline\cM_{q_i}$ depending on the end conditions $\cQ_0,\cQ_{k+1}$.
Including the intersection conditions with $\ti U_t(q_i)$, we thus describe the open set $\cV_t(\ul{q})$ as fibered product of $\Gr(\ul{q})$ and the evaluation ${\rm Ev'}=\prod_{i=1}^{k} ( \ev_- \times \ev_+)$ from the local trajectory spaces,
$$
\cV_t(\ul{q}) \; \simeq \; \Gr(\ul{q}) \times_{{\rm Ev'}} \left(
\left\{\begin{smallmatrix}
\bM_{q_1} \\
\leftexp{-}{\phantom\cM}\hspace{-3.5mm}\overline\cM_{q_i} \\
\leftexp{-}{\phantom\cM}\hspace{-3.5mm}\widetilde\cM_{q_i}
\end{smallmatrix}\right\}
\times \bM_{q_2} \times \ldots \times  \bM_{q_{k-1}}\times
\left\{\begin{smallmatrix}
\bM_{q_k} \\
\leftexp{+}{\phantom\cM}\hspace{-3.5mm}\overline\cM_{q_k} \\
\leftexp{+}{\phantom\cM}\hspace{-3.5mm}\widetilde\cM_{q_k}
\end{smallmatrix}\right\}
\right) .
$$
From Lemmas~\ref{lem:local traj}, \ref{lem:local traj +-}, and  \ref{lem:local traj last} we know that the evaluations of the local trajectories are given by the smooth family of embeddings for transition times $\ul{\tau}=(\tau_1,\ldots,\tau_k)\in I_t(\ul{q})$
\begin{align*}
\iota_{\ul{q},\ul\tau} : \; S^+_{(\cQ_0,q_1)} \times S^-_{q_1} \times \ldots  S^+_{q_k} \times S^-_{(q_k,\cQ_{k+1})} &\;\longrightarrow\;
\ti{S}^+_{(\cU_-,q_1)}\times \ti{S}^-_{q_1}  \times  \ldots  \times
\ti{S}^+_{q_k} \times \ti{S}^-_{(q_k,\cU_+)}  \;=\; \ti S(\ul{q}) \\
\left( T_-, x_1 , y_1,  \ldots, x_k, y_k , T_+ \right) \qquad
&\;\longmapsto\;
{\scriptstyle \bigl( \Psi_{T_-}(x_1 , \tau_1 y_1 ) , (\tau_1 x_1 , y_1 ) , \ldots, (x_k , \tau_k y_k ) , \Psi_{T_+} ( \tau_k x_k ,y_k )\bigr)} .
\end{align*}
Here we still identify the coordinates $\ti B^+_{q_i}\times\ti B^-_{q_i}$ with their images in $\ti U(q_i)\subset X$, and $(T_-,x_1)$ resp.\ $(x_k,T_+)$ are coordinates on $W^+_{q_1}$ resp.\ $W^-_{q_k}$, taking values in
\[
S^+_{(\cQ_0,q_1)} := \begin{cases}
\{0\}\times S^+_{q_1} &; \cQ_0=q_- , \\
\R_-\times S^+_{q_1} &;  \cQ_0=X\setminus\overline{U(q_1)} , \\
\{0\}\times\ti B^+_{q_1} &; \cQ_0=\ti U(q_1),
\end{cases}
\quad
S^-_{(q_k,\cQ_{k+1})} := \begin{cases}
S^-_{q_k}\times\{0\} &; \cQ_{k+1}=q_+ , \\
S^-_{q_k} \times \R_+&; \cQ_{k+1}=X\setminus\overline{U(q_k)} , \\
\ti B^-_{q_k} \times\{0\}&; \cQ_{k+1}=\ti U(q_k) .
\end{cases}
\]
For a complete description of $\cV_t(\ul{q})$ it remains to note that the intersection condition with $\ti U_t(q_i)$ gives rise to a restriction of the domain of $\iota_{\ul\tau}$ in the case of end conditions near $q_1$ or $q_k$. With that, and abbreviating $\iota_{\ul\tau}=\iota_{\ul q,\ul\tau}$, the evaluations of local trajectories are given by
$$
{\rm Ev'}\left(
\left\{\begin{smallmatrix}
\bM_{q_1} \\
\leftexp{-}{\phantom\cM}\hspace{-3.5mm}\overline\cM_{q_i} \\
\leftexp{-}{\phantom\cM}\hspace{-3.5mm}\widetilde\cM_{q_i}
\end{smallmatrix}\right\}
\times \bM_{q_2} \times \ldots \times  \bM_{q_{k-1}}\times
\left\{\begin{smallmatrix}
\bM_{q_k} \\
\leftexp{+}{\phantom\cM}\hspace{-3.5mm}\overline\cM_{q_k} \\
\leftexp{+}{\phantom\cM}\hspace{-3.5mm}\widetilde\cM_{q_k}
\end{smallmatrix}\right\}
\right)
=
\bigcup_{\ul{\tau}\in I_t(\ul{q}) }  \im\iota_{\ul{\tau}}(\cD_{t,\ul{\tau}}(\ul{q}))
 \;\subset\; \ti{S}(\ul{q}),
$$
\[
\cD_{t,\ul{\tau}}(\ul{q}) := \left\{ (T_-, x_1 ,\ldots,y_k, T_+) \in S^+_{(\cQ_0,q_1)} \times  \ldots  \times S^-_{(q_k,\cQ_{k+1})} \left|
\begin{smallmatrix}
\tau_1 |x_1|  < t \Delta \;\; &\text{if}\;\; \cQ_0=\ti U(q_1) , \;\;\; \\
\tau_k |y_k| < t \Delta \;\; &\text{if}\;\; \cQ_{k+1}=\ti U(q_k)
\end{smallmatrix}
\right.\right\}.
\]
Here we do not deal with the cases $k=0$ or $\cQ_0=\ti U(q_1)=\cQ_2$, for which the global charts were constructed in the previous sections. For all other critical point sequences we have achieved a complete description of the image of the evaluation homeomorphism,
\begin{equation}\label{EvVq}
{\rm Ev}(\ul{q})(\cV_t(\ul{q})) \;= {\textstyle\bigcup_{\ul{\tau}\in I_t(\ul{q})} } \iota_{\ul{\tau}}(\cD_{t,\ul{\tau}}(\ul{q})) \cap \Gr(\ul{q}) .
\end{equation}
The transition times $\ul{\tau}\in I_t(\ul{q})$ implicit in \eqref{EvVq} can be read off explicitly by the map
$$
\tau(\ul{q}):= \tau_{(\cQ_0,q_1)}\times \tau_{q_2}\times\ldots\times \tau_{q_{k-1}}
\times \tau_{(q_k,\cQ_{k+1})}  : \cV_t(\ul{q})  \;\longrightarrow\;  [0,2)^k
$$
given by the transition times from \eqref{trans} and \eqref{+-trans},
\[
 \tau_{(\cQ_0,q_1)} :=
\begin{cases}
\leftexp{-}{\tau}_{q_1} \hspace{-3mm} &; \cQ_0=\ti U(q_1) , \\
\;\;\tau_{q_1} \hspace{-3mm}&; \text{otherwise} ,
\end{cases}
\qquad
\tau_{(q_k,\cQ_{k+1})} :=
\begin{cases}
\leftexp{+}{\tau}_{q_k} \hspace{-3mm}&; \cQ_{k+1}=\ti U(q_k),\\
\;\;\tau_{q_k} \hspace{-3mm}&; \text{otherwise} .
\end{cases}
\]
Note that this does not yield a definition of $\tau(\ul{q}_1)$ in the special case $\ul{q}_1=(\ti U(q_1),q_1,\ti U(q_1))$ of Section~\ref{sec:k=1 special}. In that case we denote by
$\tau(\ul{q}_1):=\ti\tau_{q_1}$  the rescaled time for which the trajectory is defined.
Similar to the construction of the global chart in that special case, we can show in general that the transition times and evaluations provides a map that satisfies most properties of a global chart, except that it maps to a neighbourhood of the intended target.
In particular the following establishes the homeomorphism property of ${\rm Ev}(\ul{q})$.

\begin{prp} \label{prp:Ev}
For any $t>0$, critical point sequence $(q_1,\ldots, q_k)\in{\rm Critseq}(f,\cU_-,\cU_+)$, and choice
of end conditions $\cQ_0\subset \cU_-$, $\cQ_{k+1}\subset\cU_+$ from \eqref{choice} that are not covered by Section~\ref{sec:k=0} or \ref{sec:k=1 special}, the map
$$
\tau(\ul{q}) \times {\rm Ev}(\ul{q})  : \;
\cV_t(\ul{q}) \;\longrightarrow\;
{\textstyle \bigcup_{\ul{\tau}\in I_t(\ul{q})}}  \{\ul{\tau}\} \times \bigl(\iota_{\ul{\tau}}(\cD_{t,\ul{\tau}}(\ul{q})) \cap \Gr(\ul{q})\bigr)
 \;\subset\; [0,2)^k \times {\rm Gr}(\ul{q})
$$
is a homeomorphism and satisfies the following.
\begin{enumerate}
\item
Restricted to the unbroken trajectories, $\bigl(\tau(\ul{q}) \times {\rm Ev}(\ul{q}) \bigr)\big|_{\cV_t(\ul{q})_0}$ is a diffeomorphism
$$
\cV_t(\ul{q})_0  \;\longrightarrow\;
{\textstyle\bigcup_{\ul{\tau}\in I_t(\ul{q})\cap (0,\infty)^k}} \{\ul{\tau}\} \times \bigl(\iota_{\ul{\tau}}(\cD_{t,\ul{\tau}}(\ul{q})) \cap \Gr(\ul{q})\bigr)
\;\subset\; (0,2)^k \times  {\rm Gr}(\ul{q})  .
$$
\item
Restricted to the maximally broken trajectories, $\bigl(\tau(\ul{q}) \times {\rm Ev}(\ul{q}) \bigr)\big|_{\cV_t(\ul{q})_k}$ is the bijection $\cV_t(\ul{q})_k = \cM(\ul{q}) \overset{\sim}{\to} \{0\}^k \times (\im\iota_{\ul{0}} \cap  {\rm Gr}(\ul{q}))$ given by evaluating $(\gamma_0,\ldots,\gamma_k)$ to
$$
 \bigl(\, \ul{0} \,; \ev_{(\cU_-,q_1)}(\gamma_0),  \ev_{\ti{S}^-_{q_1}}(\gamma_1),  \ev_{\ti{S}^+_{q_2}}(\gamma_1), \ldots,
  \ev_{\ti{S}^+_{q_k}}(\gamma_{k-1}),  \ev_{(q_k,\cU_+)}(\gamma_k)
 \bigr) .
$$
\item
Let $\ul{Q}=(\ldots, q_i,q'_1,\ldots,q'_\ell, q_{i+1}, \ldots)$ be obtained from $\ul{q}=(\cQ_0,\ldots, q_i,q_{i+1}, \ldots,\cQ_{k+1})$ by inserting a nontrivial critical point sequence  $(q'_1,\ldots,q'_\ell)$.
Then we have
$$
\bigl(\tau(\ul{q}) \times {\rm Ev}(\ul{q}) \bigr)\big|_{\cV_t(\ul{q}) \cap \cV_t(\ul{Q})}
= F_{\ul{q},\ul{Q}} \circ \bigl(\tau(\ul{Q}) \times {\rm Ev}(\ul{Q}) \bigr)\big|_{\cV_t(\ul{q}) \cap \cV_t(\ul{Q})}
$$
with the forgetful map $F_{\ul{q},\ul{Q}} : I_t(\ul{Q})\times\Gr(\ul{Q}) \to I_t(\ul{q}) \times \ti{S}(\ul{q})$.
\item
The transition times $\tau(\ul{q})$ are given explicitly as in Theorem~\ref{thm:global with ends}.
For nontrivial critical point sequences $(q_1,\ldots,q_k)\in{\rm Critseq}(f,\cU_-,\cU_+)$ and switching end conditions from $\cQ_0=\ti U(q_1)$ to $\cQ_0=X\setminus\overline{U(q_1)}$ resp.\ from $\cQ_{k+1}=\ti U(q_k)$ to $\cQ_{k+1}=X\setminus\overline{U(q_k)}$ the homeomorphisms $\tau\times{\rm Ev}$
have overlap of domains $\ev_-^{-1}(\Psi_{\R_-}(\ti S^+_{q_1}) \cap \ti U(q_1))$ resp.\ $\ev_+^{-1}(\Psi_{\R_+}(\ti S^-_{q_k}) \cap \ti U(q_k))$ and are related by
\begin{align*}
(E_1,\ldots\phantom{,E_k}; z^+_1,\ldots\phantom{,z^-_k}) &\mapsto \bigl( |{\rm pr}_{W^+_{q_1}}(z^+_1)|\Delta^{-1} E_1 ,\ldots \phantom{, |{\rm pr}_{W^-_{q_k}}(z^-_k)|\Delta^{-1} E_k} ; z^+_1,\ldots \phantom{, z^-_k} \bigr) , \\
(\phantom{E_1,}\ldots,E_k;\phantom{z^+_1,} \ldots, z^-_k) &\mapsto \bigl( \phantom{|{\rm pr}_{W^+_{q_1}}(z^+_1)|\Delta^{-1} E_1 ,}\ldots, |{\rm pr}_{W^-_{q_k}}(z^-_k)|\Delta^{-1} E_k ; \phantom{z^+_1,} \ldots, z^-_k \bigr) .
\end{align*}
This last part includes the special case $\cQ_0=\cQ_2=\ti U(q_1)$ with $\tau(\ul{q})$ from
\eqref{titip}.
\end{enumerate}
\end{prp}

\begin{proof}
We will give the proof for $t=1$, then the general case follows by restriction to $\cV_t(\ul{q})\subset\cV_1(\ul{q})$.
The product of the evaluation maps ${\rm Ev}(\ul{q})= \ev_{(\cU_-,q_1)} \times \ev_{\ti S^-_{q_1}}\times\ldots \times \ev_{\ti{S}^+_{q_k}} \times \ev_{(q_k,\cU_+)}$ is injective since the value of
$\ev_{\ti S^+_{q_i}}\times \ev_{\ti{S}^-_{q_i}}$ (resp.\ $\ev_-\times \ev_{\ti{S}^-_{q_1}}$ in some cases of $i=1$, resp.\ $\ev_{\ti S^+_{q_k}}\times \ev_+$ in some cases of $i=k$) determines the behaviour near the critical point $q_i$, a generalized trajectory in $\cV_t(\ul{q})$ does not break at any other critical point, and the behaviour near all critical points, together with initial and end point, determine the entire trajectory.
Moreover, ${\rm Ev}(\ul{q})$ is a product of continuous maps by Lemma~\ref{lem:cont eval}.
In fact, when restricted to $\cV_t(\ul{q})_0$, then ${\rm Ev}(\ul{q})$ is a product of smooth embeddings, again by Lemma~\ref{lem:cont eval}.
The transition times are continous by Lemma~\ref{lem:trans smooth} and smooth when restricted to $\cV_t(\ul{q})_0$. This shows that $\tau\times{\rm Ev}$ is a continuous injection and $(\tau\times{\rm Ev})|_{\cV_t(\ul{q})_0}$ is an embedding into $(0,2)^k \times  {\rm Gr}(\ul{q})$. This proves (i) up to determining the image $(\tau\times{\rm Ev})(\cV_t(\ul{q}))$, since then the unbroken trajectories in $\cV_t(\ul{q})_0$ are exactly those with no breaking, i.e.\ with rescaled transition times in $(0,2)$.

The characterization of the image was given in Section~\ref{sec:ev}, based on the fact that the trajectories in $\cV_t(\ul{q})$ can be uniquely described by their behaviour near each critical point $q_1,\ldots,q_k$, including the initial or end point in case $\cU_-=X$ or $\cU_+=X$. On the other hand, a tuple of local trajectories near $q_1,\ldots,q_k$ fits together to a trajectory in $\cV_t(\ul{q})$ if and only if they satisfy the matching conditions encoded in $\Gr(\ul{q})$.

Properties (ii), (iii), and (iv) follow directly from the definition of the maps, so it remains to prove continuity of $(\tau(\ul{q}) \times {\rm Ev}(\ul{q}))^{-1}$. For that purpose it suffices to show that the map
$$
R(\ul{q}) := \bigl( {\rm Id}_{[0,2)^k} \times (\iota_{\ul\tau}^{-1})_{\ul\tau\in [0,2)^k} \bigr) \circ (\tau(\ul{q}) \times {\rm Ev}(\ul{q}))  :
\cV_t(\ul{q}) \to
{\textstyle\bigcup_{\ul{\tau}\in [0,2)^k}}  \{\ul{\tau}\} \times \bigl(\cD_{t,\ul{\tau}}(\ul{q}) \cap \iota_{\ul\tau}^{-1}(\Gr(\ul{q}))\bigr)
$$
has a continuous inverse. We will do this explicitly for the case of trajectories between critical points $\cU_\pm=q_\pm$. The case of finite end conditions $\cU_-=X$ or $\cU_+=X$ is completely analogous,
after replacing the spheres $S^+_{q_1}$ resp.\ $S^-_{q_k}$ with either a ball $\ti B^+_{q_1}$ resp.\ $\ti B^-_{q_k}$ or adding a flow time parameter in $\R_-$ resp.\ $\R_+$.
To prove continuity for $\cU_\pm=q_\pm$ first recall from Lemma~\ref{lem:local traj} that we have continuous local inverse maps
$$
R_{q_i}^{-1} : \; [0,1) \times S^+_{q_i} \times S^-_{q_i}\to\bM(X,X), \quad
(\tau_i,x_i,y_i) \mapsto \ul{\gamma}_{\tau_i,x_i,y_i} .
$$
Their images lie in the neighbourhoods $\overline{U(q_i)}$ and the matching conditions of $\iota_{\ul\tau}^{-1}(\Gr(\ul{q}))$ can be rephrased as
$\ev_-(\ul{\gamma}_{\tau_1,x_1,y_1} ) \in W^-_{q_-}$, $\ev_+(\ul{\gamma}_{\tau_k,x_k,y_k})\in W^+_{q_+}$, and
$$
\ev_-(\ul{\gamma}_{\tau_i,x_i,y_i}) \in \Psi_{\R_+}\bigl(\ev_+(\ul{\gamma}_{\tau_{i-1},x_{i-1},y_{i-1}})\bigr) \quad \forall i=2,\ldots,k .
$$
Hence the image of the full trajectory $\ul{\gamma}_{\ul{(\tau,x,y)}}:=R(\ul{q})^{-1}((\tau_i,x_i,y_i)_{i=1,\ldots,k})$ is given by the local trajectories and Morse flow lines between,
$$
\im \ul{\gamma}_{\ul{(\tau,x,y)}} =
{\textstyle\bigcup_{i=1}^{k}} \im \ul{\gamma}_{\tau_i,x_i,y_i}
\;\cup\; {\textstyle\bigcup_{i=1}^{k}} \Psi_{\R_+}(\tau_i x_i,y_i)
\;\cup\;
\Psi_{\R_-}(x_1,\tau_1 y_1) .
$$
In fact, for $i=2,\ldots, k$ we can replace $\Psi_{\R_+}(\tau_i x_i,y_i)$ by the finite flow line
$\Psi_{[0,T_i]}(\tau_i x_i,y_i)$, where $T_i>0$ is determined by
$\Psi_{T_i}(\tau_i x_i,y_i)= (x_{i-1},\tau_{i-1}y_{i-1})$.
We can now fix a neighbourhood $U\subset \bM_{q_1,t}\times \ldots \bM_{q_k,t}$ of $\ul{(\tau,x,y)}$ such that for every $\ul{(\tau',x',y')} = (\tau'_i,x'_i,y'_i)_{i=1,\ldots,k}\in U$ the corresponding flow times $T'_i>0$ satisfy $T'_i\leq 2 T_i$.
With that we can express the new image as similar union
$$
\im \ul{\gamma}_{\ul{(\tau',x',y')}} =
{\textstyle\bigcup_{i=1}^{k}} \im \ul{\gamma}_{\tau'_i,x'_i,y'_i}
\cup {\textstyle\bigcup_{i=2}^{k}} \Psi_{[0,2T_i]}(\tau'_i x'_i,y'_i)
\cup \Psi_{\R_-}(x'_1,\tau'_1y'_1)
\cup \Psi_{\R_+}(\tau'_k x'_k,y'_k) .
$$
Now, given $\ep>0$, we need to choose the neighbourhood $U$ so small that the Hausdorff distance between the images of trajectories is small,
$d_{\rm H}(\im \ul{\gamma}_{\ul{(\tau',x',y')}}, \im \ul{\gamma}_{\ul{(\tau,x,y)}} ) \leq \ep$.
(Note that adding finitely many critical points for the closure of the image will not change the Hausdorff distance.)
By the additivity property $d_{\rm H}(A_1\cup A_2 , B_1 \cup B_2 ) \leq \max\{ d_{\rm H}(A_1 , B_1 ) , d_{\rm H}( A_2 , B_2 ) \}$ it suffices to check that the corresponding local trajectories and flow lines are nearby.
Firstly, from the continuity of $R_{q_i}^{-1}$ we have
$d_{\rm H}(\im \ul{\gamma}_{\tau'_i,x'_i,y'_i}, \im \ul{\gamma}_{\tau_i,x_i,y_i})\leq \ep$
for sufficiently small $U$.
Secondly, continuity of the Morse flow $\Psi$ provides
$d_{\rm H}(\Psi_{[0,2T_i]}(\tau'_i x'_i,y'_i), \Psi_{[0,2T_i]}(\tau_i x_i,y_i) )\leq \ep$.
Finally, for the convergence to $q_-$ we can fix $T_- >0$ and choose $U$ such that
$\Psi_{-T}(x'_1,\tau'_1 y'_1)\in B_\ep(q_-)$ for all $\ul{(\tau',x',y')} \in U$ and $T\geq T_-$.
Then we obtain
$$
d_{\rm H}(\Psi_{\R_-}(x'_1,\tau'_1y'_1), \Psi_{\R_-}(x_1,\tau_1y_1) )
\leq
\max\bigl\{ \ep , d_{\rm H}(\Psi_{[-T_-,0]}(x'_1,\tau'_1y'_1), \Psi_{[-T_-,0]}(x_1,\tau_1y_1) ) \bigr\},
$$
which by continuity of the flow $\Psi$ will be bounded by $\ep$ for small $U$.
A similar argument ensures
$d_{\rm H}(\Psi_{\R_+}(\tau'_k x'_k,y'_k), \Psi_{\R_+}(\tau_k x_k,y_k) )
\leq \ep $ and finishes the proof.
\end{proof}

\subsection{Construction of general global chart} \label{sec:chart outline}

To obtain a smooth structure for $\cV_t(\ul{q})$ from Proposition~\ref{prp:Ev} note that for $\ul{\tau}=\ul{0}$ the embedding $\iota_{\ul{0}}$ intersects $\Gr(\ul{q})$ transversely. Indeed, $\im\iota_{\ul{0}}$ is the product of stable and unstable spheres $\im\iota_{\ul{0}} = S^+_{q_1}\times S^-_{q_1} \times \ldots \times S^+_{q_k} \times S^-_{q_k}$, with $S^+_{q_1}$ replaced by $W^+(q_1)$ in case $\cU_-=X$ and $S^-_{q_k}$ replaced by $W^-(q_k)$ in case $\cU_+=X$. It intersects the submanifold $\Gr(\ul{q})\subset\ti S(\ul{q})$ transversely by the Morse-Smale condition,
\begin{equation}\label{eq:0int}
\im\iota_{\ul{0}} \cap \Gr(\ul{q}) \;=\;
S^+_{q_1} \pitchfork  W^-_{q_-}  \;\times \left( {\textstyle\prod_{i=1}^{k-1}} \bigl( S^-_{q_i} \times S^+_{q_{i+1}}\bigr) \pitchfork \Gr^{q_i}_{q_{i+1}}   \right) \times\; S^-_{q_k}\pitchfork  W^+_{q_+} ,
\end{equation}
where in the case of finite ends $S^+_{q_1} \pitchfork  W^-_{q_-}$ resp.\ $S^-_{q_k} \pitchfork W^+_{q_+}$ is replaced by the trivial intersection $W^+_{q_1} \pitchfork \Psi_{\R_-}(\ti U_1(q_1))$ resp.\ $W^-_{q_k}\pitchfork \Psi_{\R_+}(\ti U_1(q_k))$.
In case $\cU_\pm=q_\pm$ the first resp.\ last factor is simply an intersection of stable and unstable manifold within $\ti S^+_{q_1}$ resp.\ $\ti S^-_{q_k}$.
In each of the middle factors the intersection is with the graph of the map $G^{q_i}_{q_{i+1}}$ which encodes the flow from $\ti S^-_{q_i}$ to $\ti S^+_{q_{i+1}}$, and hence transversality follows from the transverse intersection of the unstable manifold
$G^{q_i}_{q_{i+1}}(S^-_{q_i}) = W^-_{q_i}\cap \ti S^+_{q_{i+1}}$ with the stable manifold $S^+_{q_{i+1}}$ in $\ti S^+_{q_{i+1}}$.

\begin{rmk} \label{rmk:smooth eval}
The transversality $\im\iota_{\ul{\tau}} \pitchfork \Gr(\ul{q})$ for $\ul{\tau}=\ul{0}$ does not simply extend to small $\ul{\tau}\neq\ul{0}$ since $\Gr(\ul{q})$, and sometimes also the domain of $\iota_{\ul\tau}$, is noncompact.
We will however prove as part of the construction of the global charts that the smooth map $\iota:\bigl(\ul{\tau}, T_-, x_1 ,\ldots,y_k, T_+ \bigr)\mapsto\iota_{\ul{\tau}}( T_-, x_1 ,\ldots,y_k, T_+)$ is transverse to $\Gr(\ul{q})$, as a map from the manifold with corners $C:=\bigcup_{\ul{\tau}\in I_t(\ul{q})} \cD_{t,\ul{\tau}}(\ul{q}) \subset [0,2)^k \times S^+_{(\cQ_0,q_1)} \times  \ldots  \times S^-_{(q_k,\cQ_{k+1})}$ to $\ti S(\ul{q})$ in the following sense:
At every intersection point $c\in\iota^{-1}(\Gr(\ul{q}))$ the image of the ``interior tangent space'' $T_c^{\rm int}C$ under the differential $d_c\iota$ contains a complement of $T_{\iota(c)}\Gr(\ul{q})$. Here $T_c^{\rm int}C$ consists of those tangent vectors in $T_cC$ that are represented by paths $(-\ep,\ep)\to C$ tangent to the boundary $\partial C$.
Indeed, $d_c\iota(T_c^{\rm int}C)$ at $c=(\ul{\tau},\ldots)$ contains the image of $d\iota_{\ul{\tau}}$ on $T( S^+_{(\cQ_0,q_1)}\times \ldots   S^-_{(q_k,\cQ_{k+1})})$, so transversality follows from \eqref{req2} below.

This transversality with corners then induces a smooth structure on ${\rm Ev}(\ul{q})(\cV_t(\ul{q}))$, with the corner strata determined by the coordinates in $[0,2)^k$; see e.g.\ \cite{nielsen}.
Now the smooth structure on $\cV_t(\ul{q})$ will be defined by pullback with the homeomorphism ${\rm Ev}(\ul{q})$, so that ${\rm Ev}(\ul{q})$ is smooth by definition.
This directly implies smoothness of the evaluation maps \eqref{eval} at the endpoints since they are part of ${\rm Ev}(\ul{q})$.
\end{rmk}

\begin{rmk} \label{rmk:smooth eval 2}
At this point we can also deduce smoothness of the evaluation maps $\ev_H$ at the hypersurfaces of type \eqref{eval H}. In the interior $\cM(\cU_-,\cU_+)$ this was proven in Lemma~\ref{lem:cont eval}. For the global charts covering the boundary note that $\cV_t(\ul{q})$ intersects the domain of definition of $\ev_{\ti S^\pm_p}$ only when $p\in\ul{q}$ is part of the critical point sequence. Hence $\ev_{\ti S^\pm_p}$ is part of ${\rm Ev}(\ul{q})$, except for $\ev_{\ti S^+_{p=q_1}}$ in case $\cU_-=X$ or $\ev_{\ti S^-_{p=q_{k+1}}}$ in case $\cU_+=X$.
In the latter cases, the domain of the evaluations within the chart is
$\ev_-^{-1}(\Psi_{\R_-}(\ti S^+_p))\subset\cV_t(\ul{q})$ resp.\
$\ev_+^{-1}(\Psi_{\R_+}(\ti S^-_p))\subset\cV_t(\ul{q})$,
and the evaluations $\ev_-$ resp.\ $\ev_+$ are part of ${\rm Ev}(\ul{q})$, hence smooth by definition.
In this chart $\ev_{\ti S^+_{p=q_1}}$ is smooth since it is given by composing $\ev_-$ with the map $\Psi_{\R_-}(\ti S^+_p)\to \ti S^+_p, z \mapsto \Psi_{\R_+}(z)\cap\ti S^+_p$ which is smooth by Lemma~\ref{lem:cont eval}.
Similarly $\ev_{\ti S^-_{p=q_{k+1}}}$ is smooth since it is the composition of $\ev_+$ with the smooth map $z \mapsto \Psi_{\R_-}(z)\cap\ti S^-_p$.

For a general hypersurface $H\subset X$ transverse to the flow consider a trajectory near the boundary $\ul\g\in\cV_t(\ul{q})$ that also lies in the domain of $H$, i.e.\ $\im\ul\g\cap\Psi_{\R_\pm}(H)\neq\emptyset$. Its intersection point $\ev_H(\ul\g)$ with $H$ flows in finite time to the next entry set $\ti S^+_{q_j}$, unless it lies within $\ti U(q_j)$ or near the endpoint of $\ul\g$, in which case it flows in finite time backwards to the previous exit set $\ti S^-_{q_j}$. Now $\ev_H$ is smooth in a neighbourhood of $\ul\g\in\bM(\cU_-,\cU_+)$ since it can be expressed as composition of $\ev_{\ti S^\pm_{q_j}}$ with a smooth map from a neighbourhood of $\ev_{\ti S^\pm_{q_j}}(\ul\g)\in\ti S^\pm_{q_j}$ to a neighbourhood of $\ev_H(\ul\g)\in H$, given by the finite (backward) flow from $\ti S^\pm_{q_j}$ to $H$.
\end{rmk}

Next, recall that the evaluation maps ${\rm Ev}(\ul{q})$ identify the maximally broken trajectories in $\cV_t(\ul{q})_k=\cM(\ul{q})$ with the intersection $\im\iota_{\ul{0}}\cap \Gr(\ul{q}) $,
$$
{\rm Ev}(\ul{q}) \bigl( \cM(\ul{q}) \bigr) \;=\; \im\iota_{\ul{0}}\cap \Gr(\ul{q})  .
$$
In fact, this is an embedding by Lemma~\ref{lem:cont eval}.
In the case of finite ends, the evaluations moreover intertwine the restricted domains,
$$
{\rm Ev}(\ul{q}) \bigl( \cM_{t,\ul{\tau}}(\ul{q}) \bigr) \;=\;
\iota_{\ul{0}}(\cD_{t,\ul{\tau}}(\ul{q}))\cap \Gr(\ul{q})
\qquad \forall \ul{\tau}\in I_t(\ul{q}) .
$$
The construction of the global chart \eqref{eq:phi q} now requires identifications of $\iota_{\ul{\tau}}( \cD_{t,\ul{\tau}}(\ul{q}))\cap \Gr(\ul{q})$ with $\iota_{\ul{0}}(\cD_{t,\ul{\tau}}(\ul{q}))\cap \Gr(\ul{q})$ varying continuously with $\ul{\tau}\in I_t(\ul q)$.
We will achieve this by constructing a generalized tubular neighbourhood of the embedding of maximally broken trajectories ${\rm Ev}(\ul{q}) : \cM(\ul{q})\hookrightarrow\Gr(\ul{q})$, that is a surjective submersion
$$
\pi(\ul{q}) : \Gr(\ul{q}) \supset \cN(\ul{q}) \to \cM(\ul{q})
$$
of a neighbourhood $\cN(\ul{q})\subset \Gr(\ul{q})$ of $\im\iota_{\ul{0}}\cap \Gr(\ul{q}) $, which restricts to the diffeomorphism $\pi(\ul{q})|_{\im\iota_{\ul{0}}\cap \Gr(\ul{q}) }={\rm Ev}(\ul{q})^{-1} :  \im\iota_{\ul{0}}\cap \Gr(\ul{q})  \to \cM(\ul{q})$.
From this we will define the global chart as composition with the transition times and evaluation maps
\begin{align} \label{def phi}
\phi(\ul{q}) &:=  \tau(\ul{q})\times \bigl(\pi(\ul{q})\circ {\rm Ev}(\ul{q})  \bigr) : \;
\cV_t(\ul{q}) \;\to\;  [0,2)^k \times \cM(\ul{q}) .
\end{align}
Equivalently, this can be expressed as composition of a homeomorphism with the projection $\pi(\ul{q})$ restricted to domains varying with $\ul{\tau}\in I_t(\ul{q})$,
$$
\cV_t(\ul{q}) \; \xrightarrow{\tau(\ul{q}) \times  {\rm Ev}(\ul{q})} \;{\textstyle\bigcup_{\ul{\tau}\in I_t(\ul{q})}}  \{\ul{\tau}\} \times \bigl(\iota_{\ul{\tau}}(\cD_{t,\ul{\tau}}(\ul{q})) \cap \Gr(\ul{q}) \bigr) \; \xrightarrow{{\rm Id}_{[0,2)^k} \times  \pi(\ul{q})}\;{\textstyle\bigcup_{\ul{\tau}\in I_t(\ul{q})}}  \{\ul{\tau}\} \times  \cM_{t,\ul\tau}(\ul{q}) .
$$
In order for $\phi(\ul{q})$ to be a well defined map, we need to construct the tubular neighbourhoods and choose $t>0$ sufficiently small to ensure that
\begin{equation}\label{req1}
\iota_{\ul{\tau}}(\cD_{t,\ul{\tau}}(\ul{q})) \cap \Gr(\ul{q}) \subset \cN(\ul{q})
\qquad\forall \;  \ul{\tau} \in I_t(\ul{q}) .
\end{equation}
On the maximally broken trajectories $\cV_t(\ul{q})_k$, this map automatically has the required form by Proposition~\ref{prp:Ev}~(ii) and $\pi(\ul{q})|_{\im\iota_{\ul{0}}\cap \Gr(\ul{q}) }={\rm Ev}(\ul{q})^{-1}$.
Moreover, our definition of tubular neighbourhood ensures that each fiber $\pi(\ul{q})^{-1}(\ul{\g})$
is a smooth manifold and intersects $\iota_{\ul{0}}(\cD_{t,\ul{0}}(\ul{q}))$ uniquely and transversely in $\ul{\g}$.
In order for $\phi(\ul{q})$ to be a homeomorphism (and diffeomorphism in the interior) with the given image we need $\pi(\ul{q})$ to also induce diffeomorphisms
$\iota_{\ul{\tau}}( \cD_{t,\ul\tau}(\ul{q}))\cap \Gr(\ul{q}) \overset{\sim}{\to} \cM_{t,\ul\tau}(\ul{q})$ for $\ul\tau\neq 0$.
This can be ensured by the fiber intersections for each $\ul{\tau} \in I_t(\ul{q})$ being transverse at single points over $\cM_{t,\ul{\tau}}(\ul{q})$ and empty over the complement,
\begin{equation}\label{req2}
\ti S(\ul{q}) \;\supset\; \iota_{\ul{\tau}}(\cD_{t,\ul{\tau}}(\ul{q})) \pitchfork \pi(\ul{q})^{-1}({\ul{\g}})
= \begin{cases}
1\;\text{point} &; \ul\g\in\cM_{t,\ul{\tau}}(\ul{q}) , \\
\emptyset &; \ul\g\not\in\cM_{t,\ul{\tau}}(\ul{q}) .
\end{cases}
\end{equation}
This will also imply the transversality $\im\iota_{\ul{\tau}}\pitchfork\Gr(\ul{q})\subset \ti S(\ul{q})$ claimed in Remark~\ref{rmk:smooth eval} since $\pi(\ul{q})^{-1}({\ul{\g}})\subset\Gr(\ul{q})$.
Note also that in case $\cU_\pm=q_\pm$, when the domain of $\iota_{\ul{\tau}}$ is independent of $\ul\tau$ and compact, $\im\iota_{\ul{\tau}} \cap \pi(\ul{q})^{-1}({\ul{\g}})$ remains a single transverse intersection point for sufficiently small $|\ul{\tau}|$ and $\ul{\g}$ in a compact subset of $\cM(\ul{q})$.
In the iterative construction of the tubular neighbourhoods $\pi(\ul{q})$ the fibers over the complement of a compact subset will in fact be determined and automatically satisfy \eqref{req2} by the previous constructions.

\subsection{Tubular neighbourhoods of subspaces of maximally broken trajectories}
\label{sec:tub}

We will use the following generalized notion of tubular neighbourhoods of embeddings.

\begin{dfn} \label{dfn:tub}
Let $e: M \hookrightarrow G$ be an embedding of smooth manifolds.
Then a tubular neighbourhood of $e$ is a smooth surjective submersion $\pi:N\to M$ of an open neighbourhood $N\subset G$ of $e(M)$, which restricts to $\pi|_{e(M)}=e^{-1}$.
\end{dfn}

\begin{rmk} \label{rmk:tub}
Let $\pi:N\to M$ be a tubular neighbourhood of $e: M \hookrightarrow G$.
Then, by the implicit function theorem, for every $n\in N$ there is a diffeomorphism $V \times F \overset{\sim}{\to} U$ to a neighbourhood of $n$ that pulls back $\pi$ to the trivial fiber bundle over a neighbourhood $V\subset M$ of $\pi(n)$.
If $n=e(m)$ then one can make the pullback of $e: V\to U$ is a constant section.

If $M$ or $N$ are noncompact, then we may not deduce a global fiber bundle structure, but this local structure is sufficient for our purposes. In particular, each fiber $\pi^{-1}(m)$ is a smooth manifold and intersects $e(M)$ uniquely and transversely in $e(m)$.
\end{rmk}

The tubular neighbourhood $\pi(\ul{q}):\cN(\ul{q})\to\cM(\ul{q})$ of ${\rm Ev}(\ul{q}): \cM(\ul{q})\hookrightarrow \Gr(\ul{q})$ will be constructed as product
$\leftexp{-}{\pi}^{\cU_-}_{q_1}\times\pi^{q_1}_{q_2} \times \ldots\times \pi^{q_{k-1}}_{q_k}\times \leftexp{+}{\pi}_{\cU_+}^{q_k}$ of tubular neighbourhoods of the evaluation factors in ${\rm Ev}(\ul{q})$.
In each of these factors we will construct the tubular neighbourhoods by iteration over the following breaking numbers.

\begin{dfn}
For each pair $(\cP_-,\cP_+)$ of end conditions $\cP_\pm=p_\pm\in{\rm Crit}(f)$ or $\cP_\pm = X$ with $\cM(\cP_-,\cP_+)\neq\emptyset$ we define
$$
b(\cP_-,\cP_+):= \max\bigl\{   m \,\big| \, \exists p_1,\ldots, p_m\in{\rm Crit}(f) : \cM(\cP_-,p_1),\cM(p_1,p_2),\ldots,\cM(p_m,\cP_+)\neq \emptyset \bigr\}
$$
as maximal number of breakings of a trajectory from $\cP_-$ to $\cP_+$.
Moreover, for any tuple $\ul{q}=(\cQ_0,q_1,\ldots,q_k,\cQ_{k+1})$ of a critical point sequence and end conditions $\cQ_0\subset\cU_-$, $\cQ_{k+1}\subset\cU_+$ with $\cM(\ul{q})\neq\emptyset$ we denote by
$$
b(\ul{q}):=\max\{ b(\cU_-,q_1), b(q_1,q_2), \ldots,b(q_{k-1},q_k), b(q_k,\cU_+)\}
$$
the maximal breaking number between consecutive entries of $\ul{q}$.
\end{dfn}

To see that the breaking number is well defined recall that we defined ${\cM(p,p)=\emptyset}$. Note moreover that necessarily $\max f(\cP_-) \geq f(p_1)>\ldots>f(p_m)\geq \min f(\cP_+)$, so all breaking numbers are bounded above by the number of critical points of $f$.
We can hence use a finite iteration over $b=0,\ldots,\#{\rm Crit}(f)$ with a decreasing sequence $1\geq t_0 > t_1 > t_2 > \ldots >0$ to construct tubular neighbourhoods as follows.

\begin{itemize}
\item
For each pair $\cP_-=p_-,\cP_+=p_+\in{\rm Crit}(f)$ with $b(p_-,p_+)=b$ we will construct tubular neighbourhoods
\begin{equation}\label{tub nhbhds1}
\pi^{p_-}_{p_+} : \;\; \Gr^{p_-}_{p_+}(t)  \to \cM(p_-,p_+)
\qquad\text{of}\qquad
(\ev_{\ti S^-_{p_-}}\times \ev_{\ti S^+_{p_+}}):\cM(p_-,p_+)\hookrightarrow \Gr^{p_-}_{p_+}
\end{equation}
for $0<t\leq t_b$ by restriction of the construction for $t=t_b$ to
\begin{align*}
\Gr^{p_-}_{p_+}(t)&:=\Gr^{p_-}_{p_+} \cap \bigl( \Psi_\R(U_t(p_-)) \times \Psi_\R(U_t(p_+)) \bigr) \;\subset\;
\Gr^{p_-}_{p_+} .
\end{align*}
\item
For each pair $p_-,p_+\in{\rm Crit}(f)$ with $b(p_-,p_+)=b$ we then obtain tubular neighbourhoods
\begin{align*}
\leftexp{-}{\pi}^{p_-}_{p_+}  :  \;\; \leftexp{-}{\Gr}^{p_-}_{p_+}(t) \to \cM(p_-,p_+)
\qquad\text{of}\qquad
\ev_{\ti S^+_{p_+}}:\cM(p_-,p_+)\hookrightarrow \leftexp{-}{\Gr}^{p_-}_{p_+} = W^-_{p_-}\cap \ti{S}^+_{p_+}  \\
\leftexp{+}{\pi}_{p_+}^{p_-}:   \;\; \leftexp{+}{\Gr}^{p_-}_{p_+}(t)   \to \cM(p_-,p_+)
\qquad\text{of}\qquad
\ev_{\ti S^-_{p_-}}:\cM(p_-,p_+)\hookrightarrow \leftexp{+}{\Gr}^{p_-}_{p_+} = W^+_{p_+}\cap \ti{S}^-_{p_-}
\end{align*}
for $0<t\leq t_b$ on the domains
\begin{align*}
\leftexp{-}{\Gr}^{p_-}_{p_+}(t) := \leftexp{-}{\Gr}^{p_-}_{p_+}\cap \Psi_\R(U_t(p_+)), \qquad
\leftexp{+}{\Gr}^{p_-}_{p_+}(t) := \leftexp{+}{\Gr}^{p_-}_{p_+} \cap \Psi_\R(U_t(p_-))
\end{align*}
by pullback of $\pi^{p_-}_{p_+}$ under the embeddings to ${\Gr}^{p_-}_{p_+} = {\rm graph}(G^{p_-}_{p_+})$
$$
\bigl( (G^{p_-}_{p_+})^{-1}\times{\rm Id}_{\ti{S}^+_{p_+}} \bigr) :
\leftexp{-}{\Gr}^{p_-}_{p_+} \hookrightarrow {\Gr}^{p_-}_{p_+}, \qquad
\bigl( {\rm Id}_{\ti{S}^-_{p_-}} \times G^{p_-}_{p_+} \bigr) :
\leftexp{+}{\Gr}^{p_-}_{p_+} \hookrightarrow {\Gr}^{p_-}_{p_+} .
$$
These embeddings pull $\ev_{\ti S^-_{p_-}}\times \ev_{\ti S^+_{p_+}}$ back to $\ev_{\ti S^+_{p_+}}$ resp.\  $\ev_{\ti S^-_{p_-}}$, hence pullback of $\pi^{p_-}_{p_+}$ induces tubular neighbourhoods.
\item
For each $\cP_+=p_+\in{\rm Crit}(f)$ and $\cP_-=X$ with $b(X,p_+)=b$ we will construct tubular neighbourhoods
\begin{equation}
\leftexp{-}{\pi}^X_{p_+}  :  \;\; \leftexp{-}{\Gr}^X_{p_+}(t) \to \cM(X,p_+)
 \qquad\text{of}\qquad
 \ev_-:\cM(X,p_+)\hookrightarrow \leftexp{-}{\Gr}^X_{p_+}   = X  \label{tub nbhds2}
\end{equation}
for $0<t\leq t_b$ by restriction of the construction for $t=t_b$ to
\[
\leftexp{-}{\Gr}^X_{p_+}(t) := \Psi_{\R_-}(\ti U_t(p_+)) \;\;\subset\;\leftexp{-}{\Gr}^X_{p_+} .
\]
\item
For each $\cP_-=p_-\in{\rm Crit}(f)$ and $\cP_+=X$ with $b(p_-,X)=b$
we will construct tubular neighbourhoods
\begin{equation}
\leftexp{+}{\pi}_X^{p_-}:   \;\; \leftexp{+}{\Gr}^{p_-}_X(t)   \to \cM(p_-,X)
\qquad\text{of}\qquad
\ev_+:\cM(p_-,X)\hookrightarrow \leftexp{+}{\Gr}^{p_-}_X = X
\label{tub nbhds3}
\end{equation}
for $0<t\leq t_b$ by restriction of the construction for $t=t_b$ to
\[
\leftexp{+}{\Gr}^{p_-}_X(t) := \Psi_{\R_+}(\ti U_t(p_-)) \;\;\subset\;\leftexp{+}{\Gr}^{p_-}_X .
\]
\item
From the tubular neighbourhoods for $b(\cP_-,\cP_+)\leq b$ and $t\leq t_b$ we then obtain tubular neighbourhoods of ${\rm Ev}(\ul{q}):\cM(\ul{q})\to\Gr(\ul{q})$, given by
\begin{align*}
\pi(\ul{q}) := &\; \leftexp{-}{\pi}_{q_1}^{\cU_-}\times\pi^{q_1}_{q_2} \times \ldots \times \pi^{q_{k-1}}_{q_k} \times \leftexp{+}{\pi}^{q_k}_{\cU_+}
\;:\;  \cN_{t}(\ul{q}) \to \cM(\ul{q}) , \\
\cN_{t}(\ul{q})  := &\; \leftexp{-}{\Gr}^{\cU_-}_{q_1}(t) \times \Gr^{q_1}_{q_2}(t) \times \ldots \times \Gr^{q_{k-1}}_{q_k}(t) \times \leftexp{+}{\Gr}^{q_k}_{\cU_+}(t)   \;\subset\; \Gr(\ul{q}) , \nonumber
\end{align*}
for all tuples of critical point sequence and end conditions $\ul{q}=(\cQ_0,q_1,\ldots,q_k,\cQ_{k+1})$ with $b(\ul{q})\leq b$, not covered by Section~\ref{sec:k=0} or \ref{sec:k=1 special}.
These automatically satisfy \eqref{req1} for all $0<t\leq t_b$ since $I_t(\ul{q})$ is defined such that $\im\iota_{\ul{\tau}}\cap \Gr(\ul{q}) \subset \cN_t(\ul{q})$ for $\ul{\tau}\in I_t(\ul{q})$.
We will moreover make the construction and choice of $t_b>0$ such that the intersection properties of the fibers \eqref{req2} are satisfied for all $0<t\leq t_b$.
\item
From each tubular neighbourhood for $b(\ul{q})\leq b$ we then obtain a well defined map
$$
\phi(\ul{q}) := \bigl( {\rm Id}_{[0,2)^k} \times \pi(\ul{q}) \bigr) \circ  \bigl(\tau(\ul{q})\times{\rm Ev}(\ul{q})  \bigr) : \;
\cV_{t_b}(\ul{q}) \;\to\;  [0,2)^k \times \cM(\ul{q})
$$
as in \eqref{def phi}, and may restrict it to $\cV_t(\ul{q})$ for $t<t_b$.
\end{itemize}

\begin{rmk} \label{rmk:Grt converges}
In each case the open subsets $\leftexp{(\pm)}\Gr^{\cP_-}_{\cP_+}(t)\subset \leftexp{(\pm)}\Gr^{\cP_-}_{\cP_+}$ converge in the Hausdorff distance as $t\to 0$ to the image of $\cM(\cP_-,\cP_+)$ under the respective evaluation.
Indeed, in the identification $\Gr^{p_-}_{p_+}\simeq (\ev_- \times \ev_+)^{-1}\bigl(\ti{S}^-_{p_-}\times \ti{S}^+_{p_+}\bigr)  \subset \cM(X,X)$ from Section~\ref{sec:connections} we see that for any sequence $(\ev_- \times \ev_+)(\g^\nu)\in \Gr^{p_-}_{p_+}(2^{-\nu})$ there will be a convergent subsequence $\g^\nu\to\ul\g\in\bM(X,X)$ with $\ev_\pm(\ul\g)\in S^\pm_{p_\pm}$.
For this subsequence $(\ev_- \times \ev_+)(\g^\nu)$ converges to a point in $(\ev_- \times \ev_+)\bigl(  (\ev_- \times \ev_+)^{-1}\bigl(S^-_{p_-}\times S^+_{p_+}\bigr)  \subset \bM(X,X) \bigr) =
(\ev_{\ti S^-_{p_-}} \times \ev_{\ti S^+_{p_+}})(\bM(p_-,p_+))$, which is contained in the closure of $(\ev_{\ti S^-_{p_-}} \times \ev_{\ti S^+_{p_+}})(\cM(p_-,p_+))$.
On the other hand, this latter set is contained in $\Gr^{p_-}_{p_+}(t)$ for all $t>0$, which proves Hausdorff convergence of  $\Gr^{p_-}_{p_+}(t)$ to $(\ev_{\ti S^-_{p_-}} \times \ev_{\ti S^+_{p_+}})(\cM(p_-,p_+))$. The other cases are analogous.
\end{rmk}

In order for this construction of $\phi(\ul{q})$ to provide the global charts of Theorem~\ref{thm:global with ends}, we need to impose further conditions on the tubular neighbourhoods, taking the properties of $\tau(\ul{q})\times{\rm Ev}(\ul{q})$ given by Proposition~\ref{prp:Ev} into account. In unravelling the associativity (iii) note that the insertion of a nontrivial $\ul{q}'$ implies $b(\ul{q}')<b(\ul{q})$ and $b(\ul{Q})\leq b(\ul{q})$, so the compatibility can be phrased as condition on the factors of $\pi(\ul{q})$.

\begin{lem} \label{lem:pi}
Let the special global charts in Sections~\ref{sec:k=0} and \ref{sec:k=1 special} be fixed, and for some $b\geq 1$ suppose that the above construction of $\phi(\ul{q})$ for $b(\ul{q})\leq b-1$ satisfies Theorem~\ref{thm:global with ends} for $0<t\leq t_{b-1}$.
Then the following conditions on $\leftexp{(\pm)}\pi^{\cP_-}_{\cP_+}$ for $b(\cP_-,\cP_+)=b$ ensure that the induced maps $\phi(\ul{q})$ satisfy Theorem~\ref{thm:global with ends} up to breaking number $b$ for $0<t\leq t_b$.
\begin{enumerate}
\item
The induced maps $\pi(\ul{q})$ for any critical point sequence and end conditions with $b(\ul{q})=b$ satisfy transversality to the fibers \eqref{req2}, which we may simplify to
\[
\im\iota_{\ul{\tau}} \pitchfork \pi(\ul{q})^{-1}({\ul{\g}})
= 1\;\text{point} \qquad  \forall \; \ul\tau\in I'_{t_b}(\ul{q}) , \ul\g\in\cM(\ul{q})
\]
with $I'_t(\ul{q}):= \left\{\begin{smallmatrix}
[0,\frac 12 t] &; \cQ_0=\ti U(q_1) \\
[0,t) &;  \text{otherwise}
\end{smallmatrix}\right\}
\times [0,t)^{k-2}\times
\left\{\begin{smallmatrix}
[0,\frac 12 t] &; \cQ_{k+1}=\ti U(q_k) \\
[0,t) &;  \text{otherwise}
\end{smallmatrix}\right\} .
$
\item
(The canonical form on the maximally broken trajectories is automatically satisfied.)
\item
For any nontrivial critical point sequence $\ul{q}'$ with end conditions associated to $(\cP_-,\cP_+)$ with $b(\cP_-,\cP_+)=b$ and the associated $\ul{\ti q}'$ with open end conditions of the form
\begin{align*}
& \ul{q}'=(p_-,q'_1,\ldots,q'_\ell,p_+) , \qquad \quad
\ul{\ti q}'=\bigl(\ti U(p_-),p_-,q_1', \ldots, q_\ell', p_+,\ti U(p_+)\bigr) ,\\
\text{resp.}\quad &
 \ul{q}'=(\cQ'_0,q'_1,\ldots,q'_\ell,p_+) , \qquad\hspace{3.2mm}
\ul{\ti q}'= \bigl(\cQ_0',q_1', \ldots, q_\ell', p_+,\ti U(p_+)\bigr) , \\
\text{resp.}\quad &
\ul{q}'=  (p_-,q'_1,\ldots,q'_\ell,\cQ'_{\ell+1}) , \qquad
\ul{\ti q}'= \bigl(p_-,\ti U(p_-), q_1', \ldots, q_\ell',\cQ'_{\ell+1}\bigr)
\end{align*}
the submersions are given on the domains of trajectories intersecting all $\ti U_{t_b}(q_i')$ by
\begin{align*}
 \pi^{p_-}_{p_+}\circ \bigl({\rm ev}_-\times {\rm ev}_+\bigr)
\big|_{\cM(\ti{S}^-_{p_-},\ti{S}^+_{p_+}) \cap \cV_{t_b}(\ul{\ti q'})_0}
&= \phi(\ul{q}') \circ {\rm Pr}_{\ul{q}'} \circ \phi(\ul{\ti q}') , \\
\text{resp.}\qquad \qquad
\leftexp{-}\pi^X_{p_+} \circ \ev_- \big|_{\cM(X,\ti{S}^+_{p_+}) \cap \cV_{t_b}(\ul{\ti q'})_0}
\;\;\;&= \phi(\ul{q}') \circ {\rm Pr}_{\ul{q}'} \circ \phi(\ul{\ti q}')  , \\
\text{resp.}\qquad \qquad
\leftexp{+}\pi^{p_-}_X  \circ  \ev_+ \big|_{\cM(\ti{S}^-_{p_-},X) \cap \cV_{t_b}(\ul{\ti q'})_0}
\;\;\;&= \phi(\ul{q}') \circ {\rm Pr}_{\ul{q}'} \circ \phi(\ul{\ti q}')
\end{align*}
with the canonical projections ${\rm Pr}_{\ul{q}'}: I_t(\ul{\ti q}')\times\cM(\ul{\ti q}') \to [0,1)^\ell \times \cM(\ul{q}')$.
\item
For $b(X,p_+)=b$ resp.\ $b(p_-,X)=b$ the submersions near critical points are given explicitly via \eqref{Bpm} by
\begin{align}\label{pi near crit}
\leftexp{-}\pi^{X}_{p_+}|_{\ti U_{t_b}(p_+)}=\ev_-^{-1}\circ {\rm pr}_{\ti B^+_{p_+}},
\qquad
\bigl(\leftexp{-}\pi^{X}_{p_+}\bigr)^{-1}\bigl(\cM(\ti U_{t_b}(p_+),p_+)\bigr) \subset \ti U_{t_b}(p_+),
\\
\leftexp{+}\pi_{X}^{p_-}|_{\ti U_{t_b}(p_-)}=\ev_+^{-1}\circ {\rm pr}_{\ti B^-_{p_-}} ,
\qquad
\bigl(\leftexp{+}\pi_{X}^{p_-}\bigr)^{-1}\bigl(\cM(p_-,\ti U_{t_b}(p_-))\bigr) \subset \ti U_{t_b}(p_-).
 \nonumber
\end{align}
(The explicit transition times and relation between charts for different end conditions are automatically satisfied.)
\end{enumerate}
\end{lem}

\begin{proof}
To understand the simplification in (i) we begin by noting that $\cM_{t,\ul{\tau}}(\ul{q})=\cM(\ul{q})$ and $\cD_{t,\ul{\tau}}(\ul{q})=S^+_{\cQ_0,q_1}\times\ldots\times S^-_{q_k,\cQ_{k+1}}$ unless $\tau_1 > \tfrac t2$ in case $\cQ_0=\ti U(q_1)$ or $\tau_k > \tfrac t2$ in case $\cQ_{k+1}=\ti U(q_k)$. In the latter cases for $t\leq t_b$ we will show that the unique transverse intersection follows from the intersection property for $\tau_1=\frac t2$ or $\tau_k=\frac t2$.
We will do this in the case $\cQ_0=\ti U(q_1)$, $\cQ_{k+1}=\ti U(q_k)$, and $\tau_1,\tau_k > \tfrac t2$. The arguments for each end will clearly be separate so that this also covers the case of just one end condition near a critical point.
In the chosen case for $\ul\tau\in (\frac t2,1+t)\times[0,t)^{k-2}\times (\frac t2,1+t)$ we have
$\cD_{t,\ul{\tau}}(\ul{q})=\frac t{2\tau_1}\ti B^+_{q_1} \times S^-_{q_1}\ldots  S^+_{q_k} \times \frac t{2\tau_k}\ti B^-_{q_k}$
and by pullback to $\ti B^+_{q_1} \times \ldots  \times \ti B^-_{q_k}$ obtain
\begin{align*}
\iota_{\ul\tau}\bigl(\cD_{t,\ul{\tau}}(\ul{q})\bigr) & = R_{t,\tau_1,\tau_k} \bigl( \im\iota_{(\frac t2,\tau_2,\ldots,\tau_{k-1},\frac t2)}\bigr) \\
\text{with}\quad
R_{t,\tau_1,\tau_k} &= \left(
\bigl(\tfrac t{2\tau_1}{\rm Id}_{\ti B^+_{q_1}}\times \tfrac {2\tau_1}{t}{\rm Id}_{\ti B^-_{q_1}}\bigr)
\times {\rm Id}_{S^-_{q_1}}\ldots {\rm Id}_{S^+_{q_k}}\times
\bigl(\tfrac {2\tau_k}t{\rm Id}_{\ti B^+_{q_k}}\times \tfrac t{2\tau_k}{\rm Id}_{\ti B^-_{q_k}}\bigr)
\right) .
\end{align*}
On the other hand, by (iv) the fiber over any $\ul\g=(\g_0,\ldots,\g_k)\in \cM_{t,\ul{\tau}}(\ul{q})$ is
\begin{align*}
\pi(\ul q)^{-1}(\ul\g) &=  \left( \bigl(\ev_-(\g_0) \times \ti B^-_{q_1} \bigr), \ldots,
\bigl( \ti B^+_{q_k}\times \ev_+(\g_k) \bigr) \right)  \cap \cN_t(\ul{q}) \\
&= \bigl(R_{t,\tau_1,\tau_k} \bigl( \pi(\ul q)^{-1}(\g'_0,\g_1 \ldots \g_{k-1},\g'_k)\bigr) \bigr)\cap \cN_t(\ul{q})
\end{align*}
with $\g'_0=\ev_-^{-1}(\tfrac{2\tau_1}t \ev_-(\g_0))\in\cM(\ti U(q_1),q_1)$ and
$\g'_k=\ev_+^{-1}(\tfrac{2\tau_k}t\ev_+(\g_k))\in\cM(q_k,\ti U(q_k))$.
Hence transversality and uniqueness of $\iota_{\ul{\tau}}\bigl(\cD_{t,\ul{\tau}}(\ul{q})\bigr) \cap \pi(\ul{q})^{-1}({\ul{\g}})$ follows by linear transformation with $R_{t,\tau_1,\tau_k}$ from transversality and uniqueness of $\im\iota_{(\frac t2,\tau_2,\ldots,\tau_{k-1},\frac t2)} \cap \pi(\ul{q})^{-1}(\g_0',\ldots,\g_k')$.

Furthermore, by definition of $\iota_{\ul\tau}$ and condition (iv) any point in $\iota_{\ul{\tau}}^{-1}\bigl(\pi(\ul{q})^{-1}(\g_0,\ldots,\g_k)\bigr)$ is of the form $\bigl(0,\ev_-(\g_0), \ast \ldots \ast \bigr)$ in case $\cQ_0=\ti U(q_1)$, and of the form $\bigl( \ast \ldots \ast  , \ev_+(\g_k), 0 \bigr)$
in case $\cQ_{k+1}=\ti U(q_k)$.
Hence $\iota_{\ul{\tau}}(\cD_{t_b,\ul{\tau}}(\ul{q})) \cap \pi^{-1}(\ul{q})(\ul{\g})=\emptyset$ is automatic for $\ul\g\notin\cM_{t,\ul{\tau}}(\ul{q})$.
Moreover, for $\ul\tau\in[0,t)^k$ we have $\im\iota_{\ul\tau}\cap\Gr(\ul{q})\subset\cN_{t}(\ul{q})$ by definition, so the $t$-dependence of the domain of $\pi(\ul{q})$ is immaterial. In particular, the intersection property for $t_b$ implies the analogous property for all $0<t<t_b$.

Now the intersection conditions in (i) together with the characterization of the image
$(\tau(\ul{q}) \times {\rm Ev}(\ul{q}))(\cV_{t_b}(\ul{q}))=\bigcup_{\ul{\tau}\in I_{t_b}(\ul{q})}  \{\ul{\tau}\} \times \bigl(\iota_{\ul{\tau}}(\cD_{t_b,\ul{\tau}}(\ul{q})) \cap \Gr(\ul{q})\bigr)$ evidently imply that $\phi(\ul{q})=({\rm Id}_{[0,2)^\ell}\times \pi(\ul{q}))\circ (\tau(\ul{q}) \times {\rm Ev}(\ul{q}))$ is injective with the claimed image. Moreover, it is continuous and smooth on $\cV_{t_b}(\ul{q})_0$ by the regularity of its factors. So it remains to show that
$({\rm Id}_{\R^\ell}\times{\rm d}\pi(\ul{q}))$
restricts to an isomorphism from
${\rm T}(\tau(\ul{q}) \times {\rm Ev}(\ul{q})(\cV_{t_b}(\ul{q})))
=\R^\ell\times \bigl(\im{\rm d}\iota_{\ul{\tau}}\cap {\rm T}\Gr(\ul{q})\bigr)$
to $\R^\ell\times{\rm T}\ti\cM(\ul{q})$.
The latter follows from the transversality ${\rm T}\ti S(\ul{q}) = \ker\rd\pi(\ul{q}) +  \im{\rm d}\iota_{\ul{\tau}}$  between the fibers of $\pi(\ul{q})$ and the embeddings $\iota_{\ul{\tau}}$ since ${\rm d}\pi(\ul{q})$ is surjective and $\ker{\rm d}\pi(\ul{q})\subset {\rm T}\Gr(\ul{q})$.

Property (ii) of Theorem~\ref{thm:global with ends} follows from Proposition~\ref{prp:Ev}~(ii) and the defining property $\pi(\ul{q})|_{{\rm Ev}(\ul{q})(\cM(\ul{q}))} ={\rm Ev}(\ul{q})^{-1}$ of tubular neighbourhoods.
For property (iv) of Theorem~\ref{thm:global with ends}, the explicit form of the charts for trajectories starting or ending near critical points is equivalent to \eqref{pi near crit}.  The transition times and relation between different end conditions are determined by Proposition~\ref{prp:Ev}~(iv). Note here that the tubular neighbourhood $\pi(\ul{q})$ is the same for both choices of end conditions in $\cU_\pm =X$, and in the relevant factors is given by \eqref{pi near crit}.

To check that condition (iii) implies the compatibility of charts in Theorem~\ref{thm:global with ends}
we begin by noting that the compatibility is trivially satisfied for $k=0$ when $\phi(\ul{q})={\rm Id}$.
In the notation of the theorem the breaking numbers are related by $b(\ul{q}) \geq b(q_j,q_{j+1}) \geq b(\ul{q}')+\ell$ and $b(\ul{Q})\leq b(\ul{q}),b(\ul{q})'$, so in order to check the compatibility up to breaking number $b$ it suffices to consider the case $b(\ul{q})=b > b(\ul{q}')$ and $k\geq 1$.
Now the complement $\cV_t(\ul{Q})\setminus \cV_t(\ul{q})$ consists of those trajectories that break at one or several points of $\ul{q}'$, so the overlap $\cV_t(\ul{q}) \cap \cV_t(\ul{Q})$ consists of all those trajectories $\ul{\g}\in \cV_t(\ul{Q})$ that do not break between $q_j$ and $q_{j+1}$.
Here in case $j=0$ resp.\ $j=k$ we have to replace $q_j$ by $\cQ'_0$ resp.\ $q_{j+1}$ by $\cQ'_{k+1}$, and will consider these cases separately later.
For $0<j<k$ and $\ul\g\in\cV_t(\ul{q}) \cap \cV_t(\ul{Q})$ the transition times through $U(q_i')$ (which are positive corresponding to no breaking) $\tau(\ul{q}')(\ul\g)=\tau(\ul{q}')(\g)$ and evaluations ${\rm Ev}(\ul{q}')(\ul\g)={\rm Ev}(\ul{q}')(\g)$ near $q'_i$ are determined by the restriction $\g:=\rho^{q_j}_{q_{j+1}}(\ul\g)\in\cM(\ti S^-_{q_j},\ti S^+_{q_{j+1}})$ corresponding to
$$
( z^-_\g , z^+_\g ):=(\ev_-(\g),\ev_+(\g))=(\ev_{\ti S^-_{q_j}}(\ul\g),\ev_{\ti S^+_{q_{j+1}}}(\ul\g))
\in \Gr^{q_j}_{q_{j+1}}.
$$
With this notation the compatibility condition (iii) of Theorem~\ref{thm:global with ends} in case $0<j<k$ becomes the following tuple of conditions on the tubular neighbourhood $\pi^{q_j}_{q_{j+1}}$
for all $\g\in\cM(\ti S^-_{q_j},\ti S^+_{q_{j+1}})$ with $\im\g\cap U_t(q'_i)\neq\emptyset$ for $i=0,\ldots,\ell+1$.
\begin{align*}
\pi^{q'_i}_{q'_{i+1}} \bigl(\bigl({\rm ev}_{\ti{S}^-_{q'_i}}\times {\rm ev}_{\ti{S}^+_{q'_{i+1}}}\bigr) (\g)\bigr)
&=
\begin{cases}
\leftexp{-}\pi^{q_j}_{q'_{1}}\bigl( {\rm ev}_{\ti{S}^+_{q'_1}} \bigl( \pi^{q_j}_{q_{j+1}}(z^-_\g,z^+_\g) \bigr)\bigr)
&; i=0 , \\
\pi^{q'_i}_{q'_{i+1}}\bigl( \bigl(  {\rm ev}_{\ti{S}^+_{q'_i}} \times  {\rm ev}_{\ti{S}^+_{q'_{i+1}}}  \bigr)\bigl( \pi^{q_j}_{q_{j+1}} (z^-_\g,z^+_\g)\bigr) \bigr)
&; 0<i<\ell, \\
\leftexp{+}{\pi}^{q'_\ell}_{q_{j+1}} \bigl( {\rm ev}_{\ti{S}^-_{q'_\ell}} \bigl( \pi^{q_j}_{q_{j+1}}(z^-_\g,z^+_\g)\bigr)\bigr)
&; i=\ell ,
\end{cases} \\
\tau_{q'_i} (\g) &= \tau_{q'_i} \bigl( \pi^{q_j}_{q_{j+1}}(z^-_\g,z^+_\g) \bigr) \qquad\qquad\qquad\qquad\qquad \forall\;  1 \leq i\leq \ell .
\end{align*}
In case $j=0$, $\cU_-=q_-=q_0=q_0'$ resp.\ $j=k$, $\cU_+=q_+=q_{k+1}=q'_{\ell+1}$ the compatibility can analogously be rewritten as conditions on $\leftexp{-}\pi^{q_0}_{q_1}(z^+)$ resp.\ $\leftexp{+}\pi^{q_k}_{q_{k+1}}(z^-)$ on the right hand side for all $z^+\in \leftexp{-}{\Gr}^{q_0}_{q_1}$ resp.\ $z^-\in \leftexp{+}{\Gr}_{q_{k+1}}^{q_k}$ with $\Psi_\R(z^\pm)\cap U_t(q'_i)\neq\emptyset$ and the corresponding trajectory $\g=\Psi(\cdot,z^\pm)$.
By the pullback definition of  $\leftexp{\pm}\pi^{q_j}_{q_{j+1}}$, these are equivalent to requirements on $\pi^{q_0}_{q_1}(z_\g^-,z_\g^+)$ for $\g\in\cM(S^-_{q_0},\ti S^+_{q_1})$ as above resp.\ for $\pi^{q_k}_{q_{k+1}}(z_\g^-,z_\g^+)$ for $\g\in\cM(\ti S^-_{q_k},S^+_{q_{k+1}})$ as above.
On the left hand side, the conditions involve $\leftexp{-}{\pi}^{q'_0}_{q'_1}\circ {\rm ev}_{\ti S^-_{q'_1}}$ for $i=0$ resp.\  $\leftexp{+}{\pi}^{q'_\ell}_{q'_{\ell+1}} \circ {\rm ev}_{\ti S^+_{q'_\ell}}$ for $i=\ell$, however these equal $\pi^{q'_0}_{q'_1} \bigl({\rm ev}_{\ti{S}^-_{q'_0}}\times {\rm ev}_{\ti{S}^+_{q'_1}}\bigr)$ resp.\ $\pi^{q'_0}_{q'_1} \bigl({\rm ev}_{\ti{S}^-_{q'_0}}\times {\rm ev}_{\ti{S}^+_{q'_1}}\bigr)$
by the pullback definition of $\leftexp{\pm}\pi^{q'_i}_{q'_{i+1}}$. Hence the requirements in these cases are of the same form as those for $0<j<k$.

Next, we compare these requirements to the definition of the chart for $b(\ul{q}')<b$,
$$
\phi(\ul{q}') : \;
\cM(q_j,q_{j+1}) \supset \cV_t(\ul{q}')_0 \;\overset{\sim}{\longrightarrow}\;  (0,t)^\ell \times \cM(\ul{q}') , \quad
\eta \mapsto  (\tau'_1,\ldots,\tau'_\ell,\g'_0,\ldots,\g'_\ell) ,
$$
which is given by the transition times near $q'_1,\ldots,q'_\ell$ and projection to the trajectories,
\begin{align*}
\g'_i  = \left. \begin{cases}
\leftexp{-}\pi^{q_j}_{q'_{1}}\bigl( {\rm ev}_{\ti{S}^+_{q'_1}} (\e) \bigr)
&; i=0 \\
\pi^{q'_i}_{q'_{i+1}}  \bigl(  {\rm ev}_{\ti{S}^+_{q'_i}} \times  {\rm ev}_{\ti{S}^+_{q'_{i+1}}}  \bigr) (\e)
&; i=1,\ldots \ell-1 \\
\leftexp{+}{\pi}^{q'_\ell}_{q_{j+1}} \bigl( {\rm ev}_{\ti{S}^-_{q'_\ell}} (\e)\bigr)
&; i=\ell
\end{cases}
\right\},
\qquad
\tau'_i =  \tau_{q'_i}(\e)  \qquad \forall\; 1\leq i\leq \ell .
\end{align*}
Comparing this definition of $\phi(\ul{q}')$ with the above requirements, we see that property (iii) of Theorem~\ref{thm:global with ends} for $0<j<k$ is equivalent to the requirement for all
$\g\in\cM(\ti S^-_{q_j},\ti S^+_{q_{j+1}})\cap\cV_t(\ul{\ti q}')_0$
\begin{equation}\label{reqqq}
\phi(\ul{q}') \bigl( \pi^{q_j}_{q_{j+1}} (z^-_\g,z^+_\g) \bigr)
= \Bigl( \bigl( \tau_{q'_i}(\g) \bigr)_{i=1,\ldots,\ell}  \,,\,
\bigl(
\pi^{q'_i}_{q'_{i+1}} \bigl({\rm ev}_{\ti{S}^-_{q'_i}}\times {\rm ev}_{\ti{S}^+_{q'_{i+1}}}\bigr) (\g) \bigr)_{i=0,\ldots,\ell}
 \Bigr) .
\end{equation}
The right hand side can be expressed as composition ${\rm Pr}_{\ul{q}'} \circ \phi(\ul{\ti q}')$ of the chart for the critical point sequence $\ul{\ti q}'=\bigl(\ti U(q_j),q_j=q'_0,q_1', \ldots, q_\ell', q'_{\ell+1}=q_{j+1},\ti U(q_{j+1})\bigr)$ with the projection
$$
{\rm Pr}_{\ul{q}'} : [0,2) \times [0,1)^\ell \times [0,2) \times \cM(\ti U(q_{j}),q_{j}) \times \cM(\ul{q}') \times \cM(q_{j+1},\ti U(q_{j+1})) \to [0,1)^\ell  \times \cM(\ul{q}').
$$
Since $\phi(\ul{q}')$ is invertible and $\ev_-\times \ev_+ : \g \to (z^-_\g,z^+_\g)$ identifies the domains $\cM(\ti S^-_{q_j},\ti S^+_{q_{j+1}})\to\Gr^{q_j}_{q_{j+1}}$, this makes the requirement $\pi^{q_j}_{q_{j+1}} \circ \bigl( \ev_-\times \ev_+\bigr)=\phi(\ul{q}')^{-1}\circ {\rm Pr}_{\ul{q}'} \circ \phi(\ul{\ti q}')$, as claimed.

In case $j=0$, $\cU_-=X$ resp.\ $j=k$, $\cU_+=X$ the compatibility can be rewritten as above into conditions on $\leftexp{-}{\pi}^X_{q_1}(z_\g^-)$ resp.\ $\leftexp{+}\pi^{q_k}_X(z_\g^+)$ instead of  $\pi^{q_j}_{q_{j+1}}(z_\g^-,z_\g^+)$ on the right hand side, with the further replacements $\bigl( q_0 \leadsto X,  {\rm ev}_{\ti{S}^+_{q'_1}} \leadsto \ev_-,  \tau_{q'_1} \leadsto \tau_{(\cQ'_0,q'_1)} \bigr)$ resp.\  $\bigl( q_{k+1} \leadsto X,  {\rm ev}_{\ti{S}^-_{q'_\ell}} \leadsto \ev_+,  \tau_{q'_\ell} \leadsto \tau_{(q'_\ell,\cQ'_{k+1})} \bigr)$, and with a modification of the left hand side to $\leftexp{-}{\pi}^X_{q'_1}\circ {\rm ev}_-$ and $\tau_{(\cQ'_0,q'_1)}$ for $i=j=0$ resp.\ to $\leftexp{+}{\pi}^{q'_\ell}_X \circ {\rm ev}_+$ and $\tau_{(q'_\ell,\cQ'_{k+1})}$ for $i=\ell$, $j=k$. In these cases the requirements are for all restricted trajectories $\g\in\cM(X,\ti S^+_{q_1})$ resp.\ $\g\in\cM(\ti S^-_{q_k},X)$ with $\im\g\cap U_t(q'_i)\neq\emptyset$ for $i=1,\ldots,\ell$ and $\ev_-(\g)\in\cQ_0'$ resp.\ $\ev_+(\g)\in\cQ_{k+1}'$, and the corresponding $z_\g^\pm=\ev_\pm(\g)$.
Spelling this out for $j=0$, the requirements are
\begin{align*}
\leftexp{-}{\pi}^X_{q'_1} \bigl( {\rm ev}_- (\g) \bigr)
&=
\leftexp{-}\pi^X_{q'_{1}}\bigl( {\rm ev}_{\ti{S}^+_{q'_1}} \bigl( \leftexp{-}{\pi}^X_{q_1} (z_\g^-) \bigr)\bigr) , \\
\pi^{q'_i}_{q'_{i+1}} \bigl({\rm ev}_{\ti{S}^-_{q'_i}}\times {\rm ev}_{\ti{S}^+_{q'_{i+1}}}\bigr) (\g)
&=
\begin{cases}
\pi^{q'_i}_{q'_{i+1}} \bigl(  {\rm ev}_{\ti{S}^+_{q'_i}} \times  {\rm ev}_{\ti{S}^+_{q'_{i+1}}}  \bigr)\bigl( \leftexp{-}{\pi}^X_{q_1} (z_\g^-)\bigr)
&; 0<i<\ell, \\
\leftexp{+}{\pi}^{q'_\ell}_{q_1} \bigl( {\rm ev}_{\ti{S}^-_{q'_\ell}} \bigl( \leftexp{-}{\pi}^X_{q_1} (z_\g^-)\bigr)\bigr)
&; i=\ell ,
\end{cases} \\
\tau_{(\cQ'_0,q'_1)} (\g) &=  \tau_{(\cQ'_0,q'_1)}\bigl(  \leftexp{-}{\pi}^X_{q_1} (z_\g^-) \bigr) , \\
\tau_{q'_i} (\g) &= \tau_{q'_i} \bigl(  \leftexp{-}{\pi}^X_{q_1} (z_\g^-) \bigr) \qquad\qquad\qquad\qquad\qquad\forall\; 2 \leq i \leq \ell .
\end{align*}
We again compare this with the chart $\phi(\ul{q}')$, which now depends on the choice of end condition $\cQ'_0\subset X$ resp.\ $\cQ'_{k+1}\subset X$ via the modification $\g'_0 = \leftexp{-}\pi^X_{q'_{1}}\bigl( {\rm ev}_- (\g) \bigr)$ and $\tau'_1=\tau_{(\cQ'_0,q'_1)}(\g)$ resp.\ $\g'_\ell = \leftexp{+}{\pi}^{q'_\ell}_X \bigl( {\rm ev}_+ (\g)\bigr)$ and $\tau'_\ell = \tau_{(q'_\ell,\cQ'_{k+1})}(\g)$. Spelling this out for $j=0$, the chart is
$$
\phi(\ul{q}') : \;
\cM(\cQ'_0,q_1) \supset \cV_t(\ul{q}')_0 \;\hookrightarrow\;  (0,2)^\ell \times \cM(\ul{q}') , \quad
\g \mapsto  (\tau'_1,\ldots,\tau'_\ell,\g'_0,\ldots,\g'_\ell) ,
$$
\begin{align*}
\g'_i  = \left.\begin{cases}
\leftexp{-}\pi^X_{q'_{1}}\bigl( {\rm ev}_-(\g) \bigr)
&; i=0 \\
\pi^{q'_i}_{q'_{i+1}}  \bigl(  {\rm ev}_{\ti{S}^+_{q'_i}} \times  {\rm ev}_{\ti{S}^+_{q'_{i+1}}}  \bigr) (\g)
&; i=1,\ldots \ell-1 \\
\leftexp{+}{\pi}^{q'_\ell}_{q_1} \bigl( {\rm ev}_{\ti{S}^-_{q'_\ell}} (\g)\bigr)
&; i=\ell
\end{cases} \right\},\qquad
\tau'_i = \left.\begin{cases}
\tau_{(\cQ'_0,q'_1)}(\g) &, i=0 \\
 \tau_{q'_i}(\g)  &, 2\leq i \leq \ell
 \end{cases}\right\}.
\end{align*}
This shows that property (iii) of Theorem~\ref{thm:global with ends} for $j=0$, $\cU_-=X$ resp.\ $j=k$, $\cU_+=X$ is equivalent to a requirement of the same form as \eqref{reqqq} for $\phi(\ul{q}') \bigl(\leftexp{-}\pi^X_{q_1}(z_\g^-)\bigr)$ resp.\ $\phi(\ul{q}') \bigl(\leftexp{+}\pi^{q_k}_X(z_\g^+)\bigr)$
and all $\g\in \cM(X,\ti{S}^+_{p_+}) \cap \cV_{t_b}(\ul{\ti q'})_0$ resp.\ $\g\in \cM(\ti{S}^-_{p_-},X) \cap \cV_{t_b}(\ul{\ti q'})_0$, just with the first resp.\ last trajectory and transition time on the right hand side replaced by $\leftexp{-}\pi^X_{q'_1} \bigl({\rm ev}_- (\g)\bigr)$ and $\tau_{(\cQ'_0,q'_1)}(\g)$
resp.\ $\leftexp{+}\pi_X^{q'_\ell} \bigl({\rm ev}_+\bigr) (\g)$ and $\tau_{(q'_\ell,\cQ'_{k+1})}(\g)$.
We may again express the right hand side as composition ${\rm Pr}_{\ul{q}'} \circ \phi(\ul{\ti q}')$ of the chart for the associated critical point sequence $\ul{\ti q}'$ with the canonical projection ${\rm Pr}_{\ul{q}'}: I_t(\ul{\ti q}')\times\cM(\ul{\ti q}') \to [0,1)^\ell \times \cM(\ul{q}')$.
In case $j=0$ that is $\ul{\ti q}' = \bigl(\cQ_0',q_1', \ldots, q_\ell', q'_{\ell+1}=q_1,\ti U(q_1)\bigr)$ satisfying $b(\ul{\ti q}')=b(\ul q')$, and the projection is
$$
{\rm Pr}_{\ul{q}'} : [0,2)^\ell \times [0,2) \times \cM(\ul{q}') \times \cM(q_1,\ti U(q_1)) \to [0,2)^\ell  \times \cM(\ul{q}').
$$
Since $\phi(\ul{q}')$ is a homeomorphism and $\ev_- : \g \to z^-_\g$ identifies the domains $\cM(X,\ti S^+_{q_1})\to\leftexp{-}{\Gr}^X_{q_1}(1)$, this makes the requirement $\leftexp{-}{\pi}^X_{q_1} \circ \ev_-=\phi(\ul{q}')^{-1}\circ {\rm Pr}_{\ul{q}'} \circ \phi(\ul{\ti q}')$.
Similarly, $\ev_+ : \g \to z^+_\g$ identifies the domains $\cM(\ti S^-_{q_k},X)\to\leftexp{+}{\Gr}_X^{q_k}(1)$, which makes the requirement $\leftexp{+}{\pi}_X^{q_k} \circ \ev_+=\phi(\ul{q}')^{-1}\circ {\rm Pr}_{\ul{q}'} \circ \phi(\ul{\ti q}')$.
This finishes the proof for property (iii) of Theorem~\ref{thm:global with ends}.
\end{proof}

\subsection{Construction for breaking number $\mathbf{b=0}$:}
In this section we construct tubular neighbourhoods $\leftexp{\pm}{\pi}^{\cP_-}_{\cP_+}$ for $b(\cP_-,\cP_+)=0$ as specified in Section~\ref{sec:tub} and find $t_0>0$ such that the induced maps $\phi(\ul{q})$ satisfy Theorem~\ref{thm:global with ends} for $b(\ul{q})=0$ and $0<t\leq t_{0}$.

For pairs of critical points $p_-,p_+\in{\rm Crit}(f)$ with $\cM(p_-,p_+)\neq\emptyset$
the breaking number is $b(p_-,p_+)=0$ iff there exist no broken Morse trajectories from $p_-$ to $p_+$, which is equivalent to the space of unbroken Morse trajectories $\cM(p_-,p_+)$ being compact.
This trajectory space is embedded in the connecting trajectory space by
$$
(\ev_{\ti S^-_{p_-}}\times \ev_{\ti S^+_{p_+}}):\cM(p_-,p_+) \hookrightarrow \Gr^{p_-}_{p_+} .
$$
Its image, $M\subset \Gr^{p_-}_{p_+}$ has a standard tubular neighbourhood diffeomorphism (given by the exponential map for some metric) $\exp : {\rm N}M \supset B \overset{\sim}{\to} W \subset \Gr^{p_-}_{p_+}$ from a neighbourhood $B$ of the zero section in the normal bundle ${\rm N}M\subset{\rm T}\Gr^{p_-}_{p_+}|_M$ to a neighbourhood $W$ of $M$.
Since $M$ is compact and $\Gr^{p_-}_{p_+}(t) \to M$ in the Hausdorff distance as $t\to 0$, we find $t_{p_-,p_+}\in (0,1]$ such that $\Gr^{p_-}_{p_+}(t_{p_-,p_+}) \subset U$. Then with the projection $\Pi_M:NM\to M$ the map
$$
\pi^{p_-}_{p_+}:=(\ev_{\ti S^-_{p_-}}\times \ev_{\ti S^+_{p_+}})^{-1} \circ \Pi_M \circ \exp^{-1} :\;
\Gr^{p_-}_{p_+}(t_{p_-,p_+}) \to \cM(p_-,p_+)
$$
evidently defines a tubular neighbourhood of $\ev_{\ti S^-_{p_-}}\times \ev_{\ti S^+_{p_+}}$ in the sense of Definition~\ref{dfn:tub}.

Next, we have $b(X,p_+)=0$ iff $p_+$ is a maximum, and $b(p_-,X)=0$ iff $p_-$ is a minimum. In those cases the connecting trajectory spaces are
\[
\leftexp{-}{\Gr}^{X}_{p_+}(1)  =  \Psi_{\R_-}(\ti U_1(p_+)) = \ti U_1(p_+) , \qquad
\leftexp{+}{\Gr}^{p_-}_{X}(1) = \Psi_{\R_+}(\ti U_1(p_-)) = \ti U_1(p_-) ,
\]
and we are dealing with the embeddings of the trivial Morse trajectory spaces
\begin{align*}
\ev_- & : \cM(\ti U({p_+}),{p_+})=\{ \g_{p_+}\equiv p_+ :\phantom{-}[0,\infty)\to X \}
\;\overset{\sim}{\longrightarrow}\; \ti B^+_{p_+}=\{p_+\} \subset \ti U_1(p_+) ,  \\
\ev_+ & : \cM(\ti U({p_-}),{p_-})=\{ \g_{p_-}\equiv p_- :(-\infty,0]\to X \}
\;\overset{\sim}{\longrightarrow}\; \ti B^-_{p_-}=\{p_-\} \subset \ti U_1(p_-) .
\end{align*}
We define their tubular neighbourhoods according to \eqref{pi near crit} by
\begin{align*}
\leftexp{-}\pi^{X}_{p_+}\equiv \g_{p_+} :\; \leftexp{-}{\Gr}^{X}_{p_+}(1) \to \cM(\ti U({p_+}),{p_+}) ,
\qquad
\leftexp{+}\pi_{X}^{p_-}\equiv \g_{p_-} :\; \leftexp{+}{\Gr}^{p_-}_{X}(1) \to \cM({p_-},\ti U({p_-})) .
\end{align*}
This constructs all tubular neighbourhoods for breaking number $b=0$ as listed in Section~\ref{sec:tub} with $t'_0 :=\min\{t_{p_-,p_+} \,|\, \cM(p_-,p_+)\neq\emptyset \} \in (0,1]$.

In order for the induced maps $\phi(\ul{q}):\cV_t(\ul{q}) \to [0,2)^k\times \cM(\ul{q})$ for $b(\ul{q})=0$ to satisfy Theorem~\ref{thm:global with ends} it suffices to check the conditions of Lemma~\ref{lem:pi}. Here condition (iii) is trivially satisfied since $b(\ul{q})=0$ does not allow for the insertion of a nontrivial critical point sequence. Condition (iv) holds evidently since $\leftexp{\pm}\pi$ were only defined on $\ti U_1({p_\pm})$.
Finally, the following lemma will provide $t_0\in(0,t_0']$ such that (i) holds.
Note from above that $b(\ul{q})=0$ only for critical point end conditions $\cU_\pm=q_\pm$ or finite end conditions $\cQ_0=\ti U(q_1)$ with $q_1$ a maximum, resp.\ $\cQ_{k+1}=\ti U(q_k)$ with $q_k$ a minimum. Moreover,  $b(\ul{q})=0$ implies compactness of the subset of maximally broken trajectories $\cM(\ul{q})$.

\begin{lem} \label{lem:trans}
Let $S$ be a manifold, $G\subset S$ a submanifold, and $\iota: [0,1)^n \times Z  \to S, (\ul\tau,z)\mapsto \iota_{\ul\tau}(z)$ a smooth family of embedddings $\iota_{\ul\tau}:Z\hookrightarrow S$ such that $\im \iota_{0}\pitchfork G\subset S$ transversely.
Let $e: M \hookrightarrow S$ be an embeddding to $e(M)=\im\iota_0\cap G$, and let $\pi:G\to M$ be a tubular neighbourhood of $e$.
Suppose moreover that $\iota$ is uniformly continuous with respect to the Euclidean distance on $[0,1)^n$ and some metrics on $Z,S$ (compatible with the given topologies).

Then for every compact open subset $K\subset M$ there exists $t>0$ and a neighbourhood $N\subset
G$ of $e(M)$ such that
\begin{equation} \label{claim}
 \im\iota_{\ul\tau} \pitchfork \bigl( \pi^{-1}(m) \cap N \bigr) = 1 \;\text{point} \qquad \forall \; m\in K,\; \ul\tau \in[0,t)^n .
\end{equation}
If $Z$ is compact then this holds with $N=G$.
\end{lem}

Here all manifolds are smooth, finite dimensional, and without boundary; the difficulty lies in allowing noncompactness, which will be needed in the iteration step.
In the present case just $G:=\cN_{t'_0}(\ul{q})\subset {\rm Gr}(\ul{q})$ is noncompact. The base space $K=M:=\cM(\ul{q})$ is compact and in case of finite end conditions $\cQ_0=\ti U(q_1)$ resp.\ $\cQ_{k+1}=\ti U(q_k)$ only contains trajectories $\ul\g$ with $|\ev_-(\ul\g)|=0$ resp.\ $|\ev_+(\ul\g)|=0$, hence $\cM_{t,\ul{\tau}}(\ul{q}) =\cM(\ul{q})$.
Similarly, the embeddings $\iota_{\ul\tau}:=\iota_{\ul{q},\ul\tau}$ to $S:=\ti S(\ul{q})$ have compact domains, in case $\cQ_0=\ti U(q_1)$ resp.\ $\cQ_{k+1}=\ti U(q_k)$ given by $\ti B^+_{q_1}=\{0\}$ resp.\ $\ti B^-_{q_k}=\{0\}$. In the latter cases note that $\iota_{\ul\tau}$ is well defined for $\ul\tau\in[0,1)^k$, so we will obtain the intersection property for transition times in $[0,t)^k$, which contains $I'_t(\ul{q})$.
This finishes the construction in case $b=0$ with $t_0:=\min\{t_0',t\}$.

\begin{proof}[Proof of Lemma~\ref{lem:trans}]
To begin note that the transversality $\im \iota_{0}\pitchfork G=e(M)$ together with the submersion property of $\pi:G\to M$ implies fiber-wise transversality
$$
 \im\iota_{0} \pitchfork \pi^{-1}(m) = e(m)  \qquad \forall \; m\in M .
$$
To show that, after a restriction, this intersection property persists for small $\ul{\tau}\neq 0$, we crucially need compactness of $K$. With that it suffices, given any $k\in K$, to find $t_k>0$ and a neighbourhood $N_k\subset G$ of $e(M)$ such that \eqref{claim} holds on a neighbourhood $U_k$ of $k$.

By assumption, $z_k=\iota_{0}^{-1}(\pi^{-1}(k))$ is a unique point, and $\rd_{z_k}(\pi\circ\iota_0) :\rT_{z_k}Z \to \rd_{k} M$ is an isomorphism.
The implicit function theorem for
$F: \bigl([0,1)^n \times M\bigr) \times Z \to M\times M, \; (\ul\tau, m ; z) \mapsto  (\pi(\iota_{\ul\tau}(z)),m)$ with $(0,k;z_k) \mapsto \Delta_M$ then provides open neighbourhoods $[0,t)^n\times U\subset [0,1)^n \times M$ of $(0,k)$ and $V\subset Z$ of $z_k$ such that $F(\ul\tau, m ; \cdot)\in\Delta_M$ has unique solutions in $V$ for all $(\ul\tau,m)\in U$. That is, $\iota_{\ul\tau}(V) \cap \pi^{-1}(m)$ is a unique point for all $|\ul\tau|<t$ and $m\in U$.
By restricting $F$ to precompact neighbourhoods of $z_k$ we can ensure that $V$ is precompact. Then $\rd_z(\pi\circ\iota_{\ul\tau})\to \rd_z(\pi\circ\iota_0)$ converges uniformly in $z$ as $\ul\tau\to 0$, and hence is surjective for small $|\ul\tau|$. So by choosing $t>0$ smaller we additionally achieve transversality,
\begin{equation}\label{pitch}
\iota_{\ul\tau}(V) \pitchfork \pi^{-1}(m) = 1 \;\text{point}  \qquad \forall\; |\ul\tau|<t, m\in U .
\end{equation}
It remains to trade the restriction to $V\subset Z$ for a restriction to $N\subset G$.
For that purpose we work with open neighbourhoods throughout and write $U'\sqsubset U$ for $U'$ being precompact in $U$ (i.e.\ its closure in $U$ is compact, which yields a positive distance between $U'$ and the complement of $U$).
We can combine a local trivialization of $\pi$ from Remark~\ref{rmk:tub} with the transversality $\im\iota_0\pitchfork G = e(M)$ to find a neighbourhood $U_0\sqsubset U$ of $k$, open balls $B_0\subset\R^{\dim G - \dim M}$ and $C_0\subset\R^{\dim S - \dim G}$, and a diffeomorphism $\phi:U_0 \times B_0 \times C_0 \overset{\sim}{\to} S_0\sqsubset S$ to a neighbourhood of $\iota_0(z_k)=e(k)=\phi(k,0,0)$ such that
$$
\im\iota_0 \cap S_0 =\phi( U_0 ,0, C_0 ), \qquad
G \cap S_0 =\phi( U_0 , B_0 , 0 ), \quad
\phi^*\pi = {\rm pr}_{U_0} ,\quad
\phi^*e = {\rm Id}_{U_0} \times 0 \times 0 .
$$
Now by \eqref{pitch} we have $\iota_0^{-1}(\phi(\overline{U}_1,0,0))\subset\iota_0^{-1}(\pi^{-1}(U_0))\subset V$ for any choice of neighbourhood $U_1\sqsubset U_0$ of $k$.
Since $\iota_0$ is an embedding we then find a neighbourhood $C_1\sqsubset C_0$ of $0$ such that $Z_1:=\iota_0^{-1}(\phi(U_1,0, C_1)))\subset V$ while
$$
\iota_0(Z\setminus Z_1) = \im\iota_0 \setminus \phi(U_1, 0 , C_1)
\subset S\setminus \phi(U_1, B_0, C_1) .
$$
Next, we apply the implicit function theorem again to $F|_{[0,1)^n \times U_1 \times Z_1}$ to find $t'>0$, $V_1\subset Z_1$, and $U_2\sqsubset U_1$ such that $\iota_{\ul\tau}(V_1) \pitchfork \pi^{-1}(m)$ is a unique point for all  $|\ul\tau|<t'$ and $m\in U_2$. Since \eqref{pitch} also holds on $U_2\subset U$ and $V_1\subset Z_1\subset V$, we obtain
\begin{equation}\label{Z1}
\iota_{\ul\tau}(Z_1) \pitchfork \pi^{-1}(m) = 1 \;\text{point}  \qquad \forall\; |\ul\tau|<t, m\in U_2 .
\end{equation}
We pick further neighbourhoods $B_1\sqsubset B_0$ and $C_2\sqsubset C_1$ of $0$ to obtain a precompact neighbourhood $S_2:=\phi(U_2 , B_1 , C_2) \sqsubset \phi(U_1, B_0 , C_1)=:S_1$ of $\iota_0(z_k)$ with $\delta:=d_S(S_2, S\setminus S_1)>0$.
Now uniform continuity provides $t_\delta>0$ such that for all $|\ul\tau|<t_\delta$
\begin{align*}
&\iota_{\ul\tau}(Z\setminus Z_1)\subset B_\delta\bigl(\iota_0(Z\setminus Z_1) \bigr) \subset B_\delta(S\setminus S_1) \subset S\setminus \phi(U_2 ,B_1 ,C_2) \\
\text{and}\quad &
\iota_{\ul\tau}(Z_1) \cap \pi^{-1}(U_2) \subset B_\delta(\iota_0(Z_1)) \cap \pi^{-1}(U_2) = B_\delta(\phi(U_1,0, C_1)) \cap \pi^{-1}(U_2).
\end{align*}
Finally, for sufficiently small $\delta'\in(0,\delta]$ we obtain for all $|\ul\tau|<t_{\delta'}=:t_k$
$$
\iota_{\ul\tau}(Z_1) \cap \pi^{-1}(U_2) \subset  \phi(U_0,B_1, C_0) \cap \pi^{-1}(U_2)
\subset \phi(U_2, B_1 ,0) .
$$
Now $N_k:= \pi^{-1}(M\setminus \overline{U}_1) \cup \phi(U_0, B_1 ,0)\subset G$ is a neighbourhood of $e(M\setminus \overline{U}_1) \cup e(U_0)=e(M)$ and for all $m\in U_2$ and $|\ul\tau|<t'$ we have
$$
\im\iota_{\ul\tau}\cap \pi^{-1}(m) \cap N_k = \im\iota_{\ul\tau}\cap\phi(k,B_1,0)
=\iota_{\ul\tau}(Z_1) \cap \phi(m_0,B_1,0) = \iota_{\ul\tau}(Z_1) \cap \pi^{-1}(m) .
$$
Thus \eqref{claim} on $U_k:=U_2$  follows from \eqref{Z1}. Finally, after finding a finite open cover $K\subset\bigcup U_{k_i}$, the lemma holds with $t:=\min t_{k_i}$ and $N:=\bigcap N_{k_i}$.
If $Z$ is compact then we can moreover choose $t>0$ sufficiently small such that $\im\iota_{\ul\tau}\cap G\subset N$ for all $|\ul\tau|<t$, and hence
$\im\iota_{\ul\tau} \cap \bigl( \pi^{-1}(m) \cap N \bigr) = \im\iota_{\ul\tau} \cap \pi^{-1}(m)$.
\end{proof}

\subsection{Construction for $b\geq 1$ based on construction for $b-1$:}

Let the special global charts in Sections~\ref{sec:k=0} and \ref{sec:k=1 special} be fixed, and for some $b\geq 1$ suppose that we have given a construction of $\phi(\ul{q})=({\rm Id}\times \pi(\ul{q}))\circ(\tau(\ul{q})\times{\rm Ev}(\ul{q}))$ for $b(\ul{q})\leq b-1$ as specified in Section~\ref{sec:tub}, and
satisfying Theorem~\ref{thm:global with ends} for $0<t\leq t_{b-1}$.
Then the goal of this iteration step is to construct tubular neighbourhoods $\leftexp{(\pm)}{\pi}^{\cP_-}_{\cP_+}$ for $b(\cP_-,\cP_+)=b$ as specified in Section~\ref{sec:tub}, and find $t_b>0$ such that the induced maps $\phi(\ul{q})$ satisfy Theorem~\ref{thm:global with ends} for $b(\ul{q})\leq b$ and $0<t\leq t_{b}$. By Lemma~\ref{lem:pi} it suffices to satisfy conditions (i), (iii), (iv). Hence we start from the formulas
$$
\leftexp{\{\genfrac{}{}{0pt}{3}{}{\pm}\}}{\hat\pi}^{\cP_-}_{\cP_+} \circ
\left\{ \begin{smallmatrix} \ev_-\times\ev_+ \\ \ev_\pm \end{smallmatrix} \right\} \big|_{\cV_{t_b}(\ul{\ti q}')_0}
:=\phi(\ul{q}')^{-1}\circ {\rm Pr}_{\ul{q}'} \circ \phi(\ul{\ti q}')
, \qquad
\leftexp{\pm}\pi^{\cP_-}_{\cP_+}|_{\ti U_1(p_+)} :=\ev_\pm^{-1}\circ {\rm pr}_{\ti B^\mp_{p_\mp}}
$$
for nontrivial critical point sequences $\ul{q}'=(\cP_-\supset\cQ_0',q'_1,\ldots,q'_\ell,\cQ'_{\ell+1}\subset\cP_+)$ and the associated
$\ul{\ti q}'= \left(
\left\{ \begin{smallmatrix}  \ti U(p_-), p_-  &; \cP_-=p_- \\ \cQ_0' &; \cP_-=X \end{smallmatrix} \right\}
,q_1', \ldots, q_\ell',
\left\{ \begin{smallmatrix} p_+,\ti U(p_+) &; \cP_+=p_+ \\ \cQ_{\ell+1}' &; \cP_+=X \end{smallmatrix} \right\}
\right)$
to define maps
\begin{align*}
\hat\pi^{p_-}_{p_+} : \; \,\;\qquad\qquad {\textstyle \bigcup_{\ul{q}'}} {\Gr}^{p_-}_{p_+}(t_{b-1},\ul{q}') &\;\longrightarrow\;  {\textstyle \bigcup_{\ul{q}'}} \cV_{t_{b-1}}(\ul{q}')_0 \qquad\;\; \;\; \subset\; \cM(p_-,p_+)  , \\
\leftexp{-}{\hat\pi}^X_{p_+}  : \; \ti U_1(p_+)\cup {\textstyle \bigcup_{\ul{q}'}} \leftexp{-}{\Gr}^X_{p_+}(t_{b-1},\ul{q}') &\;\longrightarrow\;  \ev_-^{-1}(\ti B^+_{p_+})\cup{\textstyle \bigcup_{\ul{q}'}} \cV_{t_{b-1}}(\ul{q}')_0 \; \subset\; \cM(X,p_+)  , \\
\leftexp{+}{\hat\pi}^{p_-}_X  : \; \ti U_1(p_-)\cup {\textstyle \bigcup_{\ul{q}'}} \leftexp{+}{\Gr}^{p_-}_X(t_{b-1},\ul{q}') &\;\longrightarrow\;   \ev_+^{-1}(\ti B^-_{p_-})\cup{\textstyle \bigcup_{\ul{q}'}} \cV_{t_{b-1}}(\ul{q}')_0 \; \subset\; \cM(p_-,X)
\end{align*}
with the union over critical point sequences as above, and on the domains
\begin{align*}
\Gr^{p_-}_{p_+}(t_b,\ul{q'}) &:= \bigl({\rm ev}_{\ti{S}^-_{p_-}}\times {\rm ev}_{\ti{S}^+_{p_+}}\bigr)\bigl( \cM(\ti{S}^-_{p_-},\ti{S}^+_{p_+}) \cap \cV_{t_b}(\ul{\ti q'})_0 \bigr) , \\
\leftexp{-}\Gr^X_{p_+}(t_b,\ul{q'}) &:= \ev_-(\cM(X,\ti{S}^+_{p_+}) \cap \cV_{t_b}(\ul{\ti q'})_0) , \\
\leftexp{+}\Gr^{p_-}_X(t_b,\ul{q'}) &:= \ev_+(\cM(\ti{S}^-_{p_-},X) \cap\cV_{t_b}(\ul{\ti q'})_0) .
\end{align*}
If we define $\leftexp{(\pm)}{\pi}^{\cP_-}_{\cP_+}$ by extension of
$\leftexp{(\pm)}{\hat\pi}^{\cP_-}_{\cP_+}|_{\ti U_{t_b}(p_\pm)\cup {\textstyle \bigcup_{\ul{q}'}} \leftexp{(\pm)}{\Gr}^{\cP_-}_{\cP_+}(t_{b},\ul{q}')}$ to $\leftexp{(\pm)}{\Gr}^{\cP_-}_{\cP_+}(t_b)$, then (iii) and the first part of (iv) are automatically satisfied. In fact, the following lemma shows that this definition is consistent with all conditions on the tubular neighbourhoods.

\begin{lem}\label{lem:pihat}
For each $b(\cP_-,\cP_+)=b$ the maps $\hat\pi^{p_-}_{p_+}$, $\leftexp{+}{\hat\pi}^{p_-}_X$, $\leftexp{+}{\hat\pi}^{p_-}_X$ are well defined  tubular neighbourhoods of $\ev_{\ti S^-_{p_-}}\times \ev_{\ti S^+_{p_+}}$, $\ev_-$, resp.\ $\ev_+$ restricted to the above subdomains of $\cM(\cP_-,\cP_+)$, and satisfy the preimage condition in Lemma~\ref{lem:pi}~(iv).
For $0<t<t_{b-1}$ they restrict to maps ${\textstyle \bigcup_{\ul{q}'}} \leftexp{(\pm)}{\Gr}^{\cP_-}_{\cP_+}(t,\ul{q}') \to {\textstyle \bigcup_{\ul{q}'}} \cV_t(\ul{q}')_0$.
Moreover, the product
$\hat\pi(\ul{q}) = \leftexp{-}{\hat\pi}^{\cU_-}_{q_1} \times {\hat\pi}^{q_1}_{q_2}\ldots  \times  \leftexp{+}{\hat\pi}_{\cU_+}^{q_k}$ for any $b(\ul{q})=b$ satisfies the intersection condition in Lemma~\ref{lem:pi}~(i) for $\ul\tau\in I'_{t_{b-1}}(\ul{q})$ and $\ul{\g}\in  \cM(\ul{q})\cap \bigcup_{\genfrac{}{}{0pt}{3}{\ul{Q}\supset\ul{q}}{b(\ul{Q})<b}}  \cV_{t_{b-1}}(\ul{Q})\subset \bM(\cU_-,\cU_+)$.
\end{lem}

\begin{proof}
We begin by noting that the nontrivial critical point sequences have breaking number $b(\ul{\ti q}')=b(\ul q')<b(\cP_-,\cP_+)=b$, hence by the iteration hypothesis we can work with the charts $\phi(\ul{q}')$ and $\phi(\ul{\ti q}')$, satisfying the properties of Theorem~\ref{thm:global with ends}.

In order to see that $\hat\pi^{p_-}_{p_+}$ is well defined we have to check consistency of the definitions at a fixed $(z^-,z^+)\in\Gr^{p_-}_{p_+}(t_b)$ for different critical point sequences $\ul{q}'=(p_-,\ldots,p_+)$.
Note that $\ul{Q}^{z_\pm} = \{p\in{\rm Crit}(f) \,|\, \Psi_\R(z^\pm) \cap \ti U_t(p) \neq \emptyset\}$ defines a critical point sequence in ${\rm Critseq}(f;p_-,p_+)$ such that $(z^-,z^+)\in\Gr^{p_-}_{p_+}(t,\ul{Q}^{z_\pm})$. In fact, it is maximal in the sense that if $(z^-,z^+)\in\Gr^{p_-}_{p_+}(t,\ul{q}')$ then $\ul{q}'\subset\ul{Q}^{z_\pm}$. In this situation we actually have $(z^-,z^+)\in\Gr^{p_-}_{p_+}(t,\ul{Q}')$ for each intermediate critical point sequence $\ul{Q}'=(p_-,\ldots,p_+)$ with $\ul{q}'\subset\ul{Q}'\subset\ul{Q}^{z_\pm}$.
Now arguing step by step, it suffices to check the identity
$$
\phi(\ul{Q}')^{-1} \circ {\rm Pr}_{\ul{Q}'} \circ \phi(\ul{\ti Q'})
= \phi(\ul{q}')^{-1} \circ {\rm Pr}_{\ul{q}'} \circ \phi(\ul{\ti q}')
\qquad\text{at}\; \g_{z^\pm}:=(\ev_-\times\ev_+)^{-1}(z^-,z^+)\in\cM(\ti S^-_{p_-},\ti S^+_{p_+})
$$
for pairs $\ul{q}'\subset\ul{Q}'$ where $\ul{Q}'$ is obtained from $\ul{q}'$ by inserting a critical point sequence $\ul{q}''=(q'_i=q''_0,\ldots,q''_{k+1}=q'_{i+1})$ at a unique $i$.
In each step the breaking numbers $b(\ul{Q}')=b(\ul{\ti Q'})\leq b(\ul{q}')=b(\ul{\ti q}')<b$ are strictly less than $b(p_-,p_+)=b$, so the identity above follows, after applying $\phi(\ul{q}')$ to both sides, from the associativity relations
$\phi(\ul{Q}')= \bigl({\rm Id}\times \phi(\ul{q}'') \times {\rm Id} \bigr) \circ \phi(\ul{q}')$ and
$\phi(\ul{\ti Q'})= \bigl({\rm Id}\times \phi(\ul{q}'') \times {\rm Id} \bigr) \circ \phi(\ul{\ti q}')$.
That is, we have at $\g_{z^\pm}$
\begin{align*}
\phi(\ul{q}') \circ \phi(\ul{Q}')^{-1} \circ {\rm Pr}_{\ul{Q}'} \circ \phi(\ul{\ti Q'})
=
\bigl({\rm Id}\times \phi(\ul{q}'')^{-1} \times {\rm Id} \bigr) \circ {\rm Pr}_{\ul{Q}'} \circ \bigl({\rm Id}\times \phi(\ul{q}'') \times {\rm Id} \bigr) \circ \phi(\ul{\ti q}')
=
{\rm Pr}_{\ul{q}'} \circ \phi(\ul{\ti q}')
\end{align*}
since ${\rm Pr}_{\ul{q}'}$ and ${\rm Pr}_{\ul{Q}'}$ merely project out the first two and last two factors in
$$
\ti B^+_{p_-} \times [0,2) \times \cM(p_-,q'_1) \times [0,1) \times \ldots \cM(q'_i,q'_{i+1}) \ldots [0,1)\times \cM(q'_\ell,p_+) \times [0,2) \times \ti B^-_{p_+}  ,
$$
$$
\ti B^+_{p_-} \times [0,2) \times \cM(p_-,q'_1) \times [0,1) \times \ldots
 [0,1)^k \times \cM(\ul{q}'') \ldots [0,1)\times \cM(q'_\ell,p_+) \times [0,2) \times \ti B^-_{p_+} ,
$$
while $\phi(\ul{q}'')^{-1} \circ \phi(\ul{q}'') = {\rm Id}_{ \cM(q'_i,q'_{i+1})}$ cancels out on a factor not involved in the projections. Thus we have proven consistency of the definition
$\hat\pi^{p_-}_{p_+} := \phi(\ul{q}')^{-1} \circ {\rm Pr}_{\ul{q}'} \circ \phi(\ul{\ti q}') \circ (\ev_-\times\ev_+)^{-1}$.

Next, by the explicit construction of transition times and tubular neighbourhoods near critical points in \eqref{pi near crit} we have for all
$\g\in \cM(\ti S^-_{p_-},\ti S^+_{p_+})\cap\cV_t(\ul{\ti q}')_0$
$$
\phi(\ul{\ti q}')(\g) = \Bigl( 1, \tau(\ul{q}')(\g), 1 ; {\rm pr}_{\ti B^+_{p_-}}(\ev_-(\g)) ,\resizebox{.35\hsize}{!}
{$\Bigl(\pi^{q'_i}_{q'_{i+1}} \bigl({\rm ev}_{\ti{S}^-_{q'_i}}(\g), {\rm ev}_{\ti{S}^+_{q'_{i+1}}}(\g)\bigr)\Bigr)_{i=0,\ldots,\ell}$ }, {\rm pr}_{\ti B^-_{p_+}}(\ev_+(\g) \Bigr)
$$
and conversely
$\phi(\ul{\ti q}')^{-1} \bigl( 1 ,(0,t)^\ell , 1 ; \ast , \ast , \ast \bigr) \subset \cM(\ti S^-_{p_-},\ti S^+_{p_+})\cap\cV_t(\ul{\ti q}')_0$.
Hence the chart restricts to a diffeomorphism
$$
\phi(\ul{\ti q}') : \;
\cM(\ti S^-_{p_-},\ti S^+_{p_+})\cap\cV_t(\ul{\ti q}')_0 \;\to\;
\{1\} \times (0,t)^\ell \times \{1\} \times \tfrac t2 \ti B^+_{p_-} \times \cM(\ul{q}') \times \tfrac t2\ti B^-_{p_+} .
$$
Since $\phi(\ul q)$ also restricts to a diffeomorphism to $(0,t)^\ell  \times \cM(\ul{q}')$, this already shows that
$
\phi(\ul{q}')^{-1} \circ {\rm Pr}_{\ul{q}'} \circ \phi(\ul{\ti q}') : \;
\cM(\ti{S}^-_{p_-},\ti{S}^+_{p_+}) \cap \cV_{t_b}(\ul{\ti q'})_0
\;\to\; \cV_{t_b}(\ul{q'})_0 \subset
\cM(p_- ,p_+)
$
is a smooth submersion.
In fact, it is a tubular neighbourhood of the restriction
$$
\rho^{p_-}_{p_+}
=\bigl(\ev_-\times\ev_+\bigr)^{-1}\circ\bigl(\ev_{\ti{S}^-_{p_-}}\times\ev_{\ti{S}^+_{p_+}}\bigr)
: \; \cM(p_- ,p_+) \supset\ \cV_{t_b}(\ul{q'})_0 \hookrightarrow \cM(\ti{S}^-_{p_-}, \ti{S}^+_{p_+}).
$$
To see this it remains to check ${\rm Pr}_{\ul{q}'} \circ \phi(\ul{\ti q}')\circ \rho^{p_-}_{p_+} = \phi(\ul{q}')$, which by the above expression for $\phi(\ul{\ti q}')$ reduces to identifying the factors $\leftexp{-}{\pi}^{p_-}_{q'_1}\circ\ev_{\ti S^+_{q'_1}}$ and $\leftexp{+}\pi^{q'_\ell}_{p_+}\circ \ev_{\ti S^-_{q'_\ell}}$ of $\phi(\ul{q}')$ with $\pi^{q'_i}_{q'_{i+1}} \circ \bigl({\rm ev}_{\ti{S}^-_{q'_i}}\times {\rm ev}_{\ti{S}^+_{q'_{i+1}}}\bigr) \circ \rho^{p_-}_{p_+}$ for $i=0$ and $i=\ell$.
Here the effect of the restriction is $\bigl({\rm ev}_{\ti{S}^-_{p_-}}\times {\rm ev}_{\ti{S}^+_{q'_1}} \bigr) \circ \rho^{p_-}_{p_+} = {\rm ev}_-\times {\rm ev}_{\ti{S}^+_{q'_1}}$ resp.\
$\bigl({\rm ev}_{\ti{S}^-_{q'_\ell}}\times {\rm ev}_{\ti{S}^+_{p_+}} \bigr)\circ \rho^{p_-}_{p_+}
={\rm ev}_{\ti{S}^-_{q'_\ell}}\times {\rm ev}_+$, so the required identities follow from the pullback definitions
$\leftexp{-}\pi^{p_-}_{q'_1}(z^+) =\pi^{p_-}_{q'_1}(z^-,z^+)$
resp.\
$\leftexp{+}\pi^{q'_\ell}_{p_+}(z^-)=\pi^{q'_\ell}_{p_+}(z^-,z^+)$.
Since $\hat\pi^{p_-}_{p_+}$ is defined from this tubular neighbourhood of $\rho^{p_-}_{p_+}$ by pullback with the diffeomorphisms $\bigl(\ev_-\times \ev_+\bigr)^{-1} : \Gr^{p_-}_{p_+} \to \cM(\ti S^-_{p_-},\ti S^+_{p_+})$, it indeed is a tubular neighbourhood of $\ev_{\ti S^-_{p_-}}\times \ev_{\ti S^+_{p_+}}$.

In the definition of $\leftexp{-}{\hat\pi}^{X}_{p_+}$ we similarly use the explicit construction of transition times and tubular neighbourhoods near $p_+$ to see that for all $\g\in \cM(X,\ti S^+_{p_+})\cap\cV_t(\ul{\ti q}')_0$
$$
\phi(\ul{\ti q}')(\g) = \bigl( \tau(\ul{q}')(\g), 1 ;
\resizebox{.5\hsize}{!}
{$\leftexp{-}\pi^X_{q'_1}\bigl(\ev_-(\g)\bigr) ,
\bigl(\pi^{q'_i}_{q'_{i+1}} \bigl({\rm ev}_{\ti{S}^-_{q'_i}}(\g), {\rm ev}_{\ti{S}^+_{q'_{i+1}}}(\g)\bigr)\bigr)_{i=1,\ldots,\ell}$ }, {\rm pr}_{\ti B^-_{p_+}}(\ev_+(\g) \bigr)
$$
and conversely
$\phi(\ul{\ti q}')^{-1} \bigl( I_t(\ul q') , 1 ; \ast , \ast , \ast \bigr) \subset \cM(X,\ti S^+_{p_+})\cap\cV_t(\ul{\ti q}')_0$.
Hence the chart restricts to a diffeomorphism
$$
\phi(\ul{\ti q}') : \;
\cM(X,\ti S^+_{p_+})\cap\cV_t(\ul{\ti q}')_0 \;\longrightarrow\;
{\textstyle\bigcup_{\ul\tau\in I_t(\ul q') \cap (0,\infty)^\ell}} \{\ul\tau\}\times \{1\} \times \cM_{t,\ul\tau}(\ul{q}') \times \tfrac t2\ti B^-_{p_+} .
$$
Since $\phi(\ul q)$ also restricts to a diffeomorphism to $\bigcup_{\ul\tau\in I_t(\ul q') \cap (0,\infty)^\ell} \{\ul\tau\} \times \cM_{t,\ul\tau}(\ul{q}')$, this proves that
$
\phi(\ul{q}')^{-1} \circ {\rm Pr}_{\ul{q}'} \circ \phi(\ul{\ti q}') : \;
\cM(X,\ti{S}^+_{p_+}) \cap \cV_{t_b}(\ul{\ti q'})_0
\;\longrightarrow\; \cV_{t_b}(\ul{q'})_0 \subset
\cM(X, p_+)
$
is a smooth submersion. In fact, the same identities as before prove that it is a tubular neighbourhood of the restriction
$\ev_-^{-1}\circ \ev_- : \;
\cM(X,p_+) \;\supset \; \cV_{t_b}(\ul{q'})_0 \;\hookrightarrow\; \cM(X, \ti{S}^+_{p_+}) $ .
Assuming for now that $\leftexp{-}{\hat\pi}^X_{p_+}$ is well defined on $\bigcup_{\ul{q}'}\leftexp{-}{\Gr}^X_{p_+}(t_{b-1},\ul{q}')$, it is the pullback of this tubular neighbourhood by the diffeomorphism $\ev_-^{-1}: \leftexp{-}{\Gr}^X_{p_+}(1) \to \cM(X, \ti{S}^+_{p_+})$, and hence a tubular neighbourhood of $\ev_-$.
We may extend this by $\leftexp{-}{\hat\pi}^{X}_{p_+}|_{\ti U_1(p_+)}:={\rm pr}_{\ti B^+_{p_+}}$ in the identification $\ev_-:\cM(X,p_+) \overset{\sim}{\to} \ti B^+_{p_+}$ , where the domains $\leftexp{-}{\Gr}^X_{p_+}(t_{b-1},\ul{q}')$ do not intersect $\ti U_1(p_+)$ since they are subsets of $\Psi_{\R_-}(\ti U_{t_b}(q'_1))$ for $f(q'_1)>f(p_+)$. In particular this separation of domains ensures condition (iv), that is
 $\bigl(\leftexp{-}{\hat\pi}^{X}_{p_+}\bigr)^{-1}(\ti B^+_{p_+}) = \ti U_1(p_+)$.
The analogous construction of $\leftexp{+}{\hat\pi}_{X}^{p_-}$ provides a tubular neighbourhood of
$
\ev_+:\; \cM(p_-,X)  \;\supset\; \ti B^-_{p_-}\cup \bigcup_{\ul{q}'} \cV_{t_{b-1}}(\ul{q}')_0 \; \to \;
 \leftexp{+}{\Gr}^{p_-}_X(t_{b-1},\ul{q}')
$.

Finally, we check consistency of definitions for $\leftexp{-}\pi^{X}_{p_+}$ (and analogously for $\leftexp{+}\pi_{X}^{p_-}$) at $z^-\in\leftexp{-}\Gr^X_{p_+}(t)$ for different critical point sequences $\ul{q}',\ul{q}''$.
If these have the same type of end conditions $\cQ_0'=X\setminus\overline{U(q'_1)}$,  $\cQ_0''=X\setminus\overline{U(q''_1)}$ then the same argument as above applies. It remains to check consistency for the same critical points but different end conditions. For $\ul{q}'=(\ti U(q'_1),q'_1,\ldots,q'_\ell,p_+)$, $\ul{q}''=(X\setminus \overline{U(q'_1)},q'_1,\ldots,q'_\ell,p_+)$ we have
\begin{align*}
\phi(\ul{q}') \circ \phi(\ul{q}'')^{-1} \circ {\rm Pr}_{\ul{q}''} \circ \phi(\ul{\ti q''})
\;=\;
\bigl( R_t ^{-1}\times {\rm Id} \bigr) \circ {\rm Pr}_{\ul{q}''} \circ \bigl( R_t \times {\rm Id} \bigr) \circ \phi(\ul{\ti q}')
\;= \;
{\rm Pr}_{\ul{q}'} \circ \phi(\ul{\ti q}')
\end{align*}
at $\g\in\cV_t(\ul{q}')_0\cap\cV_t(\ul{q}'')_0\subset \ev_-^{-1}(\ti U_t(q'_1)\setminus \overline{U(q'_1)})$
since ${\rm Pr}_{\ul{q}''}={\rm Pr}_{\ul{q}'}$ both project out the last two factors in
$
[0,2) \times  \ti B^+_{q'_1} \times [0,1) \times \ldots  \cM(q'_\ell,p_+) \times [0,2) \times \ti B^-_{p_+}
$,
and $R_t : [0,2) \times \ti B^+_{q'_1} \supset \{ (E,x) \,|\, E|x| < t\Delta\} \to [0,t) \times \ti B^+_{q'_1}$ is a rescaling on the first two factors.

With all properties of $\hat\pi^{p_-}_{p_+}$, $\leftexp{+}{\hat\pi}^{p_-}_X$, $\leftexp{+}{\hat\pi}^{p_-}_X$ established, let us start analyzing the fibers.
For $\g\in\cM(p_-,p_+)$ and $\ul{q'}=(p_-=q'_0,q'_1,\ldots, q'_{\ell+1}=p_+)$ let us denote $\phi(\ul{q}')(\g)=(\ul\tau^\g,\ul\eta^\g) \in (0,2)^\ell \times \cM(\ul{q}')$, then it is easiest to read off the fiber in the formulation \eqref{reqqq} as
\begin{align*}
\bigl(\hat\pi^{p_-}_{p_+}\bigr)^{-1}(\g)
&= (\ev_-\times\ev_+) \left\{
\delta\in\cM(\ti S^-_{p_-},\ti S^+_{p_+}) \left|\;
\begin{aligned}
&\ul\tau^\g = \bigl( \tau_{q'_i}(\delta) \bigr)_{i=1,\ldots,\ell}   \\
&\ul\eta^\g = \bigl(\pi^{q'_i}_{q'_{i+1}} \bigl({\rm ev}_{\ti{S}^-_{q'_i}}\times {\rm ev}_{\ti{S}^+_{q'_{i+1}}}\bigr) (\delta) \bigr)_{i=0,\ldots,\ell}
\end{aligned}
\right.\right\} \\
&= \Pi_{\ul{q}'} \left(
 \left(\bigl(\pi^{q'_i}_{q'_{i+1}}\bigr)_{i=0,\ldots,\ell}\right)^{-1}(\ul\eta^\g) \cap
\left( \ti S^-_{p_-} \times {\textstyle\prod_{i=1}^{\ell}}
\im\iota_{q_i,\tau^\g_i} \times \ti S^+_{p_+} \right)
\right),
\end{align*}
where $\Pi_{\ul{q}'}: \prod_{i=0}^{\ell} \Gr^{q'_i}_{q'_{i+1}} \to \ti S^-_{p_-} \times \ti S^+_{p_+}$ is the projection to the outside factors, and
$$
\iota_{q,\tau} :\; S^-_{q} \times S^+_{q} \;\to\; \ti S^-_{q} \times \ti S^+_{q} ,\qquad
(x,y)\mapsto \bigl((x,\tau y), (\tau x, y) \bigr)
$$
for $\tau\in[0,1)$ are the slices of fixed transition time of the embedding $(\ev_-\times\ev_+):\bM_q\to  \ti S^-_{q} \times \ti S^+_{q}$ of the local trajectory space in the coordinates \eqref{Mp hom}.
The fibers $\bigl(\leftexp{-}{\hat\pi}^{\cP_-}_{p_+}\bigr)^{-1}(\g)=\ev_-\bigl\{\ldots\bigr\}$ and $\bigl(\leftexp{+}{\hat\pi}^{p_-}_{\cP_+}\bigr)^{-1}(\g)=\ev_+\bigl\{\ldots\bigr\}$ have analogous expressions involving special terms $\leftexp{-}\pi^{\cP_-}_{q'_1}\circ\ev_{\cP_-,q'_1}$ resp.\ $\leftexp{+}\pi_{\cP_+}^{q'_\ell}\circ\ev_{q'_\ell,\cP_+}$ and in case $\cP_-=X$ resp.\ $\cP_+=X$ embeddings
$\R_-\times S^-_{q} \times S^+_{q} \to X \times \ti S^+_{q}$ or $\ti B^-_{q} \times S^+_{q} \to X \times \ti S^+_{q}$, resp.\  $S^-_{q} \times S^+_{q} \times \R_+ \to \ti S^-_{q} \times X$ or $S^-_{q} \times \ti B^+_{q} \to \ti S^-_{q} \times X$ encoding the local trajectory spaces with ends in $X\setminus\overline{U(q)}$ or $\ti U(q)$ as in the definition of $\iota_{\ul{q},\ul\tau}$ in Section~\ref{sec:ev}.

Now for a critical point sequence $\ul{q}=(q_-=q_0,q_1,\ldots, q_{k+1}=q_+)$ with $b(\ul{q})=b$ let us view $\cM(\ul{q})=\cM(q_-,q_1)\times\ldots\times\cM(q_k,q_+)\subset\bM(q_-,q_+)$ as stratum of a compactified Morse trajectory space. Then the product $\hat\pi(\ul{q}) := \leftexp{-}{\hat\pi}^{q_-}_{q_1} \times {\hat\pi}^{q_1}_{q_2}\ldots  \times  \leftexp{+}{\hat\pi}_{q_+}^{q_k}$ defines a tubular neighbourhood (defined on the product of domains)
$$
\hat\pi(\ul{q}) : \;
\Gr(\ul{q}) \;\supset\; {\rm dom}\,\leftexp{-}{\hat\pi}^{q_-}_{q_1} \times \ldots  \times{\rm dom}\,\leftexp{+}{\hat\pi}_{q_+}^{q_k} \;\longrightarrow\;
\cM(\ul{q}) \cap
{  \textstyle\bigcup_{\genfrac{}{}{0pt}{3}{\ul{Q}\supset\ul{q}}{b(\ul{Q})<b}}   }  \cV_{t_{b-1}}(\ul{Q})
$$
of $\bigl({\rm ev}_{\ti{S}^-_{q'_i}}\times {\rm ev}_{\ti{S}^+_{q'_{i+1}}}\bigr)_{i=0,\ldots,\ell}: \cM(\ul{q})\to \Gr(\ul{q})$, restricted to the union of domains for critical point sequences $\ul{Q}\in{\rm Critseq}(f;q_-,q_+)$ with smaller breaking number $b(\ul{Q})<b$ and containing $\ul{q}$.
More precisely, we can write any such $\ul{Q}=\ul{q}\cup\bigcup_{j=0}^k \ul{q}^{j}$ as union of
$\ul{q}$ with (potentially trivial) critical point sequences $\ul{q}^{j}\in{\rm Critseq}(f;q_j,q_{j+1})$.
From the above we can then read off the fiber over $\ul{\g}=(\g_0,\ldots,\g_k)\in \cV_{t_{b-1}}(\ul{Q})\cap\cM(\ul{q})$ with $\phi(\ul{q}^{j})(\g_j)=(\ul\tau^j,\ul\eta^j)$ as
\begin{align*}
\hat\pi(\ul q)^{-1}(\ul\g)
&= \Pi^{\ul{Q}}_{\ul{q}} \left(  \pi({\ul{Q}})^{-1}(\ul\eta^0,\ldots,\ul\eta^k) \cap
\left( \iota(\ul{q}^{0},\ul\tau^0) \times \widetilde S_{q_1} \times \iota(\ul{q}^{1},\ul\tau^1) \ldots
\times \widetilde S_{q_k} \times \iota(\ul{q}^{k},\ul\tau^k)
\right)
\right)
\end{align*}
with the natural projection $\Pi^{\ul{Q}}_{\ul{q}}:\ti S(\ul{Q}) \to \ti S(\ul{q})$ and the shorthands
$\widetilde S_q:= \ti S^-_{q_-} \times \ti S^+_{q_+}$, and $\iota\bigl((q_0,q_1,\ldots,q_\ell,q_{\ell+1}),(\tau_1,\ldots,\tau_\ell)\bigr)=\im (\iota_{q_1,\tau_1}\times\ldots\times\iota_{q_\ell,\tau_\ell})$.
Let us denote $\ul\eta_{\ul\g}:=(\ul\eta^0,\ldots,\ul\eta^k)\in\cM(\ul{Q})$ and $\ul{T}_{\ul{\g},\ul\s}:=(\ul\tau^0,\s_1,\ul\tau^1, \ldots, \s_k,\ul\tau^k)$ for any $\ul{\s}\in[0,1)^k$, then the image of the embedding of all local trajectory spaces for $\ul{Q}$, as introduced in Section~\ref{sec:ev}, is
$$
\im\iota_{\ul{Q},\ul{T}_{\ul{\g},\ul\s}} \;=\;
\iota(\ul{q}^{0},\ul\tau^0) \times \im\iota_{q_1,\s_1} \times \iota(\ul{q}^{1},\ul\tau^1) \ldots
\times \im\iota_{q_k,\s_k} \times \iota(\ul{q}^{k},\ul\tau^k) .
$$
Comparing this with $\iota_{\ul{q},\ul{\s}}= \iota_{q_1,\s_1} \times  \ldots \times \im\iota_{q_k,\s_k}$ we obtain for every $\ul{\g}\in\cV_{t_{b-1}}(\ul{Q})\cap\cM(\ul{q})$ and $\ul{\s}\in[0,t_{b-1})^k$
$$
\hat\pi(\ul q)^{-1}(\ul\g) \cap \im\iota_{\ul{q},\ul{\s}}
\;=\; \Pi^{\ul{Q}}_{\ul{q}} \Bigl(  \pi({\ul{Q}})^{-1}(\ul\eta_{\ul\g}) \cap \im\iota_{\ul{Q},\ul{T}_{\ul{\g},\ul\s}}
\Bigr) .
$$
This is a unique point by the intersection property \eqref{req2} for the fibers of $\pi({\ul{Q}})$ with ${b({\ul{Q}})<b}$.
Moreover, this intersection is transverse since from
$ \pi({\ul{Q}})^{-1}(\ul\eta_{\ul\g}) \pitchfork \im\iota_{\ul{Q},\ul{T}_{\ul{\g},\ul\s}}$
we obtain
\begin{align*}
\rT \ti S(\ul{Q})
&= \ker \rd\pi(\ul Q)
\oplus
\bigl(
\rT \iota(\ul{q}^{0},\ul\tau^0) \times \;\{0\}\; \times \rT\iota(\ul{q}^{1},\ul\tau^1) \times \ldots \times \;\{0\}\; \times \rT\iota(\ul{q}^{k},\ul\tau^k)
\bigr) \\
&\qquad\qquad\quad\;\; \oplus
\bigl(
\;\;\;\;\; \{0\} \; \times \im\rd\iota_{q_1,\s_1} \times \;\{0\}\; \times  \ldots \times \im\rd\iota_{q_k,\s_k} \times\; \{0\} \;\;\;\;\;\,
\bigr) \\
&=
\Bigl( \ker \rd\pi(\ul Q) \cap
\bigl(
\rT \iota(\ul{q}^{0},\ul\tau^0) \times \rT\widetilde S_{q_1} \times \rT\iota(\ul{q}^{1},\ul\tau^1) \times \ldots \times \rT\widetilde S_{q_k} \times \rT\iota(\ul{q}^{k},\ul\tau^k)
\bigr)
\Bigr) \\
&\quad
+\ker\rd\Pi^{\ul{Q}}_{\ul{q}}
+
\bigl(
\;\;\;\;\; \{0\} \; \times \im\rd\iota_{q_1,\s_1} \times \;\{0\}\; \times  \ldots \times \im\rd\iota_{q_k,\s_k} \times\; \{0\} \;\;\;\;\;\,
\bigr).
\end{align*}
Here the direct sum implies
$\rT \iota(\ul{q}^{0},\ul\tau^0)^C \times \;\{0\}\; \times \ldots \times \;\{0\}\; \times \rT\iota(\ul{q}^{k},\ul\tau^k)^C \subset \ker \rd\pi(\ul Q)$
for some complements of $\rT \iota(\ul{q}^{0},\ul\tau^0)$.
Projection by $\rd\Pi^{\ul{Q}}_{\ul{q}}$ then yields the claim,
\begin{align*}
\rT \ti S(\ul{q}) &=
\rd\Pi^{\ul{Q}}_{\ul{q}} \bigl( \ker \rd\pi(\ul Q) \cap
\bigl(
\rT \iota(\ul{q}^{0},\ul\tau^0) \times \rT\widetilde S_{q_1} \times \ldots\rT\iota(\ul{q}^{k},\ul\tau^k) \bigr) \bigr)
+  \im\rd\iota_{q_1,\s_1} \times \ldots  \im\rd\iota_{q_k,\s_k}
 \\
&= \rT\bigl(\hat\pi(\ul q)^{-1}(\ul\g)\bigr) + \rT\bigl(\im\iota_{\ul{q},\ul{\s}}\bigr) .
\end{align*}
For general end conditions $\ul{q}=(\cU_-\supset\cQ_0,q_1,\ldots, \cQ_{k+1}\subset\cU_+)$ with $b(\ul{q})=b$ the same arguments show that the fibers of the product $\hat\pi(\ul{q}) = \leftexp{-}{\hat\pi}^{\cU_-}_{q_1} \times {\hat\pi}^{q_1}_{q_2}\ldots  \times  \leftexp{+}{\hat\pi}_{\cU_+}^{q_k}$ satisfy the intersection condition for any $\ul{\g}\in\cM(\ul{q})\cap\bigcup_{\ul{Q}}\cV_{t_{b-1}}(\ul{Q})$ and
$\ul{\s}\in I'_{t_{b-1}}(\ul{q})$, as claimed.
\end{proof}

We will use the following lemma to extend each $\leftexp{(\pm)}{\hat\pi}^{\cP_-}_{\cP_+}$ to a full tubular neighbourhood of the evaluation embedding $\cM(\cP_-,\cP_+)\hookrightarrow \leftexp{(\pm)}\Gr^{\cP_-}_{\cP_+}$.

\begin{lem} \label{lem:ext}
Let $e: M \hookrightarrow G$ be an embeddding between smooth manifolds, $V\subset M$ an open subset such that $M\setminus V$ is compact, and suppose that $\hat\pi:\hat G \to V$ is a tubular neighbourhood of $e|_V$ defined on an open neighbourhood $\hat G\subset G$ of $e(V)$.
Then for any open subset $V'\subset M$ such that $\overline{V'} \subset V$ and $\overline{\hat\pi^{-1}(V')}\cap e(M\setminus V) =\emptyset$ there exists a tubular neighbourhood $\pi:G \supset N \to M$ of $e$ such that $\pi^{-1}(V') = \hat\pi^{-1}(V')\subset N$ and
$\pi|_{\hat\pi^{-1}(V')}=\hat\pi$.
\end{lem}
\begin{proof}
Since $M$ is a metric space and $M\setminus V$ is compact we may enlarge $V'$ such that $M\setminus V'$ is compact. Then we
find open sets ${V=V_0\supset V_1\supset V_2 \supset V_3 \supset V_4 =V' \supset V_5}$ such that $M\setminus V_i$ is compact and $\overline{V_{i+1}} \subset V_i$.
Next we choose a metric on $G$ such that $\ker\rd_{e(m)}\hat\pi \perp \rT_{e(m)}e(M)$.
By the compactness of $M\setminus V_5$ the exponential map then induces a diffeomorphism $\exp : D_\ep \overset{\sim}{\to}  N_\ep$ from a sufficiently small disk bundle in the normal bundle $D_\ep:=\{Z\in T_{e(m)} G \,|\, m\in {M\setminus V_5}, Z \perp \rT_{e(m)}e(M) , |Z|<\ep \}\subset \rT G|_{e(M\setminus V_5)}$ to a neighbourhood $N_\ep \subset G$ of $e(M\setminus V_5)$.
The projection to the zero section in $D_\ep\simeq N_\ep$ composed with $e^{-1}$ then provides a surjective submersion $\pi_0:N_\ep \to M\setminus V_5$ such that
\begin{equation}  \label{pis}
\pi_0|_{e(V\setminus V_5)}=\hat\pi|_{e(V\setminus V_5)}, \qquad
\rd\pi_0|_{e(V\setminus V_5)}=\rd\hat\pi|_{e(V\setminus V_5)} .
\end{equation}
In fact, these are equal to $e^{-1}$ and the orthogonal projection to $\rT e(M)$.
Next, for $U\subset M\setminus V_5$ we will write abbreviate $D_\ep|_{U}:=D_\ep \cap \rT G|_{e(U)}$.
With this notation we may choose $\ep>0$ sufficiently small such that
$$
\exp(D_\ep|_{\overline{V_1}\setminus V_5}) \subset \hat G  , \qquad
\exp(D_\ep|_{M\setminus V_2}) \subset G \setminus \hat\pi^{-1}(V').
$$
Indeed, the first inclusion can be achieved since $\hat G$ is a neighbourhood of the compact set $e(\overline{V_1}\setminus V_5)$. For the second inclusion we use the assumption 
$\overline{\hat\pi^{-1}(V')}\cap e(M\setminus V) =\emptyset$ and add that $\overline{\hat\pi^{-1}(V')}\cap e(V\setminus \overline{V_3}) =\emptyset$ since $\hat\pi^{-1}(V\setminus \overline{V_3})$ is an open neighbourhood of $e(V\setminus \overline{V_3})$ disjoint from $\hat\pi^{-1}(V')$.
So the compact set $e(M\setminus V_2) \subset e(M\setminus \overline{V_3})$ is disjoint from the closed set $\overline{\hat\pi^{-1}(V')}$, and hence $\inf_{p\in M\setminus V_2} d(e(p), \hat\pi^{-1}(V'))=:\delta >0$. Hence we can choose $\ep>0$ so that $\exp(D_\ep|_{M\setminus V_2})$ is disjoint from $\hat\pi^{-1}(V')$.

Now choose a smooth cutoff function $\psi:M\setminus\overline{V_5} \to [0,1]$ such that $\psi|_{M\setminus V_1}\equiv 0$ and $\psi|_{V_2\setminus\overline{V_5}}\equiv 1$. We need to extend the linear interpolation in a local chart to a smooth construction on $V_1 \setminus\overline{V_5}$. For that purpose we equip $M$ with a metric, and for $\delta>0$ smaller than the minimal injectivity radius on the compact set $\overline{V_1}\setminus V_5$ define
\begin{align*}
S_\psi : \bigl\{ (p, q) \in M\times M \,\big|\, p\in V_1\setminus \overline{V_5}, d(p,q)<\delta \bigr\} \;\to \; M , \qquad
(p,q) \;\mapsto\; \exp_p\bigl(\psi(p)\exp_p^{-1}(q)\bigr) .
\end{align*}
Then for sufficiently small $\ep>0$ we obtain an extended tubular neighbourhood
\[
\pi := \left.\begin{cases}
\; \hat \pi &\text{on}\; \hat\pi^{-1}(V') \\
S_\psi\circ(\pi_0\times\hat\pi) &\text{on}\; \exp(D_\ep|_{V_1\setminus \overline{V_5}}) \\
\pi_0 &\text{on}\; \exp(D_\ep|_{M\setminus V_1})
\end{cases} \right\} \; : \;\;  \hat\pi^{-1}(V')\cup \exp(D_\ep|_{M\setminus \overline{V_5}})=: N \;\to\; M  .
\]
Here by \eqref{pis} and the compactness of $\overline{V_1}\setminus V_5\subset V\setminus V_5$ we may choose $\ep>0$ so that
$(\pi_0\times\hat\pi)|_{\exp(D_\ep|_{V_1\setminus \overline{V_5}})}$ takes values in the domain of $S_\psi$.
To check that this map is well defined it remains to check the overlap of the different domains.
On $\hat\pi^{-1}(V') \cap \exp(D_\ep|_{M\setminus \overline{V_5}})\subset \exp(D_\ep|_{V_2\setminus \overline{V_5}})$ we have $\psi\circ\pi_0\equiv 1$ and hence $\pi=\hat\pi$. Moreover, $S_\psi\circ(\pi_0\times\hat\pi)$ extends smoothly to $\exp(D_\ep|_{M\setminus V_1})$ since
$\psi\circ\pi_0\equiv 0$ on $\exp(D_\ep|_{\overline{V_1}\setminus V_1})$.
Hence $\pi$ is a smooth map. It is defined on $\hat\pi^{-1}(V')\cup \exp(D_\ep|_{M\setminus \overline{V_5}})$, which is a neighbourhood of $e(V')\cup e(M\setminus \overline{V_5}) = e(M)$, and on the latter clearly restricts to $e^{-1}$.
Towards ensuring that $\pi$ is a submersion, note that for any $p\in \overline{V_1}\setminus V_5$ we have by \eqref{pis}
$$
\rd_{e(p)}\pi =
\rd_{(p,p)} S_\psi \circ \bigl( \rd_{e(p)}\pi_0 \times \rd_{e(p)}\hat\pi \bigr) = \rd_{e(p)} \pi_0 .
$$
Since $\overline{V_1}\setminus V_5$ is compact, we then find $\ep>0$ such that $\rd\pi |_{\exp(D_\ep|_{V_1\setminus \overline{V_5}}) } $ continues to be a submersion.
Finally, we obtain $\pi^{-1}(V') = \hat\pi^{-1}(V')$ and  $\pi|_{\hat\pi^{-1}(V')}=\hat\pi$ if $\pi(\exp(D_\ep|_{M\setminus V_2})\bigr) \subset M\setminus V'$. This holds for sufficiently small $\ep>0$ since
$\pi|_{e(M\setminus V_2)}=e^{-1}$ maps the compact set $e(M\setminus V_2)$ to $M\setminus V_2\sqsubset M\setminus V'$ and $\pi$ is uniformly continuous on compact sets.
\end{proof}

We apply this lemma to $\hat\pi:=\leftexp{(\pm)}{\hat\pi}^{\cP_-}_{\cP_+}$ and the evaluation embedding of $M:=\cM(\cP_-,\cP_+)$ into $G:=\leftexp{(\pm)}\Gr^{\cP_-}_{\cP_+}$.
Then $V:={\textstyle \bigcup_{\ul{q}'}} \cV_{t_{b-1}}(\ul{q}')_0 \subset \cM(\cP_-,\cP_+)$ has a compact complement since it covers $\partial\bM(\cP_-,\cP_+)$, so $\cM(\cP_-,\cP_+)\setminus {\textstyle \bigcup_{\ul{q}'}} \cV_{t}(\ul{q}')_0 = \bM(\cP_-,\cP_+)\setminus {\textstyle \bigcup_{\ul{q}'}} \cV_{t}(\ul{q}')$ is the complement of an open set in a compact space. For $\cP_-=X$ or $\cP_+=X$ we add $\ev_-^{-1}(\ti B^+_{p_+})$ resp.\ $\ev_+^{-1}(\ti B^-_{p_-})$ to the open set $V$.
We can then use $V':= {\textstyle \bigcup_{\ul{q}'}} \cV_{\frac 12 t_{b-1}}(\ul{q}')_0\subset\cM(\cP_-,\cP_+)$, and in case $\cP_-=X$ or $\cP_+=X$ add $\ev_\mp^{-1}(\ti B^\pm_{p_\pm}\cap \ti U_{\frac 12 t_{b-1}}(p_\pm))$. Its closure is contained in $V$ since the closure of $\cV_{\frac 12 t_{b-1}}(\ul{q}')_0\subset\cM(\cP_-,\cP_+)$ is contained in $\cV_{t_{b-1}}(\ul{q}')_0$ by Remark~\ref{rmk:Vt}, and also $\ti U_{\frac 12 t_{b-1}}(p_\pm) \sqsubset \ti U_1(p_\pm)$.
Hence Lemma~\ref{lem:ext} yields tubular neighbourhoods $\pi=:\leftexp{(\pm)}{\pi}^{\cP_-}_{\cP_+}: N \to \cM(\cP_-,\cP_+)$ defined on neighbourhoods $N\subset \leftexp{(\pm)}\Gr^{\cP_-}_{\cP_+}$ of $e(\cM(\cP_-,\cP_+))$ that contain
$\hat\pi^{-1}(V')=\ti U_{\frac 12 t_{b-1}}(p^\pm) \cup {\textstyle \bigcup_{\ul{q}'}} \leftexp{(\pm)}{\Gr}^{\cP_-}_{\cP_+}(\frac 12 t_{b-1},\ul{q}')$.
Now taking $0<t'_b \leq \frac 12 t_{b-1}$ sufficiently small we can ensure that $\leftexp{(\pm)}{\Gr}^{\cP_-}_{\cP_+}(t'_b)\subset N$ since for $t\to 0$ as in Remark~\ref{rmk:Grt converges} we have
$\leftexp{(\pm)}{\Gr}^{\cP_-}_{\cP_+}(t)\setminus \bigcup_{\ul{q}'} \leftexp{(\pm)}{\Gr}^{\cP_-}_{\cP_+}(\frac 12 t_{b-1},\ul{q}') \to e\bigl( \cM(\cP_-,\cP_+)\setminus \bigcup_{\ul{q}'} \cV_{\frac 12 t_{b-1}}(\ul{q}')_0\bigr)$, which is a compact subset of $N$.
Hence we obtain tubular neighbourhoods
\begin{align*}
\pi^{p_-}_{p_+} : \,\;\Gr^{p_-}_{p_+}(t'_b)  \to \cM(p_-,p_+)
&\qquad\text{of}\qquad
(\ev_{\ti S^-_{p_-}}\times \ev_{\ti S^+_{p_+}}):\cM(p_-,p_+)\hookrightarrow \Gr^{p_-}_{p_+}, \\
\leftexp{-}{\pi}^X_{p_+}  :  \leftexp{-}{\Gr}^X_{p_+}(t'_b) \to \cM(X,p_+)
&\qquad\text{of}\qquad
\;\,\quad\qquad\qquad\ev_-:\cM(X,p_+)\hookrightarrow \leftexp{-}{\Gr}^X_{p_+} , \\
\leftexp{+}{\pi}_X^{p_-}  :  \leftexp{+}{\Gr}_X^{p_-}(t'_b) \to \cM(p_-,X)
&\qquad\text{of}\qquad
\;\,\quad\qquad\qquad\ev_+:\cM(p_-,X)\hookrightarrow \leftexp{+}{\Gr}_X^{p_-}.
\end{align*}
for each $p_-,p_+\in{\rm Crit}(f)$ with $b(p_-,p_+)=b$ resp.\ $b(X,p_+)=b$ resp.\ $b(p_-,X)=b$.
By construction these satisfy conditions (iii) and (iv) of Lemma~\ref{lem:pi}, so it remains to choose $0<t_b\leq t'_b$ such that the intersection condition (i) is met. For that purpose, as specified in Section~\ref{sec:tub}, we define $\leftexp{\pm}{\pi}^{p_-}_{p_+}$ by pullback from $\pi^{p_-}_{p_+}$, and then obtain well defined tubular neighbourhoods $\pi(\ul{q}):\cN_{t'_b}(\ul{q})\to \cM(\ul{q})$ for all $b(\ul{q})\leq b$. Their fibers over $\cV_{t'_b}(\ul{Q})$ for $\ul{Q}\supset\ul{q}$ with $b(\ul{Q})<b$ by construction are identical to the fibers of $\hat\pi(\ul{q})$ as defined in Lemma~\ref{lem:pihat}. Hence the intersection condition $\im\iota_{\ul{\tau}} \pitchfork \pi(\ul{q})^{-1}({\ul{\g}}) = 1\text{pt}$ already holds for $\ul\tau\in I_{t'_b}(\ul{q})$ and $\ul{\g}\in\cM(\ul{q})\cap \bigcup_{\genfrac{}{}{0pt}{3}{\ul{Q}\supset\ul{q}}{b(\ul{Q})<b}}  \cV_{t'_b}(\ul{Q})$.
For the remaining fibers we may apply Lemma~\ref{lem:trans} to the submanifold $G:=\cN_{t'_b}(\ul{q})\subset \ti S(\ul{q})=: S$ and the tubular neighbourhood $\pi:=\pi(\ul{q})$.
As before, $K:=\cM(\ul{q})\setminus \bigcup_{\genfrac{}{}{0pt}{3}{\ul{Q}\supset\ul{q}}{b(\ul{Q})<b}}  \cV_{t'_b}(\ul{Q})$ is a compact subset of $M:=\cM(\ul{q})$.
The embeddings $\iota_{\ul\tau}$ to $S$ are well defined, with $\im\iota_{\ul\tau}\cap\Gr(\ul{q})\subset\cN_{t'_b}(\ul{q})$ for $\ul\tau\in[0,t'_b)^k$, so Lemma~\ref{lem:trans} provides $0<t''_b\leq t'_b$ and a neighbourhood $N\subset G$ of $e(M)$ such that the intersection property
$\im\iota_{\ul{\tau}} \pitchfork \bigl(\pi(\ul{q})^{-1}({\ul{\g}})\cap N \bigr) = 1\text{pt}$ holds for transition times in $[0,t''_b)^k$ and the fibers within $N$ over $\ul\g\in K$.
Finally taking $0<t_b \leq \frac 12 t''_b$ sufficiently small we can ensure that $\cN_{t_b}(\ul{q})\cap\pi^{-1}(K)\subset N$ since, as in Remark~\ref{rmk:Grt converges}, $\cN_t(\ul{q})\cap \pi^{-1}(K)= \cN_t(\ul{q})\setminus {\textstyle\bigcup_{\scriptscriptstyle \ul{Q}}}  \pi^{-1}\bigl(\cV_{t'_b}(\ul{Q}) \bigr)$ for $t\to 0$ converges in the Hausdorff distance to $e\bigl( \cM(\ul{q})\setminus {\textstyle \bigcup_{\scriptscriptstyle\ul{Q}}} \cV_{t_{b-1}}(\ul{Q})\bigr)$, which is a compact subset of $N$.
Now for $\ul\tau\in I'_{t_b}(\ul{q})\cap I_{t'_b}(\ul{q})\subset[0,t)^k$ and $\ul\g\in\cM(\ul{q})=(M\setminus K) \cup K$ we have the full intersection condition $\im\iota_{\ul{\tau}} \pitchfork \pi(\ul{q})^{-1}({\ul{\g}}) = 1\text{pt}$, establishing condition (i) of Lemma~\ref{lem:pi}. Finally, this lemma implies that the maps $\phi(\ul{q})$ constructed from the tubular neighbourhoods for $b(\ul{q})\leq b$ satisfy all properties of global charts claimed in Theorem~\ref{thm:global with ends}.

Moreover, we proved the fiberwise transversality \eqref{req2}, hence Remarks~\ref{rmk:smooth eval} and ~\ref{rmk:smooth eval 2} imply smoothness of the evaluation maps with respect to these charts.

\bibliographystyle{alpha}

\begin{thebibliography}{14}

\bibitem[AB]{ab} D.M.~Austin, P.M.~Braam,
Morse Bott theory and equivariant cohomology,
The Floer memorial volume, {\it Progress in Math}, 133 (1995), 123--183.

\bibitem[AFW]{arnold} P.~Albers, J.~Fish, K.~Wehrheim,
{\it A polyfold proof of the Arnold conjecture}, work in progress.

\bibitem[Bo]{bott} R.~Bott,
Morse Theory Indomitable,
{\it Publications Math�matiques de l'IH�S} 68 (1988), 99--114.

\bibitem[BH]{bh} D.~Burghelea, S.~Haller,
On the topology and analysis of a closed one form~I,
{\it Monogr.~Enseign.~Math.} 38 (2001), 133--175.

\bibitem[CJS]{cjs} R.L.~Cohen, J.~Jones, G.~Segal, {\it Morse theory and classifying spaces}, preprint, 1995

\bibitem[F]{franks} J.M.~Franks, Morse-Smale flows and homotopy theory, {\it Topology},
18 no.~3 (1979), 199--215.

\bibitem[H]{hutchings} M.~Hutchings,
{\it Lecture notes on Morse homology (with an eye towards Floer theory and pseudoholomorphic curves)},
{\sf http://math.berkeley.edu/$\sim$hutching/teach/276/mfp.ps}.

\bibitem[HWZ1]{hwz2} H.~Hofer, K.~Wysocki, E.~Zehnder,
A General Fredholm Theory II: Implicit Function Theorems,
{\it GAFA} 19 no.~1 (2009), 206--293.

\bibitem[HWZ2]{HWZgw} H.~Hofer, K.~Wysocki, E.~Zehnder,
{\it Applications of Polyfold Theory I: The Polyfolds of Gromov-Witten Theory},
arXiv:1107.2097.

\bibitem[KM]{km-sw} P.B.~Kronheimer, T.~Mrowka, {\it Monopoles and three-manifolds},
Cambridge University Press, 2007.

\bibitem[LW]{jiayong} J.~Li, K.~Wehrheim,
{\it $A_\infty$-algebras for Lagrangians via polyfold theory for Morse trees with holomorphic disks},
work in progress.

\bibitem[Mi]{milnor} J.~Milnor,
{\it Morse Theory}, Princeton University Press

\bibitem[Mo]{morse} M.~Morse, {\it The Calculus of Variations in the Large},
American Mathematical Society Colloquim Publication 18; New York, 1934.

\bibitem[N]{nielsen}
L.T.~Nielsen, Transversality and the Inverse Image of a Submanifold with Corners,
{\it Mathematica Scandinavica }49 (1981), 211--221.

\bibitem[NP]{np} S.~Newhouse, M.~Peixoto, There is a simple arc joining any two Morse-Smale flows,
Ast\'erisque 31 (1976), 16--41.

\bibitem[P]{palis}
J.~Palis, On Morse-Smale dynamical systems, {\it Topology} 8 (1968), 385--404.

\bibitem[PS]{palis-smale}
J.~Palis, S.~Smale, Structural stability theorems, {\it Proc.\ Symp.\ Pure Math.}14 (1970), 223--232.

\bibitem[PSS]{PSS}
S.~Piunikhin, D.A.~Salamon, and M.~Schwarz, Symplectic {F}loer-{D}onaldson theory and quantum cohomology, in {\em Contact and symplectic geometry}, {\em Publ.\ Newton Inst.}  8 (1996), 171--200.

\bibitem[Q]{qin} L.~Qin,
{\it On the Associativity of Gluing}, arXiv:1107.5527.

\bibitem[Sc]{schwarz} M.~Schwarz,
{\it Morse Homology}, Birkh\"auser, 1993.

\bibitem[Sh]{shub} M.~Shub, {\it Global Stability of Dynamical Systems}, Springer, 1987.

\bibitem[Sm]{smale} S.~Smale, On Gradient Dynamical Systems,
{\it Annals of Mathematics} 74 no.~1 (1961), 199--206.

\bibitem[W]{witten} E.~Witten, Supersymmetry and Morse theory, {\it J.~Differential Geom.} 17 (1982), no.~4, 661--692.



 \end{thebibliography}

\end{document}